\documentclass[a4paper]{article}
\usepackage{amsthm}
\usepackage{amsmath}
\usepackage{amssymb}
\usepackage{graphicx}
\newcommand{\RR}{{\cal R}}
\newcommand{\R}{\mathrm{R}}
\newcommand{\QQ}{{\cal Q}}

\newcommand{\el}{\mbox{{\rm el}}}
\newcommand{\cyl}{\mbox{{\rm cyl}}}
\newcommand{\surf}{\mbox{{\rm surf}}}
\newcommand{\gen}{\mbox{{\rm gen}}}
\newcommand{\claim}[2]{\begin{equation}\mbox{\parbox{\linewidth}{{\em #2}}}\label{#1}\end{equation}}
\newtheorem{theorem}{Theorem}[section]
\newtheorem{corollary}[theorem]{Corollary}
\newtheorem{lemma}[theorem]{Lemma}
\newcount\claimno
\def\newclaim#1#2{
   \global\advance\claimno by 1\relax
   \bigskip\noindent\rlap{\rm(\the\claimno)}\ignorespaces
   \global\expandafter\edef\csname CLAIMLABEL#1\endcsname{\the\claimno}\relax
   \hangindent=33pt\hskip30pt{\sl#2}\bigskip}
\def\refclaim#1{\csname CLAIMLABEL#1\endcsname}
\def\mylabel#1{{\label{#1}}}
\def\junk#1{}

\newenvironment{subproof}{%
  \begin{proof}[Subproof]%
}{%
  \end{proof}%
}

\begin{document}
\title{Three-coloring triangle-free graphs on surfaces III.  Graphs of girth five}
\author{%
     Zden\v{e}k Dvo\v{r}\'ak\thanks{Computer Science Institute of Charles University,
           Malostransk{\'e} n{\'a}m{\v e}st{\'\i} 25, 118 00 Prague, 
           Czech Republic. E-mail: {\tt rakdver@iuuk.mff.cuni.cz}. 
           Supported by the Center of Excellence -- Inst. for Theor. Comp. Sci., Prague, project P202/12/G061 of Czech Science Foundation.}
 \and
     Daniel Kr{\'a}l'\thanks{Warwick Mathematics Institute, DIMAP and Department of Computer Science, University of Warwick, Coventry CV4 7AL, United Kingdom. E-mail: {\tt D.Kral@warwick.ac.uk}.}
 \and
        Robin Thomas\thanks{School of Mathematics, 
        Georgia Institute of Technology, Atlanta, GA 30332. 
        E-mail: {\tt thomas@math.gatech.edu}.
        Partially supported by NSF Grants No.~DMS-0739366 and DMS-1202640.}
}
\date{July 14, 2018}
\maketitle
\begin{abstract}
We show that the size of a $4$-critical graph of girth at least five is bounded by a linear function of its genus.
This strengthens the previous bound on the size of such graphs given by Thomassen.
It also serves as the basic case for the description of the structure of $4$-critical triangle-free
graphs embedded in a fixed surface, presented in a future paper of this series.
\end{abstract}

\section{Introduction}

This paper is a part of a series aimed at studying the $3$-colorability
of graphs on a fixed surface that are either triangle-free, or have their
triangles restricted in some way. Historically the first result in this direction is the following
classical theorem of Gr\"otzsch~\cite{grotzsch1959}.

\begin{theorem}
\label{grotzsch}
Every triangle-free planar graph is $3$-colorable.
\end{theorem}

Thomassen~\cite{thom-torus,thomassen1995-34,ThoShortlist}
found three reasonably simple proofs of this claim.
Recently, two of us, in joint work with Kawarabayashi~\cite{DvoKawTho}
were able to design a linear-time algorithm to $3$-color triangle-free
planar graphs, and as a by-product found perhaps a yet simpler proof
of Theorem~\ref{grotzsch}.
The statement of Theorem~\ref{grotzsch}
cannot be directly extended to any surface other than the sphere.
In fact, for every non-planar surface $\Sigma$ there are infinitely many
$4$-critical graphs that can be embedded in $\Sigma$.
For instance, the graphs obtained from an odd cycle of length five or more
by applying Mycielski's
construction \cite[Section~8.5]{BonMur} have that property.
Thus an algorithm for testing $3$-colorability of triangle-free graphs
on a fixed surface will have to involve more than just testing the
presence of finitely many obstructions.

The situation is different for graphs of girth at least five
by another deep theorem of Thomassen~\cite{thomassen-surf}, the following.

\begin{theorem}
\mylabel{thm:thomgirth5}
For every surface $\Sigma$ there are only finitely many $4$-critical
graphs of girth at least five that can be embedded in $\Sigma$.
\end{theorem}

Thus the $3$-colorability problem on a fixed surface
has a linear-time algorithm
for graphs of girth at least five, but the presence of cycles of
length four complicates matters.
Let us remark that there are no $4$-critical graphs of girth at least five
on the projective plane and the torus~\cite{thom-torus} and
on the Klein bottle~\cite{tw-klein}.

In his proof of Theorem~\ref{thm:thomgirth5}, Thomassen does not give a specific bound on
the size of a $4$-critical graph of girth at least five embedded in $\Sigma$.  It appears that if one was to extract
a bound from the argument, that bound would be at least
doubly-exponential in the genus of $\Sigma$.  In this paper, we give a different
proof of the result, which gives a linear bound.

\begin{theorem}
\mylabel{thm:mainsurf}
There exists a constant $C$ with the following property.
If $G$ is a $4$-critical graph of Euler genus $g$ and girth at least $5$, then $|V(G)|\le Cg$.
\end{theorem}

Let us now outline the relationship of this result to the structure of
triangle-free $4$-critical graphs.
The only non-planar surface for which the $3$-colorability problem
for triangle-free graphs is fully characterized is the projective plane.
Building on earlier work of Youngs~\cite{Youngs}, Gimbel and
Thomassen~\cite{gimbel} obtained the following elegant characterization.
A graph embedded in a surface is a {\em quadrangulation} if every face
is bounded by a cycle of length four.

\begin{theorem}
\mylabel{thm:gimtho}
A triangle-free graph embedded in the projective plane is $3$-colorable if and only
if it has no subgraph isomorphic to a non-bipartite
quadrangulation of the projective plane.
\end{theorem}

For other surfaces there does not seem to be a similarly nice characterization,
but in a later paper of this series we will present a polynomial-time
algorithm to decide whether a triangle-free graph in a fixed surface
is $3$-colorable.
The algorithm naturally breaks into two steps.
The first is when the graph is
a quadrangulation, except perhaps for a bounded number of larger faces
of bounded size, which will be allowed to be precolored.
In this case there is a simple topological obstruction to the existence
of a coloring extension based on the so-called ``winding number" of
the precoloring.
Conversely, if the obstruction is not present and the graph is highly
``locally planar", then we can show that the precoloring can be
extended to a $3$-coloring of the entire graph.
This can be exploited to design a polynomial-time algorithm.
With additional effort the algorithm can be made to run in linear time.

The second step covers the remaining case, when the graph has either many faces
of size at least five, or one large face, and the same holds for every
subgraph.
In that case, we reduce the problem to Theorem~\ref{thm:mainsurf} and
show that the graph is $3$-colorable.  More precisely, in a future paper of
this series, we use Theorem~\ref{thm:mainsurf} to derive the following
cornerstone result.

\begin{theorem}
\mylabel{thm:corner}
There exists an absolute constant $K$ with the following property.
Let $G$ be a graph embedded in a surface $\Sigma$ of Euler genus $\gamma$
so that every $4$-cycle bounds a $2$-cell face,
and let $t$ be the number triangles in $G$.
If $G$ is $4$-critical, 
then $\sum|f|\le K(t+\gamma-1)$,
where the summation is over all faces $f$ of $G$ of length at least five.
\end{theorem}

The fact that the bound in Theorems~\ref{thm:mainsurf} and \ref{thm:corner} is linear
is needed in our solution~\cite{dkt} of a problem of Havel~\cite{conj-havel}, as follows.

\begin{theorem}
\mylabel{havel}
There exists an absolute constant $d$ such that if $G$ is a planar
graph and every two distinct triangles in $G$ are at distance at least $d$,
then $G$ is $3$-colorable.
\end{theorem}

Our technique to prove Theorem~\ref{thm:mainsurf} is a refinement of the standard method of reducible configurations.
We show that every sufficiently generic graph $G$ (i.e., a graph that is large enough
and cannot be decomposed to smaller pieces along cuts simplifying the problem)
embedded in a surface contains one of a fixed list of configurations.   Each such configuration
enables us to obtain a smaller $4$-critical graph $G'$ with the property that every $3$-coloring of $G'$
corresponds to a $3$-coloring of $G$.  Furthermore, we perform the reduction in such a way
that a properly defined weight of $G'$ is greater or equal to the weight of $G$.
A standard inductive argument then shows that the weight of every $4$-critical graph is bounded,
which also restricts its size.  This brief exposition however hides a large number of technical
details that were mostly dealt with in the previous paper in the series~\cite{trfree2}.
There, we introduced this basic technique and used it to prove the following special
case of Theorem~\ref{thm:corner}.

\begin{theorem}
\mylabel{thm:maindisk}
Let $G$ be a graph of girth at least $5$ embedded in the plane and let $C$ be a cycle in $G$.
Suppose that there exists a precoloring $\phi$ of $C$ by three colors that does
not extend to a proper $3$-coloring of $G$.  Then there exists a subgraph $H\subseteq G$
such that $C\subseteq H$, $|V(H)|\le 1715|C|$ and $H$ has no proper $3$-coloring extending $\phi$.
\end{theorem}

Further results of \cite{trfree2} needed in this paper are summarized in Section~\ref{sec-summary}.

\section{Definitions}

In this section, we give a few basic definitions.
All graphs in this paper are finite and simple, with no loops or parallel edges.

A \emph{surface}
is a compact connected $2$-manifold with (possibly null) boundary.  Each component of the boundary
is homeomorphic to the circle, and we call it a \emph{cuff}.  For non-negative integers $a$, $b$ and $c$,
let $\Sigma(a,b,c)$ denote the surface obtained from the sphere by adding $a$ handles, $b$ crosscaps and
removing interiors of $c$ pairwise disjoint closed discs.  A standard result in topology shows that
every surface is homeomorphic to $\Sigma(a,b,c)$ for some choice of $a$, $b$ and $c$.
Note that $\Sigma(0,0,0)$ is a sphere, $\Sigma(0,0,1)$ is a closed disk, $\Sigma(0,0,2)$ is a cylinder,
$\Sigma(1,0,0)$ is a torus, $\Sigma(0,1,0)$ is a projective plane and $\Sigma(0,2,0)$ is a Klein bottle.
The \emph{Euler genus} $g(\Sigma)$ of the surface $\Sigma=\Sigma(a,b,c)$ is defined as $2a+b$.
For a cuff $C$ of $\Sigma$, let $\widehat{C}$ denote an open disk with boundary $C$ such that $\widehat{C}$ is disjoint from $\Sigma$, and let $\Sigma+\widehat{C}$ be
the surface obtained by gluing $\Sigma$ and $\widehat{C}$ together, that is, by closing $C$ with a patch.
Let $\widehat{\Sigma}=\Sigma+\widehat{C_1}+\ldots+\widehat{C_c}$, where $C_1$, \ldots, $C_c$ are the cuffs of $\Sigma$,
be the surface without boundary obtained by patching all the cuffs.

Consider a graph $G$ embedded in the surface $\Sigma$; when useful, we identify $G$ with the topological
space consisting of the points corresponding to the vertices of $G$ and the simple curves corresponding
to the edges of $G$.  We say that the embedding is \emph{normal} if every cuff of $\Sigma$ is equal to a cycle in $G$,
and we call such a cycle a \emph{ring}.
Throughout the paper, all graphs are embedded normally.
A \emph{face} $f$ of $G$ is a maximal arcwise-connected subset of $\Sigma-G$.
We write $F(G)$ for the set of faces of $G$.
The boundary of a face is equal to a union of closed walks of $G$, which we call the \emph{boundary walks} of $f$.  

Consider a ring $R$.  If $R$ is a triangle and at most one vertex of $R$ has degree greater than two in $G$, we say that $R$
is a \emph{vertex-like ring}.  A ring with only vertices of degree two is \emph{isolated}.  For a vertex-like ring $R$ that is
not isolated, the \emph{main} vertex of $R$ is its vertex of degree greater than two.
A vertex $v$ of $G$ is a \emph{ring vertex} if $v$ is contained in a ring (i.e., $v$ is drawn in the boundary of $\Sigma$), 
and $v$ is \emph{internal} otherwise.  
A cycle $K$ in $G$ is \emph{separating} or \emph{separates the surface}
if $\widehat{\Sigma}-K$ has at least two components, 
and $K$ is \emph{non-separating} otherwise.
A cycle $K$ is \emph{contractible} if there exists a closed disk $\Delta\subseteq \Sigma$ with boundary equal to $K$.
A cycle $K$ \emph{surrounds the cuff $C$} if $K$ is not contractible in $\Sigma$, but it is contractible in $\Sigma+\widehat{C}$.
We say that $K$ \emph{surrounds a ring $R$} if $K$ surrounds the cuff incident with $R$.

Let $G$ be a graph embedded in a surface $\Sigma$, let the embedding be
normal, and let $\cal R$ be the set of rings of this embedding.
In those circumstances we say that $G$ is a \emph{graph in $\Sigma$
with rings $\cal R$.}
Furthermore, some vertex-like rings are designated as \emph{weak vertex-like rings}.  

For a vertex-like ring $R$, we define the \emph{length} of $R$ as $|R|=0$ if $R$ is weak and $|R|=1$ otherwise.
For a ring $R$ that is not vertex-like, the \emph{length} $|R|$ of $R$ is the number of vertices of $R$.
For a face $f$, by $|f|$ we mean the sum of the lengths of the boundary walks of $f$ (in particular, if an edge
appears twice in the boundary walks, it contributes $2$ to $|f|$).
For a set of rings $\RR$, let us define $\ell({\RR})=\sum_{R\in\RR} |R|$.

Let $G$ be a graph with rings $\cal R$.  Let $H=\bigcup {\cal R}$ and let $H'$ be a (not necessarily induced) subgraph of $G$ obtained from $H$
by, for each weak vertex-like ring $R$, removing the main vertex and one of the non-main vertices of $R$
(or by removing two vertices of $R$ if $R$ has no main vertex),
so that $H'$ intersects $R$ in exactly one non-main vertex.  A {\em precoloring} $\psi$ of $\cal R$ is a 
$3$-coloring of the graph $H'$.  
A precoloring of $\cal R$ {\em extends to a $3$-coloring of $G$}
if there exists a $3$-coloring $\phi$ of $G$ such that $\phi(v)=\psi(v)$ for every $v\in V(H')$.
The graph $G$ is {\em $\cal R$-critical} if $G\neq H$ and for every proper subgraph
$G'$ of $G$ that contains $H$, there exists a precoloring of ${\cal R}$ that extends
to a $3$-coloring of $G'$, but not to a $3$-coloring of $G$.  For a precoloring $\kappa$ of $\cal R$
the graph $G$ is {\em $\kappa$-critical} if $\kappa$ does not extend to a $3$-coloring of $G$,
but it extends to a $3$-coloring of every proper subgraph of $G$ that contains $\cal R$.

Let us remark that if $G$ is $\kappa$-critical for some $\kappa$, then it is $\cal R$-critical,
but the converse is not true (for example, consider a graph consisting of a single ring with two chords).
On the other hand, if $\kappa$ is a precoloring of the rings of $G$ that does not extend to a $3$-coloring of $G$, then
$G$ contains a (not necessarily unique) $\kappa$-critical subgraph.

Weak vertex-like rings are just a technical device that we need at one point in the proof.
Fortunately, we can usually get by without any special considerations of weak rings,
due to the following observation.

\begin{lemma}\label{lemma-unweak}
Let $G$ be a graph embedded in a surface with rings $\RR$.  Let $\RR'$ be the same set of rings as $\RR$,
except that no vertex-like ring of $\RR'$ is designated to be weak.  If $G$ is $\RR$-critical, then $G$ also is $\RR'$-critical.
\end{lemma}
\begin{proof}
Let $G'$ be a proper subgraph of $G$ that contains $\RR'$.  Since $G$ is $\RR$-critical, there exists a precoloring $\psi$
of $\RR$ that extends to a $3$-coloring $\phi$ of $G'$, but does not extend to $G$.  Let $\psi'$ be the restriction of $\phi$ to $\bigcup\RR'$.
Then $\psi'$ gives a precoloring of $\RR'$ that extends to a $3$-coloring of $G'$ (namely, $\phi$), but does not extend to $G$.
\end{proof}

Let $G$ be a graph embedded in a disk with one ring $R$ of length $l\ge5$.
We say that $G$ is {\em exceptional} if it satisfies one of the conditions below (see Figure~\ref{fig-except}):

\begin{figure}
\begin{center}\includegraphics{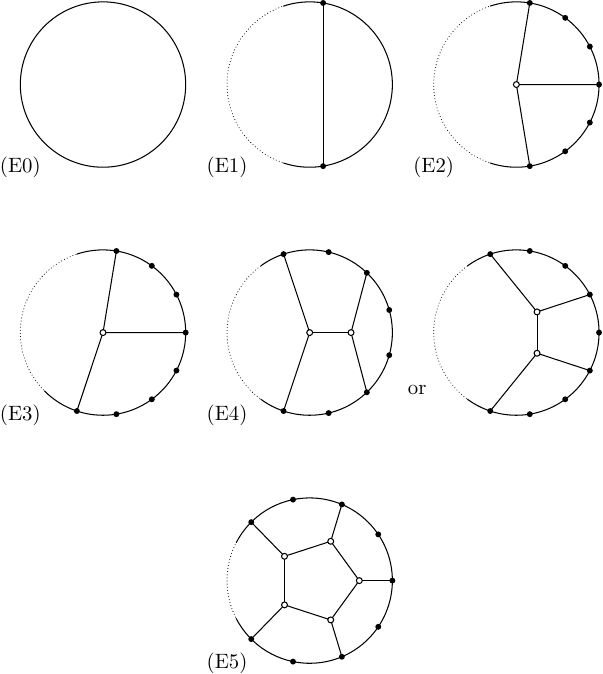}\end{center}
\caption{Exceptional graphs.}
\label{fig-except}
\end{figure}

\begin{itemize}
\item[(E0)] $G=R$,
\item[(E1)] $l\ge 8$ and $E(G)-E(R)=1$,
\item[(E2)] $l\ge 9$, $V(G)-V(R)$ has exactly one vertex of degree three,
and the faces of $G$ have lengths $5,5,l-4$,
\item[(E3)] $l\ge 11$, $V(G)-V(R)$ has exactly one vertex of degree three,
and the faces of $G$ have lengths $5,6,l-5$,
\item[(E4)]  $l\ge 10$, $V(G)-V(R)$ consists of two adjacent degree three vertices,
and the faces of $G$ have lengths $5,5,5,l-5$,
\item[(E5)] $l\ge 10$, $V(G)-V(R)$ consists of five degree three vertices forming
a facial cycle of length five,
and the faces of $G$ have lengths $5,5,5,5,5,l-5$.
\end{itemize}
\noindent We say that $G$ is {\em very exceptional} if it satisfies
(E0), (E1), (E2) or (E3).

\section{Definitions and results from \cite{trfree2}}
\label{sec-summary}

Let $G$ be a graph in a surface $\Sigma$ with rings $\RR$.
A face is {\em open $2$-cell} if it is homeomorphic to an open disk.
A face is {\em closed $2$-cell} if it is open $2$-cell and
bounded by a cycle. A face $f$ is {\em semi-closed $2$-cell} if it is open $2$-cell,
and if a vertex $v$ appears more than once in the boundary walk of $f$, then it appears exactly twice,
$v$ is the main vertex of a vertex-like ring $R$ and the edges of $R$ form part of the boundary walk of $f$.
A face $f$ is {\em omnipresent} if it is not open $2$-cell and
each of its boundary walks is a cycle bounding
a closed disk $\Delta\subseteq \widehat{\Sigma}\setminus f$ containing exactly one ring.
We say that $G$ has an {\em internal $2$-cut} if there exist
sets $A,B\subseteq V(G)$ such that $A\cup B=V(G)$, $|A\cap B|=2$,
$A\setminus B\ne\emptyset\ne B\setminus A$, $A$ includes all vertices of $\RR$,
and no edge of $G$ has one end in $A\setminus B$ and the other in $B\setminus A$.

We wish to consider the following conditions that the triple
 $G,\Sigma,{\RR}$ may or may not satisfy:
\begin{itemize}
\item[(I0)] every internal vertex of $G$ has degree at least three,
\item[(I1)] $G$ has no even cycle consisting of internal vertices of
degree three,
\item[(I2)] $G$ has no cycle $C$ consisting of internal vertices of degree
three, together with two distinct adjacent vertices $u,v\in V(G)-V(C)$ such that both
$u$ and $v$ have a neighbor in $C$,
\item[(I3)] every face of $G$ is semi-closed $2$-cell and has length at least $5$,
\item[(I4)] if a path of length at most two has both ends in $\bigcup\RR$,
then it is a subgraph of $\bigcup\RR$,
\item[(I5)] no two vertices of degree two in $G$ are adjacent, unless they belong to a vertex-like ring,
\item[(I6)] if $\Sigma$ is the sphere and $|{\RR}|=1$, or if
$G$ has an omnipresent face, then $G$ does not contain an internal $2$-cut,
\item[(I7)] the distance between every two distinct members of $\RR$
is at least four,
\item[(I8)] every cycle in $G$ that does not separate the surface has length at least seven,
\item[(I9)] if a cycle $C$ of length at most $9$ in $G$ bounds an open disk $\Delta$ in $\Sigma$,
then $\Delta$ is a face, a union of a $5$-face and a $(|C|-5)$-face, or $C$ is a $9$-cycle and $\Delta$ consists of three $5$-faces
intersecting in a vertex of degree three.
\end{itemize}

Some of these properties are automatically satisfied by critical graphs; see \cite{trfree2} for the proofs
of the following observations.

\begin{lemma}
\mylabel{lem:i012}
Let $G$ be a graph in a surface $\Sigma$ with rings $\RR$.
If $G$ is $\RR$-critical, then it satisfies {\rm (I0), (I1) and (I2)}.
\end{lemma}

\begin{lemma}
\mylabel{lem:crit3conn}
Let $G$ be a graph in a surface $\Sigma$ with rings $\RR$.  Suppose that each component of $G$ is a planar graph
containing exactly one of the rings.  If $G$ is ${\RR}$-critical and contains no non-ring triangle,
then each component of $G$ is $2$-connected and $G$ satisfies (I6).
\end{lemma}

Let $G$ be a graph in a surface $\Sigma$ with rings $\RR$,
and let $P$ be a path of length
at least one and at most four with ends
$u,v\in V(\RR)$ and otherwise disjoint from $\RR$.
We say that $P$ is {\em allowable} if
\begin{itemize}
\item $u,v$ belong to the same ring of $\RR$, say $R$,
\item $P$ has length at least three,
\item there exists a subpath $Q$ of $R$ with ends $u,v$
such that $P\cup Q$ is a cycle of length at most eight
that bounds an open disk $\Delta\subset \Sigma$,
\item if $P$ has length three, then $P\cup Q$ has length five
and $\Delta$ is a face of $G$, and
\item if $P$ has length four, then $\Delta$ includes at
most one edge of $G$, and if it includes one, then that edge joins
the middle vertex of $P$ to the middle vertex of the path $Q$, which also has length four.
\end{itemize}

\noindent
We say that $G$ is {\em well-behaved} if every path $P$ of length 
at least one and at most four with ends 
$u,v\in V(\RR)$ and otherwise disjoint from
$\RR$ is allowable.

Let $M$ be a subgraph of $G$.
A subgraph $M\subseteq G$ \emph{captures $(\le\!4)$-cycles} if $M$ contains all
cycles of $G$ of length at most $4$ and furthermore, $M$ is either null or has minimum
degree at least two.

Throughout the rest of the paper, let $\epsilon=2/4113$ and
let $s:\{5,6,\ldots\}\to{\mathbb R}$ be the function defined by
$s(5)= 4/4113$, 
$s(6)=72/4113$, 
$s(7)=540/4113$,
$s(8)=2184/4113$ and $s(l)=l-8$ for $l\ge9$.
Based on this function, we assign weights to the faces.
Let $G$ be a graph embedded in $\Sigma$ with rings $\RR$ such that every open $2$-cell face of
$G$ has length at least $5$.
For a face $f$ of $G$, we define $w(f)=s(|f|)$ if $f$ is open $2$-cell
and $w(f)=|f|$ otherwise.  We define $w(G,{\cal R})$ as the sum of $w(f)$ over all faces
$f$ of $G$.

\bigskip

Before we proceed further, let us give an intuition behind the following definitions and especially
behind the key Theorem~\ref{thm:summary} that we are about to state.  Let $G$ be a graph of girth at least $5$ in a surface $\Sigma$ with rings $\RR$.
We aim to prove that if $G$ is $\RR$-critical, then its size is bounded by a linear function of the genus of $\Sigma$ and the number and
the lengths of the rings (if $G$ has no rings, then $G$ is $4$-critical, and we obtain Theorem~\ref{thm:mainsurf}; but our proof
method needs the stronger statement to deal with issues relating to possible short non-contractible cycles in $G$).
More precisely, we will show that $w(G,\RR)$ is bounded.

\begin{figure}
\includegraphics[width=12cm]{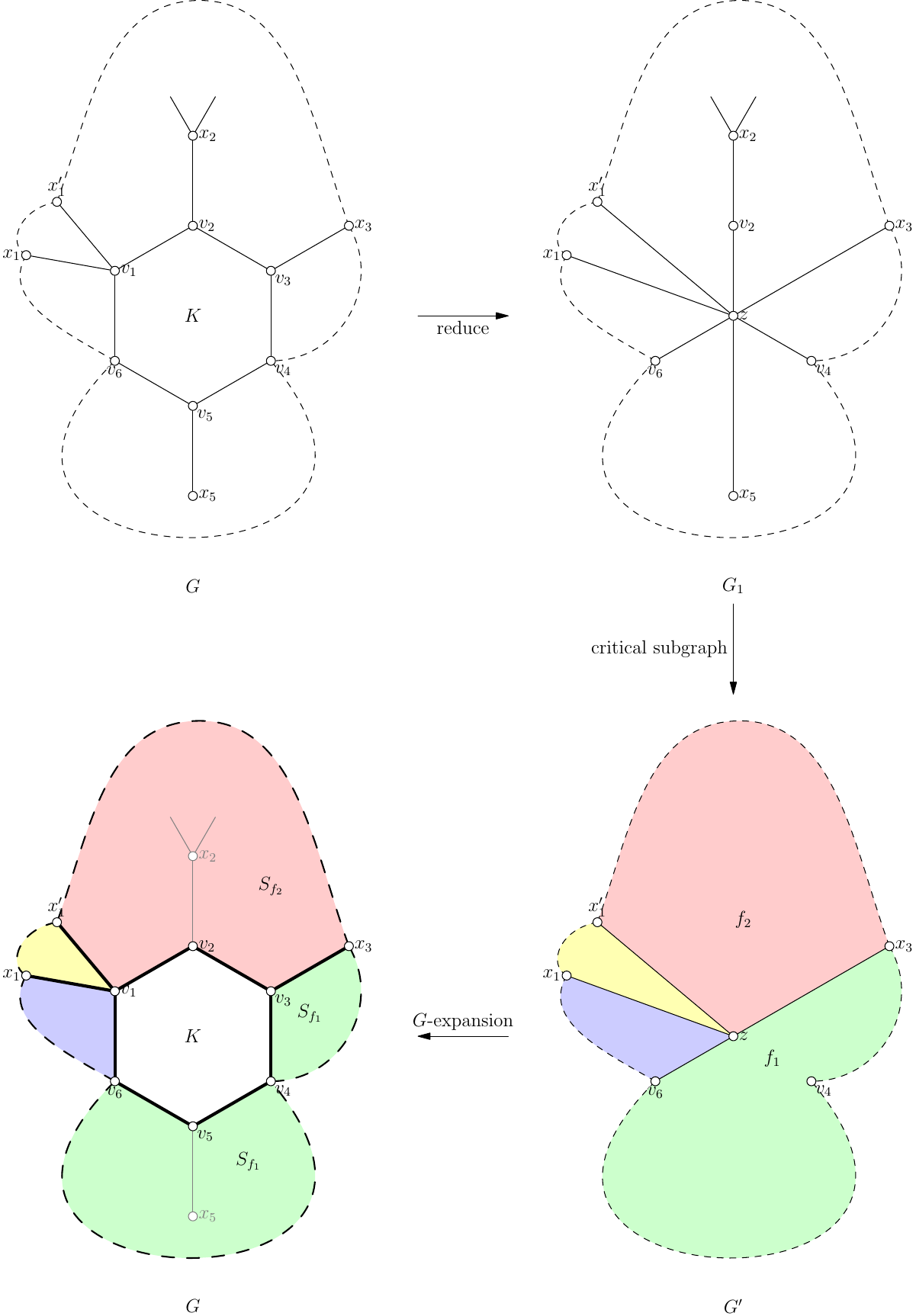}
\caption{The main idea.}\label{fig-mainstep}
\end{figure}

The proof is by induction on the complexity of the surface $\Sigma$ and the size of $G$ (formalized by the definition of the ordering $\prec$ in Section~\ref{sec-surfaces}).
Let us illustrate the main idea using a simplified example,
see Figure~\ref{fig-mainstep} for reference.  In~\cite{trfree2}, we identified a number of reducible
configurations that appear in any sufficiently large graph of girth at least $5$ embedded in a fixed surface; one of them is a face bounded by a
$6$-cycle $K=v_1v_2\ldots v_6$ such that $v_2$ and its neighbor $x_2$ outside of $K$ have degree three (and some more technical assumptions hold).  Given such a reducible
configuration, we can reduce $G$ to obtain a graph $G_1$ of girth at least $5$ embedded in $\Sigma$ with the same rings $\RR$; in our example, this is
achieved by identifying vertices $v_1$, $v_3$, and $v_5$ to a single vertex $z$.  Note that every $3$-coloring of the reduced graph $G_1$ extends to a $3$-coloring of $G$,
but $G_1$ is not necessarily $\RR$-critical (e.g., in our example, the vertex $v_2$ has degree two in $G_1$, and thus it is irrelevant
for $3$-colorability).  To be able to apply induction, we consider an $\RR$-critical subgraph $G'$ of $G_1$.  By the induction hypothesis, we obtain
a bound on $w(G',\RR)$, and thus we only need to show that $w(G,\RR)\le w(G',\RR)$.

To each face $f$ of $G'$, we can assign a set of faces of $G$ in the following natural way: Let $J'_f$ be the boundary of $f$.  Let $J_f$ be the subgraph of
$G$ obtained from $J'_f$ by undoing the reduction (in our case, if $J'_f$ contains $z$, we replace it either by one of the vertices $v_1$, $v_3$, $v_5$,
or by one of the paths $v_1v_2v_3$, $v_3v_4v_5$, $v_5v_6v_1$ as appropriate depending on the edges of $J'_f$ incident with $z$).  Now, $J_f$ has one or more
faces corresponding to $f$ (it may have more than one, see e.g. the face $f_1$ in Figure~\ref{fig-mainstep}), let the set of these faces be denoted
by $S_f$.  For each face $g$ of $G$ except for the $6$-face bounded by $K$, there exists a unique face $f$ of $G'$ such that $g$ is a subset of a face of $S_f$
(in the terms we are going to define below, $\{(J_f,S_f):f\in F(G')\}$ is a cover of $G$ by faces of $G'$).  Let $G[S_f]$ denote the subgraph of $G$ drawn in the closure
of the union of the faces of $S_f$.  To prove $w(G,\RR)\le w(G',\RR)$, we need to argue that for each $f\in F(G')$, the total weight of the faces of $G[S_f]$ is at most $w(f)$
(and in fact, that some of the inequalities are far from being tight, thus paying for the weight of the face bounded by $K$ that is not accounted for otherwise---this
difference is lower-bounded by the \emph{contribution} of a face of $G'$ as defined below).

To do so, we again apply induction: it is easy to see that since $G$ is $\RR$-critical, its subgraph $G[S_f]$ is critical with respect to $J_f$.
There is a caveat: we may not be able to directly embed $G[S_f]$ in a surface (or surfaces, if $|S_f|>1$) with rings corresponding to $J_f$; e.g., suppose that $\Sigma$ is
a torus, and $J_f$ is the union of two homotopic non-contractible cycles intersecting in one vertex $w$, and $S_f$ consists of the face $h$ of $J_f$ homeomorphic to the open cylinder.
Then the two boundary walks of $h$ intersect, but distinct rings of a graph must be vertex-disjoint.  However, we can easily overcome this difficulty by splitting the vertex
$w$ into two vertices, so that $J_f$ becomes a disjoint union of cycles and $G[S_f]$ can be naturally embedded in a cylinder with two rings (which is the surface
we below denote by $\Sigma_h$).  This splitting operation is formalized below as ``$G$-expansion of $S_f$''.

Anyway, let us ignore this subtlety for the moment. The bound on $w(G',\RR)$ is formulated in such a way that it ensures that $G[S_f]$ is embedded in at most as complex
surface as $\Sigma$, and thus we can apply induction to it.  If the sum of the lengths of the rings of $G[S_f]$ is the same as $|f|$, the induction directly
shows that the sum of weighs of faces of $G[S_f]$ is at most $w(f)$ (and actually smaller unless $S_f=\{f\}$).  However, this does not need to be the case for the
faces touching the reduced configuration; in our example, the face in $S_{f_2}$ has length $|f_2|+2$ (in the terms defined below, $f_2$ has elasticity $2$).
A more detailed accounting beyond the scope of this brief description is needed for such faces.

In~\cite{trfree2}, we did all the work establishing the existence of the reducible configurations and analyzing their reductions, and thus
in this paper we do not need to deal with the particulars.  Instead, we only use Theorem~\ref{thm:summary} below which we proved in~\cite{trfree2},
showing the existence of the critical subgraph $G'$ of the reduced graph and of the cover of $G$ by its faces. It also establishes the bounds on the contributions
and elasticities of the faces that we need to finish the argument.

\bigskip

Let us now proceed with the formal definitions.
Let $\Pi$ be a surface with boundary and $c$ a simple curve intersecting the boundary of $\Pi$ exactly in its ends.
The topological space obtained from $\Pi$ by cutting along $c$ (i.e., duplicating every point of $c$ and turning
both copies into boundary points) is a union of at most two surfaces.  If $\Pi_1,\ldots, \Pi_k$ are obtained from $\Pi$ by repeating this construction,
we say that they are \emph{fragments} of $\Pi$.

Consider a graph $H$ embedded in a surface $\Pi$ with rings $\QQ$, and let $f$ be a face of $H$.
There exists a unique surface whose interior is homeomorphic to $f$, which we denote by $\Pi_f$.
Note that the cuffs of $\Pi_f$ correspond to the facial walks of $f$.

Let $G$ be a graph embedded in $\Sigma$ with rings $\RR$.
Let $J$ be a subgraph of $G$ and let $S$ be a subset of faces of $J\cup\bigcup\RR$ such that
$J$ is equal to the union of the boundaries of the faces in $S$.
We define $G[S]$ to be the subgraph of $G$ consisting of $J$
and all the vertices and edges drawn inside the faces of $S$.
Let $C_1,C_2,\ldots,C_k$ be the boundary walks of the faces in $S$.
We would like to view $G[S]$ as a graph with rings $C_1$, \ldots, $C_k$.
However, the $C_i$'s do not necessarily have to be disjoint, and they do not have to be cycles.
To overcome this difficulty, we proceed as follows:
Suppose that $S=\{f_1,\ldots, f_m\}$.  For $1\le i\le m$, let $\Sigma_i$ be
a surface with boundary $B_i$ such that $\Sigma_i\setminus B_i$ is homeomorphic to $f_i$ (i.e., $\Sigma_i$ is homeomorphic to $\Sigma_{f_i}$).
Let $\theta_i:\Sigma_i\setminus B_i\to f_i$ be a homeomorphism that extends to a continuous mapping 
$\theta_i:\Sigma_i\to\overline{f_i}$, where $\overline{f_i}$ denotes the closure of $f_i$.
Let $G_i$ be the inverse image of $G\cap \overline{f_i}$ under $\theta_i$.
Then $G_i$ is a graph normally embedded in $\Sigma_i$.  We say that the set of embedded graphs $\{G_i:1\le i\le m\}$ 
is a {\em $G$-expansion of $S$.}
Note that there is a one-to-one correspondence between the boundary walks of the faces of $S$ and the rings of the graphs
in the $G$-expansion of $S$; however, each vertex of $J$ may be split to several copies.
For $1\le i\le m$, we let ${\RR}_i$ be the set of rings of $G_i$,
where each vertex-like ring $R$ is weak if and only if $R$ is also a weak vertex-like ring of $G$.
We say that the rings in ${\RR}_i$ are the {\em natural rings} of $G_i$.

Let now $G'$ be another $\RR$-critical graph embedded in $\Sigma$ with rings $\RR$.
Suppose that there exists a collection $\{(J_f,S_f):f\in F(G')\}$ of subgraphs $J_f$ of $G$ and sets $S_f$ of faces of $J_f\cup\bigcup\RR$ such that $J_f$
is the union of the boundary walks of the faces of $S_f$, and a set $X\subset F(G)$ such that 
\begin{itemize}
\item for every $f\in F(G')$, the graph $J_f$ is not equal to $\bigcup\RR$,
\item for every $f\in F(G')$, the surfaces embedding the components of the $G$-expansion of $S_f$ are fragments of $\Sigma_f$,
\item for every face $h\in F(G)\setminus X$, there exists unique $f\in F(G')$ such that $h$ is a subset of a member of $S_f$, and
\item if $X\neq \emptyset$, then $X$ consists of a single closed $2$-cell face of length $6$.
\end{itemize}
We say that $X$ together with this collection forms a \emph{cover of $G$ by faces of $G'$}.
We define the \emph{elasticity} $\el(f)$ of a face $f\in F(G')$ to be $\left(\sum_{h\in S_f} |h|\right)-|f|$.

We now want to bound the weight of $G$ by the weight $G'$.
To this end, we define a \emph{contribution} $c(f')$ of a face $f'$ of $G'$ that bounds the difference between the weight of $f'$
and the weight of the corresponding subgraph of $G$.  We only define the contribution in the case that every face of $G'$ is either closed $2$-cell of length at least $5$ or omnipresent.
The contribution $c(f')$ of an omnipresent face $f'$ of $G'$ is defined as follows.  Let $G'_1$, $G'_2$, \ldots, $G'_k$
be the components of $G'$ such that $G'_i$ contains the ring $R_i\in{\RR}$.  If there exist distinct indices $i$ and $j$
such that $G'_i\neq R_i$ and $G'_j\neq R_j$, then $c(f')=1$.  Otherwise, suppose that $G'_i=R_i$ for $i\ge 2$.
If $G'_1$ is very exceptional, then $c(f')=-\infty$.  If $G'_1$ satisfies (E4) or (E5),
then $c(f')=5-\el(f')-5s(5)$, otherwise $c(f')=5-\el(f')+5s(5)$.

For a closed $2$-cell face, the definition of the contribution can be found in \cite{trfree2}; here, we only
use its properties given by the following theorem, which was proved as \cite[Theorem~9.1]{trfree2}.

\begin{theorem}\mylabel{thm:summary}
Let $G$ be a well-behaved graph embedded in a surface $\Sigma$ of Euler genus $g$ with rings $\RR$
satisfying {\rm (I0)--(I9)} and let $M$ be a subgraph of $G$ that captures $(\le\!4)$-cycles.
Assume that $g>0$ or $|\RR|>1$, and that $w(G,\RR)>8g+8|{\cal R}|+(2/3+26\epsilon)\ell(\RR)+20|E(M)|/3-16$.
If $G$ is $\RR$-critical, then there exists an $\RR$-critical graph $G'$ embedded in $\Sigma$ with the same rings $\RR$
such that $|E(G')|<|E(G)|$, every vertex-like ring of $G$ is also vertex-like in $G'$, and the following conditions hold.
\begin{itemize}
\item[(a)] If $G$ has girth at least five, then there exists a set $Y\subseteq V(G')$ of size at most two such that $G'-Y$ has
girth at least five.  
\item[(b)] If $C'$ is a $(\le\!4)$-cycle in $G'$, then $C'$ is non-contractible and $G$ contains a non-contractible
cycle $C$ of length at most $|C'|+3$ such that
\begin{enumerate}
\item at most one ring vertex of $C'$ does not belong to $C$, and if $r\in V(C')\setminus V(C)$ is a ring vertex,
then there exists a path in $G-E(M)$ of length at most three from $r$ to a vertex of $C$,
\item if $C\not\subseteq M$ and $C'$ only contains ring vertices, then $V(C')\subseteq V(C)$,
\item if $C\subseteq M$, then $|C|=|C'|$ and $C\cap \bigcup \RR\subseteq C'$,
\item if $C'$ is a triangle disjoint from the rings
and its vertices have distinct pairwise non-adjacent neighbors in a ring $R$ of length $6$, then
the distance between $C$ and $R$ in $G$ is at most one.
\end{enumerate}
\item[(c)] $G'$ has a face that either is not semi-closed $2$-cell or has length at least $6$.
\item[(d)] There exists $X\subset F(G)$ and a collection $\{(J_f,S_f):f\in F(G')\}$ forming a cover of $G$ by faces of $G'$,
such that $\sum_{f\in F(G')} \el(f)\le 10$,
and if $f$ is an omnipresent face, then $\el(f)\le 5$.
Furthermore, if every face of $G'$ is semi-closed $2$-cell or omnipresent, $G'$ satisfies (I6), and every non-isolated vertex-like ring of $G'$ is also vertex-like in $G$,
then $\sum_{f\in F(G')} c(f)\ge |X|s(6)$.
\item[(e)] If every vertex-like ring of $G'$ is also vertex-like in $G$,
$f\in F(G')$ is semi-closed $2$-cell and $G_1$, \ldots, $G_k$ are the components of the $G$-expansion of $S_f$, where
$S_f$ is as in (d) and for $1\le i\le k$, $G_i$ is embedded in
the disk with one ring $R_i$, then $\sum_{i=1}^k w(G_i,\{R_i\})\le s(|f|)-c(f)$.
\item[(f)] If $G'$ has an omnipresent face, then at least one component of $G'$ is not very exceptional.
\end{itemize}
\end{theorem}
Unfortunately, we made several formulation mistakes in the (b) part of~\cite[Theorem~9.1]{trfree2},
as well as in \cite[Lemma~6.2]{trfree2} from which the part (b) follows.  Hence, in the Appendix,
we give the arguments necessary to establish the corrected formulation.  We also forgot to include part (f),
which however follows from \cite[Lemma~7.3]{trfree2}.  Finally, in the last sentence of (d), we had
``\ldots every vertex-like ring of $G'$ \ldots'' instead of having the constraint apply only to non-isolated vertex-like rings;
this assumption is only used to prove \cite[Lemma~7.1]{trfree2}, and a quick inspection of its proof shows that
it suffices to have the constraint apply only to non-isolated vertex-like rings.

A graph $G$ embedded in a surface $\Sigma$ with rings $\RR$ has {\em internal girth at least five} if every $(\le\!4)$-cycle in $G$
is equal to one of the rings.  
The main result of \cite{trfree2} (Theorem 8.5) bounds the weight of graphs embedded in the disk with one ring.

\begin{theorem}
\mylabel{thm:diskgirth5}
Let $G$ be a graph of internal girth at least $5$ embedded in the disk with one ring $R$.
If $G$ is $\{R\}$-critical, then
\begin{itemize}
\item $|R|\ge 8$ and $w(G,\{R\})\le s(|R|-3)+s(5)$, and furthermore,
\item if $R$ does not satisfy (E1), then $|R|\ge 9$ and $w(G,\{R\})\le s(|R|-4)+2s(5)$,
\item if $(G,R)$ is not very exceptional, then $|R|\ge 10$ and $w(G,\{R\})\le s(|R|-5)+5s(5)$, and
\item if $(G,R)$ is not exceptional, then $|R|\ge 11$ and $w(G,\{R\})\le s(|R|-5)-5s(5)$.
\end{itemize}
\end{theorem}
Let us remark that in \cite{trfree2}, we prove the claim for graphs of girth at least $5$,
rather than internal girth at least $5$.  If $|R|\ge 5$, then the assumption of internal girth at least $5$
is equivalent to having girth at least five.  And, Aksenov~\cite{aksenov} proved that if $G$ is
a planar graph containing exactly one cycle $R$ of length $3$ or $4$ and with all other cycles of length at least $5$,
then any precoloring of $R$ extends to a $3$-coloring of $G$; or equivalently, there exist no $\{R\}$-critical
graphs of internal girth at least $5$ embedded in the disk with one ring $R$ of length at most $4$.

We will also need the following property of critical graphs.

\begin{lemma}
\mylabel{lem:diskcritical}
Let $G$ be a graph in a surface $\Sigma$ with rings $\RR$,
and assume that $G$ is $\RR$-critical.
Let $C$ be a non-facial cycle in $G$ bounding an open disk $\Delta\subseteq \Sigma$ disjoint from the rings,
and let $G'$ be the graph consisting of the vertices and edges of $G$ drawn in the closure of $\Delta$.
Then $G'$ may be regarded as graph embedded in the disk with one ring $C$,
and as such it is $\{C\}$-critical.
\end{lemma}

This together with Theorem~\ref{thm:diskgirth5} implies that property (I9) holds for all embedded critical graphs without
contractible $(\le\!4)$-cycles.
Lemma~\ref{lem:diskcritical} is a special case of the following result.
\begin{lemma}
\mylabel{lem:surfcritical}
Let $G$ be a graph in a surface $\Sigma$ with rings $\RR$,
and assume that $G$ is $\RR$-critical.
Let $J$ be a subgraph of $G$ and let $S$ be a subset of faces of $J\cup\bigcup\RR$
such that $J$ is the union of the boundary walks of the faces of $S$.
Let $G'$ be an element of the $G$-expansion of $S$ and let ${\RR}'$ be its natural rings.
If $G'$ is not equal to the union of the rings in ${\RR}'$, then $G'$ is ${\RR'}$-critical.
\end{lemma}
\begin{proof}
Consider any edge $e'\in E(G')$ that does not belong to any of the rings in ${\RR'}$.
By the definition of $G$-expansion, there is a unique edge $e\in E(G)$ corresponding to $e'$.
Since $G$ is $\RR$-critical, there exists a precoloring $\psi$ of $\RR$ that does not extend to a $3$-coloring of $G$,
but extends to a $3$-coloring $\phi$ of $G-e$.  We define a precoloring $\psi'$ of ${\RR'}$ in the natural way:
each ring vertex $v'\in V(G')$ to be precolored corresponds to a unique vertex $v\in V(G)$, and we set $\psi'(v')=\phi(v)$.
Observe that $\phi$ corresponds to a $3$-coloring of $G'-e'$ that extends $\psi'$.  If $\psi'$ extends to a $3$-coloring $\phi'$ of $G'$, then define $\phi_1$ in the following way.
If $v\in V(G)$ corresponds to no vertex of $G'$, then $\phi_1(v)=\phi(v)$.  If $v\in V(G)$ corresponds to at least one vertex $v'\in V(G')$,
then $\phi_1(v)=\phi'(v')$.  Note that if $v$ corresponds to several vertices of $G'$, then all these vertices belong to rings that are not weak vertex-like
and all of them have color $\phi(v)$, thus the definition does not depend on which of these vertices we choose.  Observe that
$\phi_1$ is a $3$-coloring of $G$ extending $\psi$, which is a contradiction.  Therefore, $\psi'$ extends to a $3$-coloring of $G'-e'$,
but not to a $3$-coloring of $G'$.  Since this holds for every choice of $e'$, it follows that $G'$ is ${\RR}'$-critical.
\end{proof}

Similarly, one can prove the following:

\begin{lemma}\mylabel{lemma-crcon}
Let $G$ be a graph in a surface $\Sigma$ with rings $\RR$,
and assume that $G$ is $\RR$-critical.  Let $c$ be a simple closed curve in $\Sigma$
intersecting $G$ in a set $X$ of vertices.  Let $\Sigma'_0$ be one of the surfaces obtained from $\widehat{\Sigma}$ by
cutting along $c$, and let $\Sigma_0=\Sigma'_0\cap \Sigma$.  Let us split the vertices of $G$ along $c$, let $G'$ be the part of the resulting graph embedded in $\Sigma_0$,
let $X'$ be the set of vertices of $G'$ corresponding to the vertices of $X$ and let ${\RR}'\subseteq {\RR}$ be the rings of $G$ that
are contained in $\Sigma_0$.
Let $\Delta$ be an open disk or a disjoint union of two open disks disjoint from $\Sigma'_0$ such that the boundary of $\Delta$ is equal to the cuff(s) of $\Sigma'_0$ corresponding to $c$.
Let $\Sigma'=\Sigma_0\cup \Delta$.
Let $Y$ consist of all vertices of $X'$ that are not incident with a cuff in $\Sigma'$.  For each $y\in Y$, choose an open disk $\Delta_y\subset \Delta$
such that the closures of the disks are pairwise disjoint and the boundary of $\Delta_y$ intersects $G'$ exactly in $y$.
Let $\Sigma''=\Sigma'\setminus \bigcup_{y\in Y}\Delta_y$.  For each $y\in Y$, add to $G'$ a triangle $R_y$ with $y\in V(R_y)$ tracing the boundary of $\Delta_y$,
and let ${\RR}''={\cal R'}\cup \{R_y:y\in Y\}$, where the rings $R_y$ are considered as non-weak vertex-like rings,
and furthermore, all weak vertex-like rings whose main vertices belong to $X'\setminus Y$
are turned into non-weak vertex-like rings.
If $G'$ is not equal to the union of the rings in ${\RR}''$, then $G'$ is ${\RR}''$-critical.
\end{lemma}

In particular, if $G'$ is a component of an ${\RR}$-critical graph, ${\RR}'$ are the rings contained in $G'$ and
$G'$ is not equal to the union of ${\RR}'$, then $G'$ is ${\RR}'$-critical.

\section{$(\le\!4)$-cycles on a cylinder}

The most technically difficult part of the proof of Theorem~\ref{thm:mainsurf} is dealing with long cylindrical subgraphs
of the considered graph.  We work out the details of this situation in the following two sections.
We start with the case of a graph embedded in the cylinder with rings of length at most four.
We will need the following result on graphs embedded in the disk with a ring of length at most twelve, which follows
from the results of Thomassen~\cite{thomassen-surf}.

\begin{theorem}\label{thm-planechar}
Let $G$ be a graph of girth $5$ embedded in the disk with a ring $R$ such that $|R|\le 12$.  If $G$ is $\{R\}$-critical and
$R$ is an induced cycle, then

\begin{itemize}
\item[(a)] $|R|\ge 9$ and $G-V(R)$ is a tree with at most $|R|-8$ vertices, or
\item[(b)] $|R|\ge 10$ and $G-V(R)$ is a connected graph with at most $|R|-5$ vertices
containing exactly one cycle, and the length of this cycle is $5$, or
\item[(c)] $|R|=12$ and every second vertex of $R$ has degree two and is contained in a facial $5$-cycle.
\end{itemize}
\end{theorem}

\begin{figure}
\begin{center}\includegraphics{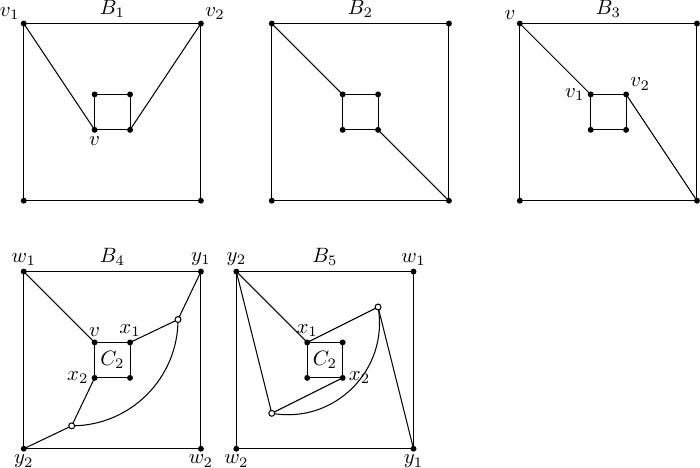}\end{center}
\caption{Maximal basic graphs.}\label{fig-basic}
\end{figure}

A graph $H$ embedded in the cylinder with (vertex-disjoint) rings $C_1$ and $C_2$ of length at most $4$ is {\em basic} if
every contractible cycle in $H$ has length at least five, $H$ is $\{C_1,C_2\}$-critical, and one of the following holds:
\begin{itemize}
\item $H$ contains a triangle, or
\item $H$ is not $2$-connected, or
\item the distance between $C_1$ and $C_2$ is one and $|V(H)\setminus V(C_1\cup C_2)|\le 2$.  
\end{itemize}
Consider a basic $2$-connected triangle-free graph.  We can cut the graph along a shortest path between $C_1$ and $C_2$,
resulting in a graph embedded in a disk bounded by a cycle $C$ of length $10$.  Note that the resulting graph is
$\{C\}$-critical by Lemma~\ref{lem:surfcritical}.  A straightforward case analysis using Theorem~\ref{thm-planechar} shows that every $2$-connected triangle-free basic graph
is a subgraph of one of the graphs drawn in Figure~\ref{fig-basic}.

\begin{figure}
\begin{center}\includegraphics{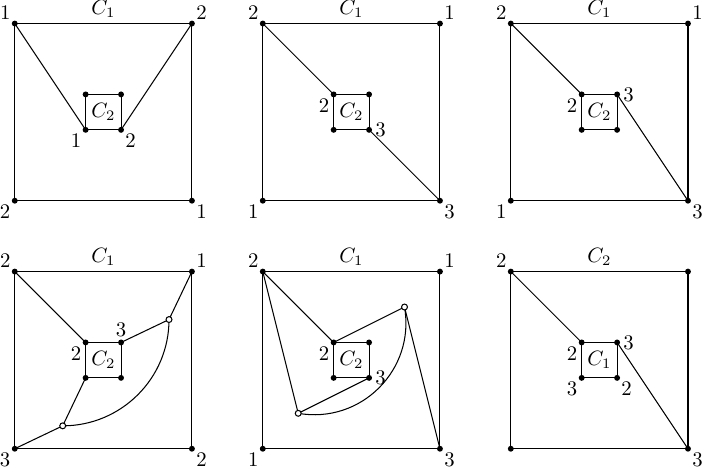}\end{center}
\caption{Colorings of basic graphs from (\ref{cl-basicsim}).}\label{fig-bascol}
\end{figure}

Observe furthermore that these graphs have the following properties.
\claim{cl-basicsim}{Let $C_1$ and $C_2$ be the rings of a triangle-free $2$-connected basic graph $H$.
There exists a $3$-coloring $\psi$ of $C_1$, vertices $v_1,v_2\in V(C_2)$ and colors $c_1\neq c_2$ such that
if $\phi$ is a $3$-coloring of $C_1\cup C_2$ matching $\psi$ on $C_1$ and satisfying $\phi(v_i)\neq c_i$
for $i\in\{1,2\}$, then $\phi$ extends to a $3$-coloring of $H$.}
The colorings for (\ref{cl-basicsim}) are indicated in Figure~\ref{fig-bascol}.

\claim{cl-twoone}{Let $C_1$ and $C_2$ be the rings of a triangle-free $2$-connected basic graph $H$, let $v_1$ and $v_2$ be
two distinct vertices of $C_1$ and let $c_1\neq c_2$ be two colors.  There exists a vertex $v\in V(C_2)$ and a color $c$ such that
every $3$-coloring $\psi$ of $C_2$ such that $\psi(v)\neq c$ extends to a $3$-coloring $\phi$ of $H$ satisfying
$\phi(v_1)\neq c_1$ and $\phi(v_2)\neq c_2$.}
\begin{proof}
Let us label the triangle-free $2$-connected basic graphs and their vertices as in Figure~\ref{fig-basic}.
If $H$ is $B_1$ or $B_3$ and $v_1$ and $v_2$ are as depicted, then set $c=c_2$ and let $v$ be the vertex indicated in the figure.  If $H=B_4$,
then let $c$ be the unique color distinct from $c_1$ and $c_2$ and let $v$ be the vertex indicated in the figure.

Consider any $3$-coloring $\psi$ of $C_2$.  For each vertex $w\in V(C_1)$,
let $L'_\psi(w)\subseteq \{1,2,3\}$ be the
list consisting of all colors not used by $\psi$ on the neighbors of $w$.  Let $L_\psi(w)=L'_\psi(w)$ if $w\neq v_i$ for $i\in\{1,2\}$,
and $L_\psi(w)=L'_\psi(w)\setminus \{c_i\}$ if $w=v_i$.

Suppose first that $H$ is $B_1$, $B_2$ or $B_3$.
Note that the sum of the sizes of the lists $L_\psi$ is at least $8$ and each of the lists
has size at least one.  Therefore, $C_1$ can be colored from the lists given by $L_\psi$, unless $C_1$ contains two adjacent
vertices with the same list of size one.  That is only possible if $H$ is $B_1$ or $B_3$ and $v_1$ and $v_2$ are as depicted in Figure~\ref{fig-basic}.
However, then we can choose the vertex $v$ as indicated in the figure and set $c=c_2$. This ensures that $c_2$ belongs to $L_\psi(v_1)\setminus L_\psi(v_2)$,
and thus $L_\psi(v_1)\neq L_\psi(v_2)$.  Therefore, $\psi$ extends to a $3$-coloring $\phi$ of $H$ satisfying
$\phi(v_1)\neq c_1$ and $\phi(v_2)\neq c_2$.

Suppose now that $H$ is $B_4$ or $B_5$.  Let us first consider the case that $\psi(x_1)\neq \psi(x_2)$.  If $H=B_4$, then we can by symmetry assume that $\{v_1,v_2\}\neq \{w_1,y_1\}$,
as otherwise we can swap the labels $x_1$ with $x_2$ and $y_1$ with $y_2$.  Let $L(w)=L_\psi(v)$ for $w\in V(C_1)\setminus \{y_1\}$
and $L(y_1)=L_\psi(y_1)\setminus\{\psi(x_2)\}$ and observe that any coloring of $C_1$ from lists given by $L$ extends to a $3$-coloring $\phi$ of $H$
matching $\psi$ on $C_2$ and satisfying $\phi(v_1)\neq c_1$ and $\phi(v_2)\neq c_2$.  Again, the sum of the sizes of the lists $L$
is at least $8$ and each of the lists has size at least one, thus such a coloring exists unless $C_1$ contains two adjacent
vertices with the same list of size one.  This is not possible, since if $H=B_4$, then $\{v_1,v_2\}\neq \{w_1,y_1\}$.

Finally, let us consider the case that $\psi(x_1)=\psi(x_2)$.  If there exists a coloring $\psi'$ of $C_1$ from lists $L_\psi$ such that
$\psi'(y_1)\neq\psi'(y_2)$, then the union of $\psi$ and $\psi'$
extends to a $3$-coloring $\phi$ of $H$, which clearly satisfies $\phi(v_1)\neq c_1$ and $\phi(v_2)\neq c_2$.  Let us find such a coloring $\psi'$.
If $C_1$ contains a vertex $w\not\in\{y_1,y_2\}$ such that $|L_\psi(w)|=3$, then it suffices to color the vertices of $Y=V(C_1)\setminus\{w\}$
by pairwise distinct colors from their lists and then color $w$ differently from its neighbors.  Such a coloring of $Y$ always exists, since
$\sum_{y\in Y}|L_\psi(y)|\ge 6$, all the lists have size at least one and at most three and if all of them have size two, then $v_1$ and $v_2$ belong to $Y$
and $L_\psi(v_1)=\{1,2,3\}\setminus \{c_1\}$ is different from $L_\psi(v_2)=\{1,2,3\}\setminus \{c_2\}$.  Therefore, we can assume that all
vertices in $V(C_1)\setminus \{y_1,y_2\}$ have lists of size at most two.

Suppose that either $H=B_5$, or $H=B_4$ and $|L_\psi(w_1)|=2$.
In this case $|L_\psi(w_1)|=|L_\psi(w_2)|=2$ and by symmetry, we can assume that $|L_\psi(y_1)|=3$ and $|L_\psi(y_2)|\ge 2$.
Let us choose a color $\psi'(w_1)=\psi'(w_2)\in L_\psi(w_1)\cap L_\psi(w_2)$ and then color $y_1$ and $y_2$ by distinct colors from
$L(y_1)\setminus\{\psi'(w_1)\}$ and $L(y_2)\setminus\{\psi'(w_1)\}$, respectively.

Therefore, we can assume that $H=B_4$ and $|L_\psi(w_1)|=1$.  Note that $|L_\psi(w_2)|=2$ and $|L_\psi(y_1)|=|L_\psi(y_2)|=3$.
By symmetry between $v_1$ and $v_2$, we can assume that $v_1=w_1$ and $v_2=w_2$.  The coloring $\psi'$ exists unless $L_\psi(w_2)=\{1,2,3\}\setminus L_\psi(w_1)$.
However, this is prevented by the choice of $c$.
\end{proof}

For a $4$-cycle $C=x_1x_2x_3x_4$, the {\em type} of its $3$-coloring $\lambda$ is the set of the vertices $x_i$ of $C$
such that $\lambda(x_i)\neq \lambda(x_{i+2})$ (where $x_5=x_1$ and $x_6=x_2$).  Note that the type of $\lambda$ is $\emptyset$,
$\{x_1,x_3\}$ or $\{x_2,x_4\}$.  In (\ref{cl-basicsim}), any coloring of the same type as $\psi$ has the same property,
possibly with different colors $c_1$ and $c_2$.

Let $G$ and $H$ be graphs with common rings $\{C_1,C_2\}$. We say that $H$ {\em subsumes} $G$ if every precoloring of $C_1\cup C_2$
that extends to a $3$-coloring of $H$ also extends to a $3$-coloring of $G$.

\begin{lemma}
\mylabel{lemma-cylbase}
Let $G$ be a graph embedded in the cylinder with rings $\{R_1, R_2\}$ of length at most $4$.
If every cycle of length at most $4$ in $G$ is non-contractible, then
there exists a basic graph $H$ with rings $\{R_1,R_2\}$ that subsumes $G$.
Furthermore, either $H=G$ or $|V(H)|+|E(H)|<|V(G)|+|E(G)|$.
\end{lemma}
\begin{proof}
Suppose for a contradiction that $G$ is a counterexample such that $|V(G)|+|E(G)|$ is minimal.
It follows that $G$ is $\{R_1,R_2\}$-critical, $2$-connected and triangle-free,
and in particular $|R_1|=|R_2|=4$.  Let $R_1=a_1a_2a_3a_4$ and $R_2=b_1b_2b_3b_4$,
where the labels are assigned in the clockwise order.  Since $G$ is triangle-free and all $4$-cycles are non-contractible, it follows that
every internal vertex has at most one neighbor in each of the rings.  

Suppose that $G$ contains a $5$-face $C=v_1v_2v_3v_4v_5$ such that all its vertices are internal and have degree three.  For $1\le i\le 5$, let
$x_i$ be the neighbor of $v_i$ different from $v_{i-1}$ and $v_{i+1}$ (where $v_0=v_5$ and $v_6=v_1$).  Observe that if $x_1=x_3$, then
$x_2\neq x_4$, thus by symmetry assume that $x_1\neq x_3$.  Let $G'=(G-V(C))+x_1x_3$.  Suppose that $K'$ is a cycle of length
at most $4$ in $G'$ that contains the edge $x_1x_3$.  Then $G$ contains a cycle $K$ of length at most $7$ obtained from $K'$ by replacing $x_1x_3$
by $x_1v_1v_2v_3x_3$.  Since $v_1$ and $v_2$ have neighbors on the opposite sides of this path, $K$ does not bound a face.
By Theorem~\ref{thm-planechar}, we conclude that $K$ and $K'$ are non-contractible.  Therefore, all $(\le\!4)$-cycles in $G'$
are non-contractible.
Furthermore, every precoloring of $R_1$ and $R_2$ that extends to a $3$-coloring of $G'$ also extends to a $3$-coloring of $G$ (the $3$-coloring
of $G'$ assigns different colors to $x_1$ and $x_3$, thus it can be extended to $C$).  Thus, $G'$ subsumes $G$, and consequently
it contradicts the minimality of $G$.
We conclude that
\claim{cl-no53}{every $5$-face in $G$ is incident with a ring vertex or a vertex of degree at least $4$.}

It follows that the distance between $R_1$ and $R_2$ is at least two: otherwise, if say $a_1$ is
adjacent to $b_1$, then apply Theorem~\ref{thm-planechar} to the graph obtained from $G$ by cutting open along the walk $a_1a_2\ldots a_1b_1b_2\ldots b_1$.  Outcome (b)
is excluded by (\ref{cl-no53}), thus $G-V(R_1\cup R_2)$ would have at most two vertices and $G$ would be basic.

Suppose that $G$ contains a face $C=v_1v_2\ldots v_k$ of length $k\ge 7$.  We may assume that $v_1$ is an internal vertex.  Let $G'$ be the graph
obtained from $G$ by identifying $v_1$ with $v_3$ to a vertex $v$.  Consider a cycle $K'\subseteq G'$ of length at most $4$ that does not appear in $G$.
Such a cycle corresponds to a cycle $K$ in $G$ of length at most $6$, obtained by replacing $v$ by $v_1v_2v_3$.  Since $v_1$ is an internal vertex,
$v_2$ cannot be a ring vertex of degree two.  It follows that $K$ does not bound a face and it is non-contractible by Theorem~\ref{thm-planechar}.  Therefore,
all $(\le\!4)$-cycles in $G'$ are non-contractible.  Furthermore, every $3$-coloring of $G'$ extends to a
$3$-coloring of $G$, and we obtain a contradiction with the minimality of $G$.  Therefore, each face of $G$ has length at most $6$.

Suppose that $G$ contains a face $C=v_1v_2\ldots v_6$ of length $6$.  We can assume that $v_1$ is an internal vertex.  If $v_3$ or $v_5$ is an internal vertex,
then let $G'$ be the graph obtained from $G$ by identifying $v_1$, $v_3$ and $v_5$ to a single vertex.  As in the previous paragraph, we obtain
a contradiction.  It follows that $v_3$ and $v_5$ are ring vertices, and by a symmetrical argument, two of $v_2$, $v_4$ and $v_6$ are ring vertices.
If $v_2$ is internal, then since the distance between $R_1$ and $R_2$ is at least two, we can assume that $V(R_1)=\{v_3,v_4,v_5,v_6\}$,
and thus $v_3$ and $v_6$ are adjacent.  In this situation, we consider the graph obtained from $G$ by identifying $v_1$ with $v_5$ and $v_2$
with $v_4$ (which is isomorphic to $G-\{v_4,v_5\}$, and thus contains no contractible $(\le\!4)$-cycles), and again obtain
a contradiction with the minimality of $G$.  Thus $v_2$ is not internal, and by symmetry, $v_6$ is not internal either.  Therefore, $v_4$ is
internal and $v_2$ and $v_6$ are ring vertices.  Since the distance between $R_1$ and $R_2$ is at least two,
we may assume that $v_2=a_2$, $v_3=a_3$, $v_5=b_4$ and $v_6=b_1$.  We apply Theorem~\ref{thm-planechar} to the $10$-cycle
$B=a_1a_2v_1b_1b_2b_3b_4v_4a_3a_4$.  The case (b) is excluded by (\ref{cl-no53}), thus either $B$ is not induced or (a) holds.
If $B$ is not induced, then its chord joins $v_1$ with $v_4$.  By Lemma~\ref{lem:diskcritical} and Theorem~\ref{thm-planechar},
we conclude that $G$ is the graph consisting of $R_1$, $R_2$, the paths $a_2v_1b_1$ and $b_4v_4a_3$, and the edge $v_1v_4$.
Observe that every precoloring $\psi$ of the rings that does not extend to a $3$-coloring
of $G$ satisfies $\psi(a_2)=\psi(b_4)$.  Therefore, $G$ is subsumed by the basic graph $H$ consisting of $R_1$, $R_2$ and the edge between $a_2$ and $b_4$.
If $B$ is an induced cycle, then by (a), $G-V(B)$ is a tree $F$ with at most two vertices.  If $F$ has only one vertex $w$, then $w$ cannot be adjacent to
both $v_1$ and $v_4$, hence one of these vertices has degree two, which is a contradiction.  If $V(F)=\{x,y\}$, then since $v_1$ and $v_4$ have degree
at least three, we can assume that $x$ is adjacent to $v_1$ and $a_4$ and $y$ is adjacent to $b_2$ and $v_4$.  However, by identifying $v_1$ with $b_4$ and $v_4$ with $a_2$,
we obtain a graph isomorphic to the graph $B_5$ of Figure~\ref{fig-basic}, which subsumes $G$.
Therefore,
\claim{cl-all5}{all faces of $G$ have length $5$.}

Together with Lemma~\ref{lem:diskcritical} and Theorem~\ref{thm-planechar}, this implies that

\claim{cl-no67}{$G$ contains no contractible cycles of length $6$ or $7$, and the disk bounded by any contractible $8$-cycle $K$ of $G$
consists of two $5$-faces separated by a chord of $K$.}

Suppose that $G$ contains a $4$-cycle $C=v_1v_2v_3v_4$ different from $R_1$ and $R_2$.  By the assumptions,
$C$ is non-contractible; for $i\in\{1,2\}$, let $G_i$ be the subgraph of $G$ drawn between $R_i$ and $C$ and
let $d_i$ be the distance between $R_i$ and $C$.

Let us first consider the case that $d_i\ge 1$ and $G_i$ is not basic for some $i\in\{1,2\}$.  By the minimality of $G$,
there exists a basic graph $G'_i$ that subsumes $G_i$ (considered to be embedded in a cylinder with rings $R_i$ and $C$) such that $|V(G'_i)|+|E(G'_i)|<|V(G_i)|+|E(G_i)|$.
Let $G'=G'_i\cup G_{3-i}$ and observe that $G'$ subsumes $G$ and $|V(G')|+|E(G')|<|V(G)|+|E(G)|$.
Note that every contractible cycle in $G'$ has length at least five, since neither $G'_i$ nor $G_{3-i}$ contains a contractible $(\le\!4)$-cycle
and $G_{3-i}$ is triangle-free.  Therefore, by the minimality of $G$, there exists a basic graph $H$ which subsumes $G'$ and $|V(H)|+|E(H)|\le|V(G')|+|E(G')|$.
However, then $H$ also subsumes $G$, which is a contradiction.

We conclude that if $d_i\ge 1$, then $G_i$ is a basic graph, and since it is $2$-connected and triangle-free, it follows that $d_i=1$.
Let us choose the labels of $R_1$ and $R_2$ and the cycle $C$ so that $d_1$ is as small as possible.
In particular, $d_1\le d_2$.  Let us discuss the possible cases:
\begin{itemize}
\item $d_1=d_2=0$: Since the distance between $R_1$ and $R_2$ is at least two, we conclude that $|V(R_1)\cap V(C)|=|V(R_1)\cap V(C)|=1$.
We can assume that $v_1=a_1$ and $v_3=b_3$.  By Theorem~\ref{thm-planechar}, the open disks bounded by the closed walks $a_1v_2v_3v_4a_1a_4a_3a_2$ and $b_3b_4b_1b_2b_3v_4a_1v_2$
contain no vertices, and since $v_2$ and $v_4$ have degree at least three, we may assume that $v_2$ is adjacent to $a_4$ and $v_4$ to $b_2$.  However, then
$G$ contains a triangle $a_1v_2a_4$, which is a contradiction.
\item $d_1=0$, $d_2=1$:  We may assume that $a_1=v_1$.  Since $G$ is triangle-free, (\ref{cl-no67}) implies that $|V(C)\cap V(R_1)|=1$
and $a_3v_3\in E(G)$.  Since $d_2=1$, $G_2$ is a basic graph, and by (\ref{cl-all5}), we conclude that $G_2$ is isomorphic to $B_4$ or $B_5$ from Figure~\ref{fig-basic}.
Let $w_1w_2=G_2-V(C\cup R_2)$.  Up to symmetry, there are two cases to consider:

\begin{figure}
\begin{center}\includegraphics[width=12cm]{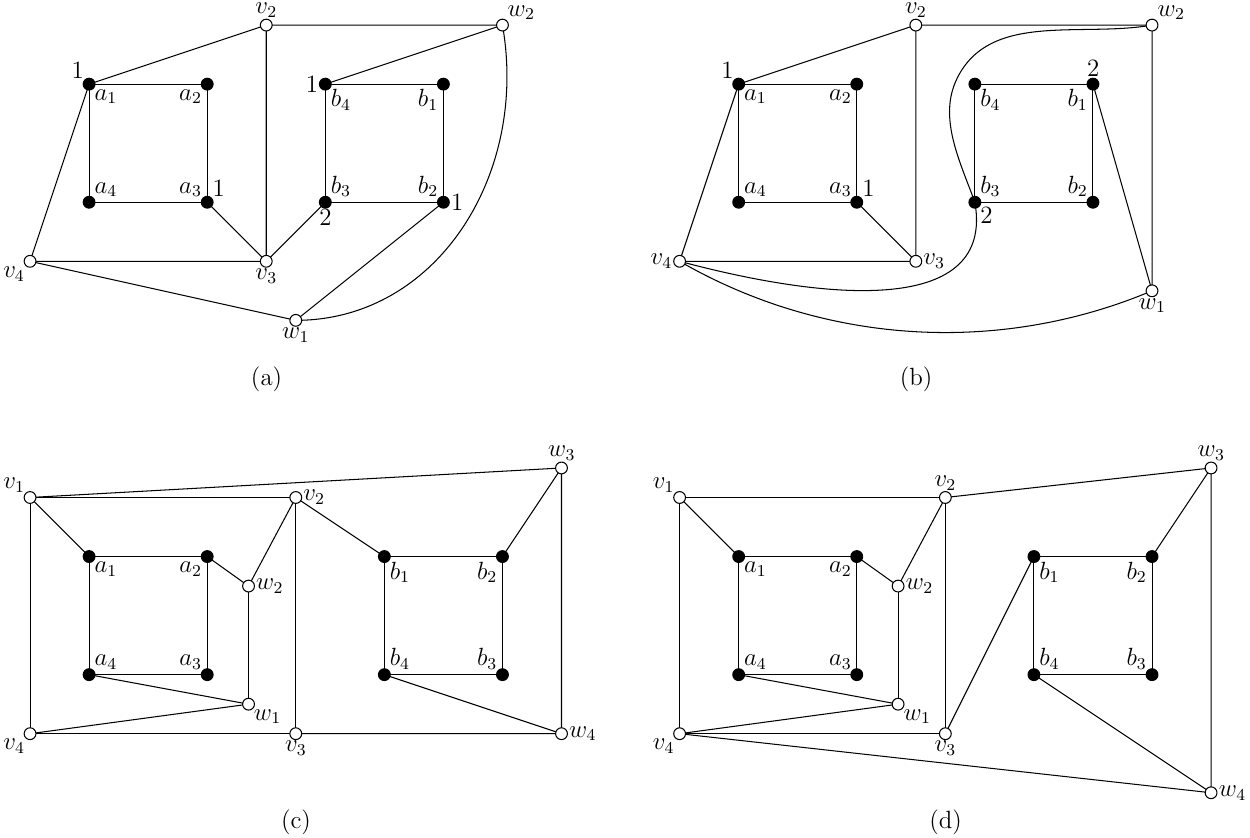}\end{center}
\caption{Cases in the proof of (\ref{cl-4only}); numbers indicate a non-extendable $3$-coloring.}\label{fig-basan}
\end{figure}

\begin{itemize}
\item $b_3$ is adjacent to $v_3$.  Since $v_2$ and $v_4$ have degree at least three, we can assume that $w_1$ is adjacent to $v_4$ and $b_2$ and
$w_2$ is adjacent to $v_2$ and $b_4$; see Figure~\ref{fig-basan}(a).  Note that every precoloring of $R_1\cup R_2$ that assigns $a_3$ and $b_3$ the
same color extends to a 3-coloring of $G$.  Also, the precolorings of $R_1\cup R_2$ that assign $a_3$ and $b_2$ the same color and do
not extend to a 3-coloring of $G$ are obtained from the one depicted in Figure~\ref{fig-basan}(a) by permuting the colors and coloring $a_2$, $a_4$ and $b_1$
arbitrarily.  Hence, $G$ is subsumed by the basic graph $H$ consisting of $R_1$, $R_2$ and a vertex $z$, with $a_1$ adjacent to $b_2$ and $z$ to $b_2$, $b_3$ and $a_3$.
\item $b_3$ is adjacent to $v_4$.  Since $v_2$ has degree at least three, we can assume that
$w_1$ is adjacent to $b_1$ and $v_4$ and $w_2$ is adjacent to $b_3$ and $v_2$; see Figure~\ref{fig-basan}(b).
Note that if $\phi$ is a precoloring of $R_1\cup R_2$ that does not extend to a $3$-coloring of $G$ and $\phi(a_1)\neq\phi(b_3)$, then
$\phi$ is obtained from the coloring depicted in Figure~\ref{fig-basan}(b) by permuting the colors and coloring $a_2$, $a_4$, $b_2$ and $b_4$
arbitrarily.
Hence, $G$ is subsumed by the basic graph $H$ consisting of $R_1$, $R_2$, adjacent vertices $z_1$ and $z_2$, and edges
$a_1b_3$, $a_1z_1$, $b_1z_1$, $b_3z_2$ and $a_3z_2$.
\end{itemize}
\item $d_1=d_2=1$: By the choice of $C$, $G$ does not contain a $4$-cycle distinct from $R_1$ and $R_2$ that intersects
one of them.  Additionally, all faces of $G$ have length $5$ and $G_1$ and $G_2$ are basic graphs.
The inspection of graphs in Figure~\ref{fig-basic} shows that $G_1$ and $G_2$ are isomorphic to $B_4$.
Hence, we can assume that $a_1$ is adjacent to $v_1$ and $G_1-V(R_1\cup C)=w_1w_2$ with $w_1$ adjacent to $a_4$ and $v_4$ and $w_2$ adjacent to $a_2$ and $v_2$.
Since $v_3$ has degree at least three, $v_1$ cannot have a neighbor in $R_2$, thus there are up to symmetry
two possible cases:
\begin{itemize}
\item $b_1$ is adjacent to $v_2$, $G_2-V(R_2\cup C)=w_3w_4$, and $w_3v_1, w_3b_2,w_4v_3,w_4b_4\in E(G)$; see Figure~\ref{fig-basan}(c).
If $\phi$ is any precoloring of $R_1\cup R_2$ that assigns $a_1$ and $b_2$ different colors, then we can give $v_1$ the color $\phi(b_2)$
and extend $\phi$ to a $3$-coloring of $G$ greedily.  Hence, $G$ is subsumed by the basic graph $H$ consisting of $R_1$, $R_2$ and the edge $a_1b_2$.
\item $b_1$ is adjacent to $v_3$, $G_2-V(R_2\cup C)=w_3w_4$, and $w_3v_2, w_3b_2,w_4v_4,w_4b_4\in E(G)$; see Figure~\ref{fig-basan}(d).  But then every precoloring $\phi$
of $R_1$ and $R_2$ extends to a $3$-coloring of $G$ (we can assign $v_2$ and $v_4$ the same color unless $\phi(a_2)=\phi(a_4)\neq \phi(b_2)=\phi(b_4)$,
in which case we can color $v_2$ and $v_4$ by distinct colors $c_2$ and $c_4$ so that $\phi(a_1),\phi(b_1)\in\{c_2,c_4\}$),
contrary to the assumption that $G$ is $\{R_1,R_2\}$-critical.
\end{itemize}
\end{itemize}

Therefore, 
\claim{cl-4only}{$R_1$ and $R_2$ are the only $4$-cycles in $G$.}
Suppose that $G$ has a face $C=v_1v_2v_3v_4v_5$ such that $v_2$, \ldots, $v_5$ are
internal vertices of degree three.  For $2\le i\le 5$, let $x_i$ be the neighbor of $v_i$ that is not incident with $C$.
By (\ref{cl-4only}), the vertices $x_i$ are distinct.
If at least one of $x_3$ and $x_4$ is internal, then let $G'$ be the graph obtained from $G-\{v_2,\ldots,v_5\}+x_2x_5$ by identifying $x_3$ with $x_4$ to a new vertex $x$.
Observe that every $3$-coloring of $G'$ extends to a $3$-coloring of $G$.  Furthermore, suppose that $K'$ is a cycle of length at most $4$ in $G'$
that does not appear in $G$, and let $K$ be the corresponding cycle in $G$ obtained by replacing $x_2x_5$ by $x_2v_2v_1v_5x_5$ or $x$ by $x_3v_3v_4x_4$ or
both.  If $|K|\le 7$, then since $K$ cannot bound a face, Theorem~\ref{thm-planechar} implies that $K$ and $K'$ are non-contractible.
If $|K|\ge 8$, then $K$ contains both $x_2v_2v_1v_5x_5$ and $x_3v_3v_4x_4$.  By symetry, we can assume that either $K'=x_5x_2x$ or $K'=x_5x_2xu$ for some vertex $u$.
Since $G$ is embedded in the cylinder, it cannot contain both the edge $x_2x_4$ and either the edge $x_3x_5$ or the path $x_3ux_5$.
It follows that $G$ contains the edge $x_2x_3$ (and either the edge $x_4x_5$ or the path $x_4ux_5$).
However, this is excluded by (\ref{cl-4only}).  It follows that all $(\le\!4)$-cycles in $G'$ are non-contractible,
and $G'$ is a smaller counterexample than $G$, which is a contradiction.

Let us now consider the case that both $x_3$ and $x_4$ are ring vertices.  Here,
we exclude the possibility that $x_3$ and $x_4$ belong to different rings:  If
that were the case, then we can assume that $x_3=a_1$ and $x_4=b_1$.   Since all 
faces of $G$ have length $5$, it follows that $x_4$ and
$x_5$ have a common neighbor $v$.  We apply Theorem~\ref{thm-planechar} to the
disk bounded by the closed walk $a_1a_2a_3a_4a_1vb_1b_2b_3b_4b_1v$ of length $12$.  By (\ref{cl-no53}), the case
(b) is excluded.  Since $v$ has degree at least three, $a_1vb_1$ cannot be incident with two $5$-faces
and the case (c) is excluded as well.  Therefore,
$G-V(R_1\cup R_2)-\{v\}$ is a tree with four vertices $v_2$, $v_3$, $v_4$ and
$v_5$.  By (\ref{cl-4only}), $v$ is not equal to $x_2$, $x_5$ or $v_2$, hence
two of these vertices belong to the same ring.  Since $G$ is triangle-free,
(\ref{cl-4only}) implies that no internal vertex has two neighbors in the same
ring, thus we can assume that $v_2\in V(R_1)$ and $x_2,x_5\in V(R_2)$.  However,
the path $x_2v_2v_3v_4v_5x_5$ together with a subpath of $R_2$ forms a cycle
that separates $v_2$ from $R_1$, which contradicts the assumption that $G$ is embedded in the cylinder.
Therefore,
\claim{cl-l5f}{if $C=v_1v_2v_3v_4v_5$ is a face such that $v_2$, \ldots, $v_5$ are internal vertices of degree three,
then for some $i\in \{1,2\}$, both $v_3$ and $v_4$ have a neighbor in $R_i$.}

Let us now assign charge to vertices and faces of $G$ as follows: each face $f$ gets the initial charge $|f|-4$ and each vertex $v$ gets the initial charge $\deg(v)-4$.  The sum
of the initial charges is $-8$.  Let us redistribute the charge: each \hbox{$5$-face} sends $1/3$ to each incident vertex $v$ such that $v$ is internal and has degree
three.  Furthermore, for each ring vertex $w$ of degree two, if there exists a face $f=v_1v_2v_3v_4v_5$ such all vertices incident with $f$ except for $v_1$
are internal of degree three and if $v_3v_4$ is incident with the same face as $w$, then $w$ sends $1/3$ to $f$.
Note that after this procedure, all vertices and faces have non-negative charge, with the following exceptions:
the ring vertices of degree two have charge at most $-7/3$ and the ring vertices of degree three have charge $-1$.
For $i\in\{1,2\}$, let $c_i$ be the sum of the charges of the vertices of $R_i$, together with the charges of the faces that share an edge with $R_i$
(such a face cannot share an edge with $R_{3-i}$, since the distance between $R_1$ and $R_2$ is at least two and all faces have length $5$).
Note that $c_1+c_2\le -8$,  and we may assume that $c_1\le -4$.

For $i\in\{1,2,3,4\}$, let $f_i$ denote the face sharing the edge $a_ia_{i+1}$ with $R_1$. If a vertex $a_i$ has degree three, then let $x_i$ denote its internal neighbor.
Since $G$ is $2$-connected, at most two vertices of $R_1$ have degree two.  Let us discuss several cases.
\begin{itemize}
\item $R_1$ contains two vertices of degree two:  Since all faces have length $5$ and $G$ is triangle-free, these two vertices are non-adjacent, say $a_2$ and $a_4$.
Similarly, since $G$ does not contain a $4$-cycle different from $R_1$ and $R_2$, both $a_1$ and $a_3$ have degree at least four, and since the sum of the charges of the
vertices of $R_1$ is at most $-4$, we conclude that $\deg(a_1)=\deg(a_3)=4$.  Let $f_2=a_1a_2a_3x'_3x'_1$ and $f_4=a_1a_4a_3x''_3x''_1$.
Note that both $f_2$ and $f_4$ send charge to at most two vertices, hence their final charge is $1/3$, and since $c_1\le -4$,
it follows that the charge of $a_2$ and $a_4$ is $-7/3$.  Therefore, the vertices $x'_1$, $x''_1$, $x'_3$, $x''_3$ and their neighbors distinct from
$a_1$ and $a_3$ are internal vertices of degree three.  However, these vertices form an $8$-cycle, contradicting the criticality of $G$.
\item $R_1$ contains one vertex of degree two, say $a_2$, and $a_1$, $a_3$ and $a_4$ have degree three: by (\ref{cl-all5}), $x_1$ is adjacent
to $x_3$, $x_1$ and $x_4$ have a common neighbor $x_{41}$ and $x_3$ and $x_4$ have a common neighbor $x_{43}$.
Suppose that $x_1$ and $x_3$ have degree three.  The path $x_{41}x_1x_3x_{43}$ is a part of a boundary of a $5$-face $f$;
let $y$ be the fifth vertex of $f$.  Then $x_{41}x_4x_{43}y$ is a $4$-cycle, contradicting (\ref{cl-4only}).  Therefore, we may assume that $x_1$ has degree greater than three.  This implies that $a_2$ does not send any charge and its final charge is $-2$.  Furthermore, $f_2$ has charge at least $2/3$
and $f_4$ has charge at least $1/3$, and thus $c_1=-4$.  Furthermore, $x_3$, $x_4$, $x_{41}$ and $x_{43}$
are internal and have degree three.
\item $R_1$ contains one vertex of degree two, say $a_2$, and at least one vertex of $R_1$ has degree at least four:  note that the sum of the charges
of $a_2$ and $f_2$ is at least $-2$.  It follows that exactly one vertex of $R_1$ has degree four, two vertices have degree three, and $c_1=-4$.
\item $R_1$ contains no vertices of degree two.  Since $c_1\le -4$, it follows that all vertices of $R_1$ have degree three and all internal vertices of the faces sharing an edge with $R_1$ have degree three.  But then $G$ contains an $8$-cycle of internal vertices of degree $3$, contradicting the criticality of $G$.
\end{itemize}

We conclude that $c_1=-4$, and by symmetry, $c_2=-4$.  It follows that all charges that are not counted in $c_1$ and $c_2$ are equal to zero.
Let us now go over the possible
cases for the neighborhood of $R_1$ again, keeping the notation established in the previous paragraph:
\begin{itemize}
\item $R_1$ contains one vertex of degree two, say $a_2$, and $a_1$, $a_3$ and $a_4$ have degree three:  Since all internal vertices have zero charge, $x_1$ has degree exactly four.
Let $y_1$, $y_{41}$ and $y_{43}$ be the neighbors of $x_1$, $x_{41}$ and $x_{43}$, respectively, not incident with $f_2$, $f_3$ and $f_4$.  By (\ref{cl-all5}),
$y_{43}$ is adjacent to $y_1$ and to $y_{41}$, and the vertices $y_1$ and $y_{41}$ have a common neighbor $z$ distinct from $y_{43}$.
By (\ref{cl-4only}), we have $R_2=y_1y_{43}y_{41}z$.
However, then we can set $H$ to be the graph consisting of $R_1$, $R_2$ and a vertex $w$, with edges $a_4y_{41}$, $wy_1$, $wa_1$ and $wa_4$.
\item $R_1$ contains one vertex of degree two, say $a_2$, one vertex of degree four and two of degree three.
Let $a_i$ be the vertex of degree four and $x'_i$ and $x''_i$ its internal neighbors.
Since $c_1=-4$, all internal vertices incident with the faces $f_2$, $f_3$ and $f_4$ have degree three,
and by (\ref{cl-all5}) they form a path $P$ with ends $x'_i$ and $x''_i$.  Furthermore, $x'_i$ and $x''_i$ have adjacent neighbors $y'_i$ and $y''_i$.
We let $G'$ consist of $G-V(P)$ and a new vertex $w$ adjacent to $y'_i$, $y''_i$ and $a_i$, and observe that every $3$-coloring of $G'$
extends to a $3$-coloring of $G$.  This contradicts the minimality of $G$.
\end{itemize}
\end{proof}

\begin{figure}
\begin{center}
\includegraphics{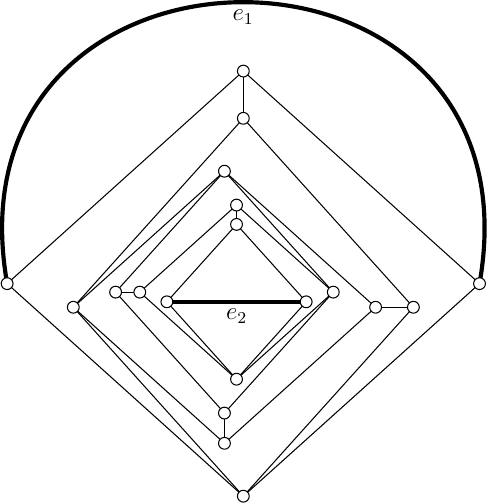}
\end{center}
\caption{An example of an $(e_1,e_2)$-chain}\label{fig-chain}
\end{figure}

A graph $G$ is an {\em $(e_1,e_2)$-chain} if either $G$ is the complete graph on four vertices and $e_1$ and $e_2$ form a
matching in $G$; or,
there exists an $(e_1,u_1u_2)$-chain $H$ and $G$ consists of $H-u_1u_2$, vertices $y_1$, $y_2$ and $u'_2$
and edges $y_1y_2$, $u_2u'_2$, $u_1y_1$, $u_1y_2$, $u'_2y_1$ and $u'_2y_2$, where $e_2=y_1y_2$.  See Figure~\ref{fig-chain} for an illustration.
Let us note that each $(e_1,e_2)$-chain
is a planar graph with chromatic number $4$ containing exactly four triangles (two incident with each of $e_1$ and $e_2$),
and all other faces of $G$ have length $5$.  The graph $G$ can be embedded in the Klein bottle by putting crosscaps on the edges $e_1$ and $e_2$; we call such an embedding {\em canonical}.
Note that no cycle of length less than $5$ is contractible in a canonical embedding of $G$, and that all $4$-cycles of the canonical embedding are separating
(cutting along any of them splits the Klein bottle in two M\"obius strips).
Thomas and Walls~\cite{tw-klein} proved the following:

\begin{theorem}\label{thm-klein}
If $G$ is a $4$-critical graph embedded in the Klein bottle so that no cycle of length at most $4$ is contractible,
then $G$ is a canonical embedding of an $(e_1,e_2)$-chain, for some edges $e_1,e_2\in E(G)$.
\end{theorem}

For the torus, Thomassen~\cite{thom-torus} showed that the situation is even simpler.
\begin{theorem}\label{thm-torus}
If $G$ is embedded in the torus so that no cycle of length at most $4$ is contractible, then $G$ is $3$-colorable.
\end{theorem}

Aksionov~\cite{aksenov} proved that if $G$ is a planar graph, $C$ is a $(\le\!4)$-cycle in $G$, and $G$ contains at most one triangle distinct from $C$,
then any precoloring of $C$ extends to a $3$-coloring of $G$.  As a corollary, we get the following.
\begin{theorem}\label{thm-aksen}
Let $G$ be a graph embedded in the cylinder with rings $R_1$ and $R_2$ with $|R_1|\le4$, such that every $(\le\!4)$-cycle
in $G$ is non-contractible.  Let $G_1$ be the component of $G$ that contains $R_1$.  If $R_2$ is not contained in $G_1$,
then every precoloring of $R_1$ extends to a $3$-coloring of $G_1$.  In particular, if $G$ is $\{R_1,R_2\}$-critical and not connected,
then $R_1$ forms a connected component of $G$.
\end{theorem}
\begin{proof}
To prove the first claim, let $K_1=R_1, K_2,K_3,\ldots, K_m$ be a maximal sequence of $(\le\!4)$-cycles in $G_1$ such that
for $1\le i<j\le m$, the cycle $K_i$ is contained in the subgraph of $G$ between $R_1$ and $K_j$.
For $i=1,\ldots, m-1$ in turn, we apply the aforementioned result of Aksionov~\cite{aksenov} to the subgraph between $K_i$ and $K_{i+1}$,
gradually extending the $3$-coloring to $G_1$.

For the second claim, note that by Theorem~\ref{grotzsch} each component of $G$ contains a ring,
and thus if $G$ is not connected, then it has exactly two components, the component $G_1$ containing $R_1$ and another component containing $R_2$.
Since every precoloring of $R_1$ extends to $G_1$, the $\{R_1,R_2\}$-criticality of $G$ implies that $G_1=R_1$.
\end{proof}

Let us now give a description of $\{R_1,R_2\}$-critical graphs on a cylinder, where $|R_1|,|R_2|\le 3$.

\begin{lemma}\label{lemma-critshort}
Let $G$ be an $\{R_1,R_2\}$-critical graph embedded in the cylinder, where $|R_1|,|R_2|\le 3$.
If every cycle of length at most $4$ in $G$ is non-contractible, then one of the following holds:
\begin{itemize}
\item $G$ consists of $R_1$, $R_2$ and an edge between them, or
\item neither $R_1$ nor $R_2$ is vertex-like and $G$ consists of $R_1$, $R_2$ and two edges between them, or
\item neither $R_1$ nor $R_2$ is vertex-like and $G$ consists of $R_1$, $R_2$ and two adjacent vertices of degree three,
each having a neighbor in $R_1$ and in $R_2$.
\end{itemize}
\end{lemma}
\begin{proof}
By Theorem~\ref{thm-aksen}, we have that $G$ is connected.  By Lemma~\ref{lemma-unweak}, we may assume that neither $R_1$ nor $R_2$ is a weak vertex-like ring.
If the distance between $R_1$ and $R_2$ is at most two, then let $J$ be the subgraph of $G$ equal
to the union of $R_1$, $R_2$ and the shortest path between them and let $f$ be the face of $J$.  Let $G'$ be the unique element of
the $G$-expansion of $\{f\}$, and let $R'$ be its natural ring.  Note that $G'$ is embedded in the disk and $|R'|\le 10$,
and that either $G'=R'$ or $G'$ is $R'$-critical by Lemma~\ref{lem:surfcritical}.  If $G'=R'$, then $G$ consists of $R_1$, $R_2$ and an edge between them,
and hence the first outcome of the lemma holds.
If $G$ is equal to $R'$ with a chord, then $G$ consists of $R_1$, $R_2$ and two edges between them,
and hence the second outcome of the lemma holds.  If $G'\neq R'$ and $R'$ is an induced cycle, then
$G'$ is one of the graphs described in Theorem~\ref{thm-planechar}(a) or (b).  As the corresponding graph $G$ must be $\{R_1,R_2\}$-critical,
a straightforward case analysis shows that this is only possible if $G$ is one of the graphs described in the last outcome of this lemma.
Therefore, assume that the distance between $R_1$ and $R_2$ is at least three.

Since $G$ is $\{R_1,R_2\}$-critical, there exists a precoloring $\psi$ of $R_1\cup R_2$ that
does not extend to a $3$-coloring of $G$.  We identify the vertices of $R_1$ and $R_2$
to which $\psi$ assigns the same color and we obtain a graph $G'$ embedded in the torus or in the Klein bottle (in the latter case, we can assume that neither $R_1$ nor $R_2$
is vertex-like, as otherwise we can exchange the colors of vertices of $R_1$ or $R_2$ of degree two before the identification).
Note that $G'$ has no loops, since $R_1$ and $R_2$ are not adjacent.  Observe also that $G'$ contains no contractible $(\le\!4)$-cycle.
Since $G'$ is not $3$-colorable, Theorems~\ref{thm-klein} and \ref{thm-torus} imply that $G'$ is embedded in the Klein bottle
and contains a canonical embedding of an $(e_1,e_2)$-chain as a subgraph.  Therefore, $G'$ contains a separating non-contractible $4$-cycle $C$.
Let $c$ be a simple closed curve in the Klein bottle tracing $C$.  Cutting the Klein bottle along the triangle $R$ obtained by the identification of $R_1$ with $R_2$
splits $c$ into several curves with ends in $R_1\cup R_2$ ($c$ intersects $R$, since $c$ is non-contractible).  Let $n_{ij}$ denote the number of these curves with
the starting point in $R_i$ and the ending point in $R_j$.  Since $c$ is $2$-sided and non-contractible, we conclude that $n_{12}+n_{21}$ is even and non-zero.
Consequently, the subgraph of $G$ corresponding to $C$ contains at least two paths joining $R_1$ and $R_2$.  However, this implies that the distance
between $R_1$ and $R_2$ is at most two, which is a contradiction.
\end{proof}

\begin{corollary}\label{cor-critshort}
Let $G$ be an $\{R_1,R_2\}$-critical graph embedded in the cylinder, where $R_1$ is a weak vertex-like ring.
If every cycle of length at most $4$ in $G$ is non-contractible, then $R_2$ is a ring of length at least $4$.
\end{corollary}
\begin{proof}
By Lemma~\ref{lemma-critshort}, if $|R_2|\le 3$, then $G$ consists of an edge joining a vertex of $R_2$ with $R_1$.
However, since $R_1$ is weak, every precoloring of $\{R_1,R_2\}$ extends to a $3$-coloring of $G$, contradicting the
assumption that $G$ is $\{R_1,R_2\}$-critical.
\end{proof}

\begin{figure}
\begin{center}\includegraphics{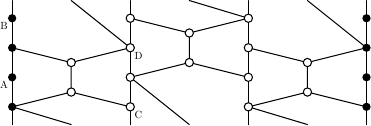}\end{center}
\caption{Arbitrarily large critical graph with rings of length four.}\label{fig-cex4}
\end{figure}

Finally, we give a similar result for $\{R_1,R_2\}$-critical graphs, where each of $R_1$ and $R_2$ has length at most four.
A {\em broken chain} is a graph obtained from an $(e_1,e_2)$-chain by removing the edges $e_1$ and $e_2$, see
Figure~\ref{fig-cex4} for an illustration (the top of the picture is identified with the bottom, giving an embedding in the cylinder).
Note that in any $3$-coloring of the graph depicted in Figure~\ref{fig-cex4},
if $A$ and $B$ have different colors, then the colors of $C$ and $D$ differ as well.  Repeating this observation, we conclude that if the colors
of $A$ and $B$ differ, then the colors of the corresponding vertices of the rightmost cycle differ as well.
Consequently, this gives an example of an $\{R_1,R_2\}$-critical graph embedded in
the cylinder, where $R_1$ and $R_2$ are arbitrarily distant $4$-cycles.

Dvo\v{r}\'ak and Lidick\'y~\cite{dvolid} gave a
complete list of such $\{R_1,R_2\}$-critical graphs embedded in a cylinder without contractible $(\le\!4)$-cycles
(other than broken chains, there are only finitely many).
However, their proof is computer assisted.
In this paper, we give a much weaker bound on the size of the graphs, which however suffices for our purposes.
We start with the case that there are many $(\le\!4)$-cycles separating $R_1$ from $R_2$.

\begin{lemma}\label{thm-compens}
Let $G$ be an $\{R_1,R_2\}$-critical graph embedded in the cylinder $\Sigma$, where $|R_1|,|R_2|\le 4$.
Suppose that every cycle of length at most $4$ in $G$ is non-contractible.  If $G$ contains at least $34$ cycles
of length at most $4$, then $|R_1|=|R_2|=4$ and $G$ is a broken chain.
\end{lemma}
\begin{proof}
Since $G$ is $\{R_1,R_2\}$-critical, it is not equal to $R_1\cup R_2$, and thus $G$ is connected by Theorem~\ref{thm-aksen}.
Let $C_1$ and $C_2$ be distinct cycles of length at most $4$ in $G$.  We claim that
$C_1$ bounds a closed disk in $\widehat{\Sigma}$ that contains $C_2$.  Indeed, otherwise each of the open disks in $\widehat{\Sigma}$
bounded by $C_1$ contains a vertex of $C_2$, and we conclude that the set $X=V(C_1)\cap V(C_2)$ has size two.
But then there exist three disjoint paths of length at most two between the vertices of $X$, and one of the $(\le\!4)$-cycles formed by these paths is contractible
in $\Sigma$, contradicting the assumptions.

We write $C_1<C_2$ if the closed disk bounded by $C_1$ in $\Sigma+\widehat{R_2}$ contains $C_2$.  Note that $<$ is a linear ordering of the
cycles of length at most four in $G$.  Let $K_1, K_2, \ldots, K_m$ be the list of all cycles of length at most four in $G$ sorted according
to this ordering (we have $K_1=R_1$ and $K_m=R_2$).  
For $i<j$, let $G_{ij}$ be the subgraph of $G$ drawn between $K_i$ and $K_j$.  Note that if $K_i$ and $K_j$ are vertex-disjoint, then
$G_{ij}$ is a $\{K_i,K_j\}$-critical graph embedded in the cylinder with rings $\{K_i,K_j\}$. In that case Lemma~\ref{lemma-cylbase} implies that $G_{ij}$ is subsumed by a $\{K_i,K_j\}$-critical basic graph $H_{ij}$.
If $K_i$ and $K_j$ are not vertex-disjoint, then we define $H_{ij}=G_{ij}$.

Consider cycles $K_i$ and $K_j$ for some $i<j$.  If $|K_i|=|K_j|=3$ and $V(K_i)\cap V(K_j)\neq\emptyset$,
then Theorem~\ref{thm-planechar} implies $j=i+1$.  If $|K_i|=|K_j|=3$ and $K_i$ and $K_j$ are vertex-disjoint, then Lemma~\ref{lemma-critshort} implies that
$j\le i+3$.  If $K_i$ and $K_j$ are not necessarily triangles and $V(K_i)\cap V(K_j)\neq\emptyset$, then by Theorem~\ref{thm-planechar} the area between $K_i$ and $K_j$ consists either of
one face or of two $5$-faces, and thus either $j=i+1$, or $j=i+2$ and $K_{i+1}$ is a triangle.  In particular, $K_i$ and $K_{i+3}$
are vertex-disjoint for $1\le i\le m-3$.

Consider indices $i<j<k$ and a graph $B\in \{G_{ij},H_{ij}\}$, and suppose that $B\cup H_{jk}$ contains a contractible cycle $C$ of length at most $4$.  By the definition of a basic graph,
$C\not\subseteq B$ and $C\not\subseteq H_{jk}$; hence, $C$ has length $4$ and $C=v_1v_2v_3v_4$, where $v_2,v_4\in V(K_j)$, $v_1\in
V(B)\setminus V(K_j)$ and $v_3\in V(H_{jk})\setminus V(K_j)$.  Furthermore, $v_2$ and $v_4$ must be consecutive vertices of $K_j$,
since otherwise $v_2v_3v_4$ together with one of the paths between $v_2$ and $v_4$ in $K_j$ forms a contractible $4$-cycle in $H_{jk}$.
Consequently, both $B$ and $H_{jk}$ contain a triangle incident with an edge of $K_j$.
Therefore, we have the following.
\claim{cl-noconcyc}{For any $i<j<k$ and a graph $B\in \{G_{ij},H_{ij}\}$, if $B\cup H_{jk}$ contains a contractible cycle of length at most $4$,
then both $B$ and $H_{jk}$ contain a triangle with an edge in $K_j$.}

An \emph{interval} is a pair $(i,j)$ such that $1\le i < j\le m$.  The interval $(i,j)$ is \emph{isolated from triangles} if $|K_t|=4$ for $\max(i-1,1)\le t\le \min(j+1,m)$.
The interval $(i,j)$ is \emph{safe} if it is isolated from triangles, and furthermore $H_{t,t+2}$ is triangle-free and $2$-connected for $i\le t\le j-2$.
Consider two intervals $(i_1,j_1)$ and $(i_2,j_2)$, where $i_2\ge j_1+6$.  Suppose that neither of the intervals is safe.
For both $(i,j,p)\in\{(i_1,j_1,1),(i_2,j_2,2)\}$, perform the following trasformation: If $(i,j)$ is not isolated from triangles, then do not modify $G$ and let $T_p=K_t$
for some $t$ such that $|K_t|=3$ and $\max(i-1,1)\le t\le \min(j+1,m)$.  If $(i,j)$ is isolated from triangles, then let $t$ be an index such that $i\le t\le j-2$ and
$H_{t,t+2}$ either contains a triangle or a cutvertex.  Replace the subgraph $G_{t,t+2}$ in $G$ by $H_{t,t+2}$ (and note that since $(i,j)$ is isolated from triangles,
(\ref{cl-noconcyc}) implies that this does not create a contractible $(\le\!4)$-cycle).  If $H_{t,t+2}$ contains a triangle, then let $T_p$ denote this triangle.
Otherwise, $H_{t,t+2}$ contains a cutvertex $w$.  By the criticality of $H_{t,t+2}$, $w$ separates $K_t$ from $K_{t+2}$, and thus there exists a non-contractible curve
$c$ intersecting $H_{t,t+2}$ exactly in $w$; we add a triangle $T_p$ (with vertex set consisting of $w$ and two new vertices) tracing $c$ to the graph.
Let $G'$ denote the graph created from $G$ by these operations, and note that $G'$ subsumes $G$.  For $i\in \{1,2\}$, let $G'_i$ be the subgraph of $G'$ drawn between $R_i$ and $T_i$,
and let $G''$ be the subgraph of $G'$ drawn between $T_1$ and $T_2$.  By Theorem~\ref{thm-aksen}, every precoloring of $\{R_1, R_2\}$ extends to a $3$-coloring of $G'_1\cup G'_2$.
Furthermore, the distance between $T_1$ and $T_2$ in $G'$ is at least three (because there are at least three cycles $K_{j_1+2}, \ldots, K_{i_2-2}$ separating them), and
by Lemma~\ref{lemma-critshort}, every precoloring of $\{T_1,T_2\}$ extends to a $3$-coloring of $G''$.  Consequently, every precoloring of $\{R_1, R_2\}$ extends to a $3$-coloring of $G'$,
and thus also to a $3$-coloring of $G$.  This is a contradiction, since $G$ is $\{R_1, R_2\}$-critical.
Therefore,
\claim{cl-safe}{if $(i_1,j_1)$ and $(i_2,j_2)$ are intervals with $i_2\ge j_1+6$, then at least one of them is safe.}

Consider now a safe interval $(i,i+6)$.  Note that since the interval is isolated from triangles, $K_i$,
$K_{i+2}$, $K_{i+4}$, and $K_{i+6}$ are pairwise vertex disjoint.  Since the interval is safe, $H_{i,i+2}$, $H_{i+2,i+4}$, and $H_{i,i+6}$
are $2$-connected and triangle-free.
Combining (\ref{cl-basicsim}) and (\ref{cl-twoone}) shows that there exists a precoloring $\psi$ of $K_{i+6}$, a vertex $v\in V(K_{i+2})$
and a color $c$ such that every precoloring $\phi_2$ of $K_{i+2}\cup K_{i+6}$ that matches $\psi$ on $K_{i+6}$ and satisfies
$\phi_2(v)\neq c$ extends to a $3$-coloring of $H_{i+2,i+4}\cup H_{i+4,i+6}$.
Furthermore, an inspection of the basic $2$-connected triangle-free graphs shows that that every $3$-coloring of $K_i$
extends to a $3$-coloring of $H_{i,i+2}$ that assigns $v$ a color different from $c$.  
It follows that every precoloring $\phi$ of $K_i\cup K_{i+6}$ that matches $\psi$ on $K_{i+6}$ extends to a $3$-coloring of
$H_{i,i+2}\cup H_{i+2,i+4}\cup H_{i,i+6}$, and thus also to a $3$-coloring of $G_{i,i+6}$.  In fact, it is sufficient to assume that $\phi$ has the same
type on $K_{i+6}$ as $\psi$ to obtain this conclusion. Together with Theorem~\ref{thm-aksen}, we conclude that
\claim{cl-safec}{if $(i,i+6)$ is a safe interval, then there exists a type $S$ such that every precoloring of $R_1\cup K_{i+6}$ whose
type on $K_{i+6}$ is $S$ extends to a $3$-coloring of $G_{1,i+6}$.  Symmetrically, there exists a type $S'$ such that every precoloring of $R_2\cup K_i$ whose
type on $K_i$ is $S'$ extends to a $3$-coloring of $G_{i,m}$.}

If $(1,7)$ is not safe, then $(13,19)$ and $(28,34)$ are safe by (\ref{cl-safe}).  If $(28,34)$ is not safe, then $(1,7)$ and $(16,22)$ are safe (\ref{cl-safe}).
Otherwise, both $(1,7)$ and $(28,34)$ are safe.  Hence, we can fix safe intervals $(i-6,i)$ and $(j,j+6)$ such that $j\ge i+9$.

Let $G'=G_{ij}$ with rings $K_i=a_1a_2a_3a_4$ and $K_j=b_1b_2b_3b_4$.  By (\ref{cl-safec}), there exist types $S_i$ and $S_j$
such that every precoloring of $R_1\cup K_i$ whose type on $K_i$ is $S_i$ extends to $G_{1i}$, and
every precoloring of $R_2\cup K_j$ whose type on $K_j$ is $S_j$ extends to $G_{jm}$.
Since $j\ge i+9$, the distance between $K_i$ and $K_j$ is at least three.  Let $G''$ be the graph obtained from the embedding of $G'$ in the cylinder in the
following way: Cap the holes of the cylinder by disks.  If $S_i=\{a_t,a_{t+2}\}$ for some $t\in\{1,2\}$, then add the edge $a_ta_{t+2}$ to the face bounded by $K_i$ and
add a crosscap to the middle of this edge.  If $S_i=\emptyset$, then identify $a_1$ with $a_3$ to a vertex $a_{13}$ and $a_2$ with $a_4$ to
a vertex $a_{24}$.  Observe that at most two vertices of $K_i$ are incident with a $(\le\!4)$-cycle distinct from $K_i$ in $G'$,
and if there are two such vertices, then they are adjacent.  By symmetry, we can assume that $K_i$ is the only $(\le\!4)$-cycle
incident with $a_2$ and $a_3$.  We add a crosscap on the edge $a_{13}a_{24}$ and draw the edges from $a_{13}$ to the neighbors of $a_3$ and the edges from
$a_{24}$ to the neighbors of $a_2$ through the crosscap.  Transform $K_j$ in the same way according to $S_j$.  Note that $G''$
is embedded in the Klein bottle and it has no loops.

Consider a cycle $C$ of length at most $4$ in $G''$.  Since the distance between $K_i$ and $K_j$ is at least three,
we may assume that $C$ does not contain any of the vertices $b_1$, \ldots, $b_4$, $b_{13}$ or $b_{24}$.  Let us first
consider the case that $S_i=\{a_t,a_{t+2}\}$ for some $t\in\{1,2\}$.  If $C$ does not contain the edge $a_ta_{t+2}$, then $C$ is non-contractible
in $G$, and thus it separates the crosscaps in $G''$.  If $C$ contains the edge $a_ta_{t+2}$, then $C$ is one-sided.
Suppose now that $S_i=\emptyset$; as in the construction of $G''$, we assume that $K_i$ is the only $(\le\!4)$-cycle incident with $a_2$ and $a_3$
in $G'$.  If $C$ contains the edge $a_{13}a_{24}$, then $C$ corresponds to a $(\le\!4)$-cycle in $G'$ containing one of the edges of $K_i$,
which necessarily must be $a_1a_4$; hence, no other edge of $C$ passes through the crosscap and $C$ is one-sided.  If $C$ contains neither $a_{13}$ nor $a_{24}$,
then $C$ is non-contractible in $G$ and separates the crosscaps in $G''$.  If $C$ contained both $a_{13}$ and $a_{24}$,
but not the edge $a_{13}a_{24}$, then since $a_2$ and $a_3$ are not incident with $(\le\!4)$-cycles in $G'$, we conclude that
$a_1a_4$ is incident with two triangles in $G'$.  However, then either one of the triangles or the $4$-cycle contained in their union
is contractible in $G$, which is a contradiction.  It remains to consider the
case that $C$ contains exactly one of $a_{13}$ and $a_{24}$.  By symmetry, assume that $C$ contains $a_{13}$.
Let $e_1'$ and $e_2'$ be the edges incident with $a_{13}$ in $C$,
and let $e_1$ and $e_2$ be the corresponding edges in $G$.  Since no $(\le\!4)$-cycle different from $K_i$ is incident with $a_3$, we may assume that
$e_1$ is incident with $a_1$.  If $e_2$ is incident with $a_3$, then $C$ is one-sided.  If $e_2$ is incident with $a_1$, then $C$ separates the
crosscaps.  We conclude that every $(\le\!4)$-cycle in $G''$ is non-contractible.

If $G''$ is $3$-colorable, then the corresponding $3$-coloring of $G'$ has type $S_i$ on $K_i$ and type $S_j$ on $K_j$.
It follows that every precoloring of $R_1\cup R_2$ extends to a $3$-coloring of $G$, contradicting the criticality of $G$.

Therefore, $G''$ is not $3$-colorable and it contains a $4$-critical subgraph $F$.
By Theorem~\ref{thm-klein}, $F$ is an $(x_1x_2,y_1y_2)$-chain, for some vertices $x_1,x_2,y_1,y_2\in V(G'')$,
and its embedding derived from the embedding of $G''$ is canonical.  Suppose that $S_i=\emptyset$ and that $K_i$ is the only $(\le\!4)$-cycle
in $G'$ incident with $a_2$ and $a_3$.  By symmetry, we can assume that $x_1=a_{13}$ and $x_1x_2$ corresponds to an edge
$a_3v$ in $G'$.  Since $x_1x_2$ is incident with two triangles $x_1x_2v_1$ and $x_1x_2v_2$ in $F$, but $a_3$ is not incident with
a triangle, we have $a_1v_1,a_1v_2\in E(G')$.  Let us remark that $x_2\neq a_{24}$, as otherwise we would similarly
have $a_4v_1, a_4v_2\in E(G')$ and at least one of the cycles $a_1a_4v_1$, $a_1a_4v_2$ and $a_1v_1a_4v_2$ would be contractible in $G$.
Hence, both $v_1$ and $v_2$ are adjacent to the vertex $v$ in $G'$.  Since the $4$-cycle $a_1v_1vv_2$ is non-contractible in $G'$,
using Theorem~\ref{thm-planechar} we can assume that $a_1a_2a_3vv_1$ and $a_1a_4a_3vv_2$ are faces of $G'$ and $a_2$ and $a_4$
have degree two.  On the other hand, if $S_i=\{a_i,a_{i+2}\}$ for some $i\in\{1,2\}$,
then one of $x_1x_2$, $y_1y_2$ is equal to $a_ia_{i+2}$, and since this edge is incident with two triangles in $F$,
it follows that $K_i$ is a subgraph of $F$.  A symmetrical claim holds at $K_j$.
As all faces of $F$ have length at most six, Theorem~\ref{thm-planechar} implies that every face of $F$ not incident with $x_1x_2$ and $y_1y_2$ is also a face of $G'$.
Let us recall that $F$ is a $(x_1x_2,y_1y_2)$-chain, and consequently observe that
in all the cases, $G'$ is a broken chain.

Choose the labeling of $K_i$ and $K_j$ so that $a_1$ and $b_1$ are vertices of degree four in $G'$.
Observe that a precoloring $\psi$ of $K_i\cup K_j$ extends to a $3$-coloring of $G'$
if and only if $\psi(a_1)\neq \psi(a_3)$ or $\psi(b_1)\neq \psi(b_3)$.
Consider a precoloring $\phi$ of $R_1\cup K_j$ that does not extend to a $3$-coloring of $G_{1j}$.
Let $X$ be the graph obtained from $G_{1i}$
in the following way: first, we add the edge $a_1a_3$ and put a crosscap on it.  If $R_1$ is a triangle,
then we paste a crosscap over the cuff incident with $R_1$.  If $R_1$ is a $4$-cycle, then we either add an edge between
two of its vertices or identify its opposite vertices according to the type of $\phi$ on $R_1$ and put a crosscap
in the appropriate place, using the same rules as in the construction of $G''$.  Note that $X$ is embedded in
the Klein bottle so that all contractible cycles have length at least five.  If $X$ is $3$-colorable, then
its $3$-coloring corresponds to a $3$-coloring of $G_{1i}$ that matches $\phi$ on $R_1$ and assigns $a_1$ and $a_3$
different colors.  Hence, this coloring extends to a $3$-coloring of $G_{1j}$ that matches $\phi$ on $R_1\cup K_j$, which is a contradiction.

Therefore, $X$ is not $3$-colorable, and by Theorem~\ref{thm-klein}, $X$ contains a canonical embedding of an $(e_1,e_2)$-chain $F_1$ as a subgraph,
for some edges $e_1,e_2\in E(X)$.  Since $F_1$ contains four one-sided triangles, it follows that $|R_1|=4$.
As in the analysis of $G'$, we conclude that $G_{1i}$ is a broken chain.  By symmetry, $G_{jm}$
is a broken chain as well.  This implies that $G$ is a broken chain.
\end{proof}

The case of a cylinder with two rings of length at most four is now easy to handle using Theorem~\ref{thm:summary},
thanks to the bound on the size of a subgraph that captures $(\le\!4)$-cycles given by Lemma~\ref{thm-compens}.
We will need the following observation.

\begin{lemma}
\mylabel{lem:openweight}
Let $G$ be an $\RR$-critical graph embedded in a surface $\Sigma$ with rings $\RR$ so that every $(\le\!4)$-cycle is non-contractible,
let $G'$ be another $\RR$-critical graph embedded in $\Sigma$ with rings $\RR$ and let $X\subset F(G)$ and $\{(J_f,S_f):f\in F(G')\}$
be a cover of $G$ by faces of $G'$.  Let $f$ be an open $2$-cell face of $G'$ and let $G_1$, \ldots, $G_k$ be the components of the $G$-expansion of $S_f$,
where for $1\le i\le k$, $G_i$ is embedded in the disk with one ring $R_i$.  In this situation, $\sum_{i=1}^k w(G_i,\{R_i\})\le s(|f|)+\el(f)$.
\end{lemma}
\begin{proof}
By Theorem~\ref{thm:diskgirth5} and Lemma~\ref{lem:diskcritical}, we have
$$\sum_{i=1}^k w(G_i,\{R_i\})\le \sum_{i=1}^k s(|R_i|).$$
Note that we have $s(x)+s(y)\le s(x+y)\le s(x)+y$ for every $x,y\ge 5$; hence,
$$\sum_{i=1}^k s(|R_i|)\le s\left(\sum_{i=1}^k |R_i|\right)=s(|f|+\el(f))\le s(|f|)+\el(f).$$
\end{proof}

Let $\cyl$ be a function satisfying the following for all non-negative integers $x$ and $y$:
\begin{itemize}
\item $\cyl(0,0)=0$
\item $\cyl(x,y)=\cyl(y,x)$
\item if $x>0$, then $\cyl(x,y)\ge \cyl(0,y)+x+13$
\item if $x,y>1$, then $\cyl(x,y)\ge \cyl(1,x)+\cyl(1,y)+19$
\item for any non-negative integer $y'<y$, we have $$\cyl(x,y)\ge \cyl(x,y')+s(y-y'+8)\ge \cyl(x,y')+1$$
\item $\cyl(x,y)\ge s(x+y+14)$
\item if $x\ge 4$, then $\cyl(x,y)\ge 886$
\item $\cyl(7,7)\ge 2\cyl(6,7)$
\item if $x\le 4$ and $5\le y\le 6$, then $$\cyl(x,y)\ge (2/3+52\epsilon)(x+y)+20((5\cyl(4,4)+90)/s(5))/3$$
\item if $x\le 7$, then $\cyl(x,7)\ge 3/2(x+7)+20((5\cyl(6,6)+90)/s(5))/3$
\item if $x,y\ge 5$, then $\cyl(x,y)\ge \cyl(4,x)+\cyl(4,y)+\cyl(4,4)$
\end{itemize}
Note that such a function exists, since the lower bounds on $\cyl(x,y)$ only involve values $\cyl(x',y')$ satisfying
either $\max(x',y')<\max(x,y)$, or $\max(x',y')=\max(x,y)$ and $x'+y'<x+y$.

We are now ready to prove the main result of this section.  

\begin{theorem}\mylabel{thm:4cycles}
Let $G$ be a graph embedded in the cylinder with rings $R_1$ and $R_2$ of length at most four.
Suppose that every $(\le\!4)$-cycle in $G$ is non-contractible.
If $G$ is $\{R_1,R_2\}$-critical and not a broken chain, then $w(G,\{R_1,R_2\})\le \cyl(|R_1|,|R_2|)$.
\end{theorem}
\begin{proof}
We proceed by induction, and assume that the claim holds for all graphs with fewer than $|E(G)|$ edges.
By Lemma~\ref{lemma-critshort}, we can assume that $|R_2|=4$.  By Theorem~\ref{thm-aksen}, $G$ is connected,
and thus every face of $G$ is open $2$-cell.
By Lemma~\ref{lem:i012}, $G$ satisfies (I0), (I1) and (I2).  Furthermore, we already observed that every critical
graph without contractible $(\le\!4)$-cycles satisfies (I9), and (I6) and (I8) hold trivially.

Next, we will show that $w(G,\{R_1,R_2\})\le \cyl(|R_1|,|R_2|)$ when (I3) is not satisfied.
\begin{subproof}
Since (I3) is not satisfied, $G$ contains a cutvertex $v$ that is not the main vertex of a vertex-like ring.
Observe that $v$ separates $R_1$ from $R_2$.  Add to $G$ a non-contractible triangle $T$ with vertex set consisting
of $v$ and two new vertices.  For $i=1,2$, let $G_i$ denote the subgraph of the resulting graph drawn between $R_i$ and $T$.
Suppose that $v\in V(R_1)$ (so $G_1=R_1\cup T$); in this case $|R_1|=4$, since $v$ is not the main vertex of a vertex-like ring,
and $T$ forms a non-weak vertex-like ring of $G_2$.  The graph $G_2$ is $\{T,R_2\}$-critical by Lemma~\ref{lemma-crcon},
hence by the induction hypothesis we have $w(G_2,\{T,R_2\})\le \cyl(1,|R_2|)$.
But then $w(G,\{R_1,R_2\})\le w(G_2,\{T,R_2\})+1\le \cyl(1,|R_2|)+1\le \cyl(|R_1|,|R_2|)$
as required.

Hence, we can assume that $v\not\in V(R_1)$, and by symmetry, $v\not\in V(R_2)$.  
It follows that $G_1$ can be seen as embedded in the cylinder with rings $\{R_1,T\}$, where $T$ is a non-weak vertex-like ring.
Note that $G_1$ is $\{R_1,T\}$-critical by Lemma~\ref{lemma-crcon}.  By Corollary~\ref{cor-critshort}, $R_1$ is not a weak vertex-like ring, and thus $|R_1|\ge 1$.
Furthermore, Lemma~\ref{lemma-critshort} implies that if $|R_1|\le 3$, then $G_1$ consists of $R_1$ and an edge between
$v$ and a vertex $w$ of $R_1$.  If that is the case, then $G_2$ is $\{T,R_2\}$-critical, where $T$ is taken as a weak vertex-like ring:
consider any edge $e$ of $G_2$ not belonging to $R_2$.  Since $G$ is $\{R_1,R_2\}$-critical, there exists a precoloring $\phi$ of $R_1$ and $R_2$
that extends to a coloring of $G-e$, but not to $G$.  Let $\psi$ be the precoloring of $\{T,R_2\}$ such that a neighbor $z$ of $v$ in $T$ is assigned the color
$\phi(w)$ and $\psi$ matches $\phi$ on $R_2$ (by the definition of precoloring of a weak vertex-like ring, $z$ is the only vertex of $T$ that is assigned
color by $\psi$).  Note that $\psi$ extends to a coloring of $G_2-e$, but not to $G_2$.  Since the choice of $e$ was arbitrary,
this shows that $G_2$ is $\{T,R_2\}$-critical with $T$ weak.
By the induction hypothesis, we have $w(G_2, \{T,R_2\})\le \cyl(0,|R_2|)$.  Let $f$ be the face of $G_2$ that shares edges with $T$.  We have
$w(G, \{R_1,R_2\})=w(G_2, \{R_2,T\})-s(|f|)+s(|f|+2)\le \cyl(0,|R_2|)+s(|f|+2)-s(|f|)\le \cyl(0,|R_2|)+2\le \cyl(|R_1|,|R_2|)$.

Hence, we can assume that $|R_1|=4$.  Recall that $|R_2|=4$.  Note that $G_i$ is $\{R_i,T\}$-critical for $i\in\{1,2\}$ by Lemma~\ref{lemma-crcon} ($T$ is taken as a non-weak vertex-like ring).
Let $f_1$ and $f_2$ be the faces of $G_1$ and $G_2$, respectively, incident with the edges of $T$.
By the induction hypothesis, we have $w(G, \{R_1,R_2\})\le 2\cyl(1,4)+s(|f_1|+|f_2|-6)-s(|f_1|)-s(|f_2|)\le 2\cyl(1,4)+2\le \cyl(4,4)$.
\end{subproof}
Therefore, we can assume that (I3) holds.

If (I5) is false, then the two adjacent vertices $r_1$ and $r_2$ of degree two belong to a ring of length four, say to $R_2$.
If the face incident with $r_1r_2$ has length five, then a triangle $T$ separates $R_1$ from $R_2$.
By applying induction to the subgraph of $G$ drawn between $R_1$ and $T$, we conclude that
$w(G, \{R_1,R_2\})\le \cyl(|R_1|,3)+s(5)<\cyl(|R_1|,|R_2|)$.  If the face incident with $r_1r_2$ has length at least $6$,
we apply induction to the graph obtained by contracting the edge $r_1r_2$, and obtain $w(G, \{R_1,R_2\})\le \cyl(|R_1|,3)+1\le \cyl(|R_1|,|R_2|)$.
Hence, assume that (I5) holds.

Suppose now that the distance between $R_1$ and $R_2$ is at most four.  We use Lemma~\ref{lem:surfcritical} with $J$ equal to the union of $R_1\cup R_2$
and the shortest path between $R_1$ and $R_2$ and $S$ the only face of $J$; let $G'$ be the unique element of the $G$-expansion of $S$ and let $R$
be its natural ring, where $|R|\le |R_1|+|R_2|+14$ (with equality when the distance between $R_1$ and $R_2$ is four and $|R_1|=|R_2|=0$,
i.e., both $R_1$ and $R_2$ are weak vertex-like rings).
By Theorem~\ref{thm:diskgirth5}, we have $w(G, \{R_1,R_2\})=w(G, \{R\})\le s(|R_1|+|R_2|+14)\le \cyl(|R_1|,|R_2|)$.
Therefore, we can assume that the distance between $R_1$ and $R_2$ is at least five, and in particular
(I7) holds (in (I7), we actually only require that the distance between $R_1$ and $R_2$ is at least four,
however, the stronger statement is needed in the following paragraph).

Consider a path $P$ of length at most four with both ends being ring vertices.  By the previous paragraph, both
ends of $P$ belong to the same ring $R$.  Since $G$ is embedded in the cylinder,
there exists a subpath $Q$ of $R$ such that $P\cup Q$ is a contractible cycle.  Note that $5\le |P\cup Q|\le |P|+3\le 7$,
and by (I9), $P\cup Q$ bounds a face.  
By (I5), $P$ has length at least three.  Therefore, $G$ is well-behaved and satisfies (I4).

Let $M$ be the subgraph of $G$ consisting of all edges incident with $(\le\!4)$-cycles.  Since $G$ is not a broken chain, Lemma~\ref{thm-compens}
implies that $|E(M)|\le 132$.  Note that $M$ captures $(\le\!4)$-cycles of $G$.
If the assumptions of Theorem~\ref{thm:summary} are not satisfied, then
$w(G,\{R_1,R_2\})\le (2/3+26\epsilon)\ell(\{R_1,R_2\})+20|E(M)|/3< 886\le \cyl(|R_1|,|R_2|)$.
Therefore, assume the contrary.

Then, there exists an $\{R_1,R_2\}$-critical graph $G'$ embedded in the cylinder with rings $R_1$ and $R_2$
such that $|E(G')|<|E(G)|$, satisfying conditions (a)--(e) of Theorem~\ref{thm:summary}.  By (b), all
$(\le\!4)$-cycles in $G'$ are non-contractible.  By Theorem~\ref{thm-aksen}, $G'$ is connected, and thus all its faces are open $2$-cell.
Let $X\subset F(G)$ and $\{(J_f,S_f):f\in F(G')\}$ be the cover of $G$ by faces of $G'$ as in (d).  For $f\in F(G')$, let $G^f_1$, \ldots, $G^f_{k_f}$
be the components of the $G$-expansion of $S_f$.  Since $\Sigma_f$ is a disk and all surfaces of the $G$-expansion of $S_f$ are fragments
of $\Sigma_f$, it follows that for $1\le i\le k_f$, $G^f_i$ is embedded in the disk with one ring $R^f_i$.
By the definition of a cover of $G$ by faces of $G'$, we have
$$w(G,\{R_1,R_2\})=\sum_{f\in F(G)} w(f)=\sum_{f\in X} w(f) + \sum_{f\in F(G')} \sum_{i=1}^{k_f} w(G^f_i,\{R^f_i\}).$$

Suppose first that all faces of $G'$ are semi-closed $2$-cell and all vertex-like rings of $G'$ are also vertex-like in $G$.  By Theorem~\ref{thm:summary}(c),
$G'$ has a face of length at least $6$, hence $G'$ is not a broken chain.  Therefore, by induction we have $w(G',\{R_1,R_2\})\le \cyl(|R_1|,|R_2|)$.
Since each internal face of $G'$ is semi-closed $2$-cell, Theorem~\ref{thm:summary}(e) implies that $$\sum_{i=1}^{k_f} w(G^f_i,\{R^f_i\})\le s(|f|)-c(f)$$ for every $f\in F(G')$,
and consequently (using Theorem~\ref{thm:summary}(d) in the last inequality), we have
\begin{align*}
\sum_{f\in F(G')} \sum_{i=1}^{k_f} w(G^f_i,\{R^f_i\})&\le \sum_{f\in F(G')} s(|f|)-c(f)\\
&=w(G',\{R_1,R_2\})-\sum_{f\in F(G')} c(f)\\
&\le w(G',\{R_1,R_2\})-|X|s(6).
\end{align*}
Putting the inequalities together, we obtain 
\begin{align*}
w(G,\{R_1,R_2\})&\le w(G',\{R_1,R_2\})+\left(\sum_{f\in X} w(f)\right)-|X|s(6)\\
&= w(G',\{R_1,R_2\})\le \cyl(|R_1|,|R_2|),
\end{align*}
since the face in $X$ (if any) has length $6$ by the definition of a cover of $G$ by faces of $G'$.

It remains to consider the cases that either a face of $G'$ is not semi-closed $2$-cell or a vertex-like ring of $G'$ is not vertex-like in $G$.
If a face of $G'$ is not semi-closed $2$-cell, then $G'$ contains a cutvertex $v$ that is not the main vertex of a vertex-like ring.
We add to $G'$ a non-contractible triangle $T$ consisting of $v$ and two new vertices.  For $i=1,2$, let $G_i$ denote the subgraph of $G'$ drawn between $R_i$ and $T$.
Similarly to the analysis of the property (I3) for $G$, we show the following: $|R_i|\ge 1$; if $v\in V(R_i)$, then $|R_i|=4$ and $w(G',\{R_1,R_2\})\le \cyl(1,|R_{3-i}|)+1\le \cyl(|R_1|,|R_2|)-11$;
if $|R_i|\le 3$, then $w(G', \{R_1,R_2\})\le \cyl(0,|R_{3-i}|)+2|\le \cyl(|R_1|,|R_2|)-11$;
and if $|R_1|=|R_2|=4$ and $v\not\in V(R_1\cup R_2)$, then $w(G', \{R_1,R_2\})\le 2\cyl(1,4)+2\le \cyl(|R_1|,|R_2|)-11$.

If a vertex-like ring (say $R_2$) of $G'$ is not vertex-like in $G$, then $R_2$ has length $|R_2|=3$ in $G$ and $1$ in $G'$,
and thus again $w(G', \{R_1,R_2\})\le \cyl(|R_1|,1)\le \cyl(|R_1|,|R_2|)-11$.

In both cases Lemma~\ref{lem:openweight} implies
$$\sum_{f\in F(G')} \sum_{i=1}^{k_f} w(G^f_i,\{R^f_i\})\le w(G', \{R_1,R_2\}) + \sum_{f\in F(G')} \el(f)\le w(G', \{R_1,R_2\}) + 10.$$
Combining the inequalities, we have
\begin{align*}
w(G,\{R_1,R_2\})&\le w(G', \{R_1,R_2\}) + 10 + \sum_{f\in X} w(f)\\
&\le  \cyl(|R_1|,|R_2|)-1+\sum_{f\in X} w(f)\\
&< \cyl(|R_1|,|R_2|).
\end{align*}
\end{proof}

\section{Narrow cylinder}\label{sec-narrowcyl}

In this section, we consider graphs embedded in the cylinder with two rings of length at most $7$.
First, let us state an auxiliary result that will also be useful in the case of general surfaces.
Consider a graph embedded in a surface $\Sigma$.
If $K_1$ and $K_2$ are two cycles surrounding a cuff $C$ and $\Delta_1$ and $\Delta_2$ are the open disks bounded by
$K_1$ and $K_2$, respectively, in $\Sigma+\widehat{C}$, then we say that $K_1$ and $K_2$ are {\em incomparable} if
$\Delta_1\not\subseteq \Delta_2$ and $\Delta_2\not\subseteq\Delta_1$.  A \emph{$\Theta$-graph} is a graph consisting
of three paths intersecting exactly in their endvertices.  If $H$ is a $\Theta$-graph appearing as a subgraph of
a graph embedded in a surface with rings, we say that $H$ is \emph{essential} if none of the cycles of $H$ is contractible.

\begin{lemma}\mylabel{lem:concentric}
Let $G$ be a graph in a surface $\Sigma$ with rings $\RR$, such that $G$ is $\RR$-critical,
every $(\le\!4)$-cycle is non-contractible, and no $\Theta$-subgraph of $G$ with at most $12$ vertices is essential.
Let $K_0$ be a cycle in $G$ of length at most seven surrounding a ring $R$, let $C$ be the cuff incident with $R$
and let $\Delta$ be the closed disk in $\Sigma+\widehat{C}$ bounded by $K_0$.
In this situation, at most $10|K_0|$ edges of $G$ drawn outside of $\Delta$
are incident with $(\le\!7)$-cycles surrounding $R$ that are incomparable with $K_0$.
\end{lemma}
\begin{proof}
Let $X$ be the set of edges drawn outside of $\Delta$ that belong to $(\le\!7)$-cycles surrounding $R$ which are
incomparable with $K_0$.  Let us define a mapping $\xi$ from $X$ to faces of $G$ as follows.

For an edge $x\in X$, choose a $(\le\!7)$-cycle $K$ surrounding $R$ incomparable with $K_0$ and containing $x$.
Note that at least one edge $e_1$ of $E(K)\setminus E(K_0)$ is drawn in $\Delta$.
Let $K=P_1\cup P_2$, where $P_1$ and $P_2$ are paths intersecting only in their endvertices
such that $x\in E(P_2)$ and $P_2$ intersects $\Delta$ exactly in its endvertices.  Let $K_0=P_3\cup P_4$,
where $P_3$ and $P_4$ are paths sharing endvertices with $P_1$ and $P_2$ and the cycle $K'=P_2\cup P_3$ is contractible
(such a cycle exists by the assumptions of the lemma, since $P_2\cup P_3\cup P_4$ is a $\Theta$-subgraph of $G$
with at most $12$ vertices).  Let $\xi(x)$ be the face incident with $x$ that is drawn in the open disk bounded by $K'$.

In the situation of the previous paragraph, let $m_i$ be the length of $P_i$ for $1\le i\le 4$.
Note that the closed walk $P_1\cup P_4$ is contractible, and since $e_1\in E(P_1)\setminus E(P_4)$,
the graph $P_1\cup P_4$ contains a cycle.  Since all $(\le\!4)$-cycles are non-contractible,
we have $m_1+m_4\ge 5$ (note that $P_1\cup P_4$ might contain only non-contractible cycles, but
in that case it must contain at least two of them and the sum of their lengths is at least $6$).
Since $m_1+m_2+m_3+m_4=|K_0|+|K|\le 14$, it follows that $|K'|=m_2+m_3\le 9$.
Since $m_2\le 6$, $K'$ shares at least $|K'|-6$ edges with $K_0$.
By Theorem~\ref{thm-planechar}, one of the following holds:
\begin{itemize}
\item $|K'|=9$, the open disk bounded by $K'$ contains one vertex of degree three, and the incident edges
split the disk into three $5$-faces of $G$.  Or,
\item $|K'|\ge 8$ and a chord of $K'$ splits the open disk bounded by $K'$ into a $5$-face and a $(|K'|-3)$-face of $G$.  Or,
\item the open disk bounded by $K'$ is a face of $G$.
\end{itemize}
Let $f$ be a face of $G$ such that $\xi^{-1}(f)\neq\emptyset$.  Note that $f$ lies outside of $\Delta$, and by the preceding analysis,
one of the following holds.
\begin{itemize}
\item $|f|\le 9$ and $f$ is incident with an edge of $K_0$; or,
\item $|f|\in \{5,6\}$ and the boundary of $f$ shares an edge with a face $f'$ of length at most $11-|f|$
incident with at least $|f|-3$ edges of $K_0$.
\end{itemize}
In the latter case, we say that $f$ is \emph{attached} to $f'$.  For a face $f'$ incident with an edge of $K_0$,
let $$\Xi(f')=\xi^{-1}(f')\cup\bigcup_{\text{$f$ attached to $f'$}} \xi^{-1}(f).$$
Let $m\ge 1$ be the number of edges of $K_0$ incident with $f'$.
If $|f'|\in \{7,8,9\}$ or $m=1$, then $|\Xi(f')|=|\xi^{-1}(f')|\le 8$, and $|\Xi(f')|\le 8m$.
If $|f'|=6$ and $m\ge 2$, then at most four $5$-faces are attached to $f'$,
$|\Xi(f')|\le 20$, and $|\Xi(f')|\le 10m$.  Finally, if $|f'|=5$ and $m\ge 2$, then
at most three $(\le\!6)$-faces are attached to $f'$,
$|\Xi(f')|\le 18$, and $|\Xi(f')|\le 9m$.

Consequently,
$$|X|\le \sum_{f'} \Xi(f')\le 10|K_0|,$$
where the sum goes over all faces $f'$ lying outside of $\Delta$ and incident with an edge of $K_0$.
\end{proof}

We will also need a result on non-contractible cycles near a ring.  For a ring $R$ and integers $l,d\ge 0$, an \emph{$(R,l,d)$-noose}
is a non-contractible $(\le\!l)$-cycle whose distance from $R$ is at most $d$.
\begin{lemma}\label{lem:natr}
Let $G$ be a graph embedded in the cylinder with rings $R_1$ and $R_2$, such that $G$ is $\{R_1,R_2\}$-critical and
every $(\le\!4)$-cycle is non-contractible.  Let $l\le 7$ and $d$ be non-negative integers and let $K'_0$ be an $(R_1,l,d)$-noose of $G$.
Let $Y$ be the set of edges of $G$ either belonging to $(R_1,l,d)$-nooses or drawn between the cycles $R_1$ and $K'_0$.
Then $|Y|<(3|R_1|+3l+5d)/s(5)$.
\end{lemma}
\begin{proof}
Let $C_1$ be the cuff incident with $R_1$.
Let $K_0$ be an $(R_1,l,d)$-noose such that the closed disk $\Delta$ bounded by $K_0$ in $\Sigma+\widehat{C_1}$
contains $K'_0$ and subject to that $\Delta$ is as large as possible.  Observe that every edge of $Y$ not drawn in $\Delta$
belongs to an $(R_1,l,d)$-noose that is incomparable with $K_0$, and by Lemma~\ref{lem:concentric}, there are at most $10l$ such edges.
Let $Q$ be a shortest path between $R_1$ and $K_0$; clearly, $Q$ has length at most $d$.  Let $J=R_1\cup Q\cup K_0$
and let $S$ be the set of faces of $J$ contained in $\Delta$.  The sum of the lengths of these faces is at most $|R_1|+l+2d$.
Let $F_0$ be the set of faces of $G$ contained in $\Delta$.  Using Lemma~\ref{lem:surfcritical} and Theorem~\ref{thm:diskgirth5}, we conclude
$\sum_{f\in F_0} w(f)\le s(|R_1|+l+2d)$.  Hence, the number of edges of $G$ drawn in $\Delta$ is at most
\begin{align*}
\frac{1}{2}\Big(|R_1|+|K_0|+\sum_{f\in F_0} |f|\Bigr)&\le \frac{1}{2}\Big(|R_1|+|K_0|+\sum_{f\in F_0} \frac{5w(f)}{s(5)}\Bigr)\\
&\le \frac{|R_1|+l+5s(|R_1|+l+2d)/s(5)}{2}.
\end{align*}
Hence,
$$|Y|\le10l + \frac{|R_1|+l+5s(|R_1|+l+2d)/s(5)}{2}<\frac{3|R_1|+3l+5d}{s(5)}.$$
\end{proof}

The main result of the first paper of this series~\cite[Theorem~13]{trfree1} together with the result of Aksenov et al.~\cite{bor7c} implies the following.
\begin{theorem}\label{thm:sevbnd}
Let $G$ be a graph embedded in the cylinder with rings $R$ and $R'$, where $|R|\le 7$ and $R'$ is a component of $G$.
Suppose that all $(\le\!4)$-cycles in $G$ are non-contractible and that $G$ has girth at least $|R|-3$.
If $G$ is $\{R,R'\}$-critical and $R$ is an induced cycle, then $|R|=6$ and $G$ contains a triangle $C$ such that all vertices of
$C$ are internal and have mutually distinct and non-adjacent neighbors in $R$.
\end{theorem}

We now can prove the main result of this section.

\begin{lemma}\label{lem:cyl47}
Let $G$ be a graph embedded in the cylinder with rings $R_1$ and $R_2$, where
$|R_1|\le |R_2|$ and $5\le |R_2|\le 7$.  Suppose that every $(\le\!4)$-cycle in $G$ is non-contractible.
Furthermore, assume that the following conditions hold:
\begin{itemize}
\item if $|R_1|=4$, then all other $4$-cycles in $G$ are vertex-disjoint from $R_1$,
\item if $|R_1|\ge 5$, then $G$ contains no $(\le\!4)$-cycle, and
\item if $|R_2|=7$, then $G$ contains no triangle distinct from $R_1$.
\end{itemize}
If $G$ is $\{R_1,R_2\}$-critical, then $w(G,\{R_1,R_2\})\le \cyl(|R_1|,|R_2|)$.
\end{lemma}
\begin{proof}
As the induction hypothesis, we assume that the claim holds for all graphs with fewer than $|E(G)|$ edges.
If $G$ is disconnected, then by Theorems~\ref{thm-planechar}, \ref{thm-aksen} and \ref{thm:sevbnd}, we conclude that
$|R_1|\le 4$ and $R_1$ is a component of $G$, $|R_2|\in\{6,7\}$, and either the component of $G$ containing $R_2$ consists of $R_2$ with a chord, or
$|R_2|=6$ and the component of $G$ containing $R_2$ consists of $R_2$, a triangle $T$, and three edges joining distinct
vertices of $T$ to distinct non-adjacent vertices of $R_2$.  Consequently, $w(G,\{R_1,R_2\})\le 8 + s(5)\le \cyl(|R_1|,|R_2|)$.
Hence, we can assume that $G$ is connected.

Note that $G$ satisfies (I0), (I1), (I2), (I6), (I8) and (I9) by Theorem~\ref{thm-planechar} and Lemmas~\ref{lem:i012} and \ref{lem:crit3conn}.
The cases that $G$ has a face that is not semi-closed $2$-cell or that the distance between $R_1$ and $R_2$ is
at most four are dealt with in the same way as in the proof of Theorem~\ref{thm:4cycles},
hence assume that (I3) and (I7) hold.

If $P$ is a path of length at most four with both ends being ring vertices and otherwise disjoint from the rings, then by the previous paragraph
we can assume both ends belong to the same ring $R_i$ for some $i\in\{1,2\}$.
Since $G$ is embedded in the cylinder, there exists a subpath $Q$ of $R_i$ such that $P\cup Q$ is a contractible cycle.
Let us consider the case that $|Q|>|P|$, and let $Q'$ be the path with edge set $E(R_i)\setminus E(Q)$.
Note that $Q'\cup P$ is a non-contractible cycle shorter than $|R_i|$. We apply induction (or Theorem~\ref{thm:4cycles})
to the subgraph of $G$ between $R_{3-i}$ and $Q'\cup P$.  Furthermore, we use Theorem~\ref{thm:diskgirth5} to bound the weight of
the subgraph embedded in the disk bounded by $Q\cup P$.  We conclude that
$w(G,\{R_1,R_2\})\le\cyl(|R_{3-i}|, |Q'\cup P|)+s(|Q\cup P|)=\cyl(|R_{3-i}|,|Q'\cup P|)+s(|R_i|-|Q'\cup P|+2|P|)\le \cyl(|R_{3-i}|,|R_i|)$,
since $2|P|\le 8$.

Therefore, we can assume that $|Q|\le |P|$ for each such path $P$.  This implies that (I4) holds.
Furthermore, $|P\cup Q|\le 8$, and by Theorem~\ref{thm-planechar}, at most two faces of $G$ are in the open disk bounded
by $P\cup Q$.  Furthermore, if there are two, then $|P|=|Q|=4$ and the unique edge in the disk joins the middle vertices of $P$ and $Q$.

Suppose that (I5) is false, and a non-vertex-like ring $R_i$ for some $i\in\{1,2\}$ contains adjacent vertices $r_1$ and $r_2$ of degree two.
By the previous paragraph, the face incident with $r_1r_2$ has length
at least $6$.  We apply induction or Theorem~\ref{thm:4cycles} to the graph obtained by contracting the edge
$r_1r_2$ (in the latter case, observe that the graph is not a broken chain, since if $|R_1|=4$, then all other non-contractible
$4$-cycles are vertex-disjoint from $R_1$).  We conclude that $w(G, \{R_1,R_2\})\le \cyl(|R_{3-i}|,|R_i|-1)+1\le \cyl(|R_1|,|R_2|)$.  Hence, assume that (I5) holds.
Together with the observations from the previous paragraph, this implies that $G$ is well-behaved.

If $|R_1|=|R_2|=7$ and $G$ contains a non-contractible $(\le\!6)$-cycle, then by induction we have
$w(G, \{R_1,R_2\})\le 2\cyl(6,7)\le \cyl(7,7)$, hence we can assume that if $|R_1|=|R_2|=7$, then
all non-contractible cycles have length at least seven.

Suppose that $|R_i|\in \{6,7\}$ for some $i\in\{1,2\}$ and $G$ contains an $(R_i,4,1)$-noose $C$.  By the assumptions, we have $|R_1|\le 4$,
and thus $i=2$.  The subgraph of $G$ drawn between $R_1$ and $C$ is not a broken chain, since if $|R_1|=4$,
then $R_1$ is vertex-disjoint from all other $4$-cycles.  Let $Q$ be a shortest path between $C$ and $R_2$; we have $|Q|\le 1$,
since $C$ is an $(R_2,4,1)$-noose.  We apply Theorem~\ref{thm:4cycles} to the subgraph of $G$ between $R_1$ and $C$, and
Lemma~\ref{lem:surfcritical} (with $S=R_1\cup Q\cup C$) and Theorem~\ref{thm:diskgirth5} to the subgraph of $G$ between $R_2$ and $C$,
concluding that $w(G, \{R_1,R_2\})\le \cyl(|R_1|,4)+s(13)\le \cyl(|R_1|,|R_2|)$.  Hence, we can assume that $G$ does not contain
$(R_i,4,1)$-nooses with $|R_i|\in \{6,7\}$.

Let $k=6$ if $|R_2|=7$ and $k=4$ otherwise.  Let $M$ be the minimal subgraph of $G$ such that
\begin{itemize}
\item $E(M)$ contains all edges incident with non-contractible $(\le k)$-cycles,
\item if $|R_1|=4$, then $E(M)$ contains all edges of all $(R_1,7,3)$-nooses,
\item if $|R_1|=4$ and there exists an $(R_1,4,3)$-noose
vertex-disjoint from $R_1$, then for some such noose $K$, the set $E(M)$ includes all edges drawn between $R_1$ and $K$,
\item if $|R_i|=6$ for some $i\in\{1,2\}$, then $E(M)$ includes all edges of $(R_i, 6, 1)$-nooses, and
\item if $|R_i|=7$ for some $i\in\{1,2\}$, then $E(M)$ includes all edges of $(R_i, 7, 0)$-nooses.
\end{itemize}
Let us bound the number of edges of $M$.  Suppose that there exists a non-contractible $(\le\!k)$-cycle $C$, and choose $C$ so that
the closed subset $\Sigma_1$ of $\Sigma$ between $R_1$ and $C$ is as large as possible.
Let $G_1$ be the subgraph of $G$ drawn in $\Sigma_1$.
If $|C|\le 4$, then by assumptions $|R_1|\le 4$ and either $|R_1|<4$ or $R_1$ is vertex-disjoint from all other $4$-cycles;
and in particular, $G_1$ is not a broken chain.  By Theorem~\ref{thm:4cycles} (or by Lemma~\ref{lem:surfcritical} and Theorem~\ref{thm:diskgirth5}
when $R_1$ and $C$ intersect), the sum of the weights of the faces of $G_1$ is at most $\max(\cyl(4,4),s(8))\le \cyl(k,k)$.
If $|C|>4$, then $k=6$ and $|R_2|=7$, and thus $|R_1|\le 6$ (since all non-contractible cycles have length at least seven when $|R_1|=|R_2|=7$).
By the induction hypothesis (or by Lemma~\ref{lem:surfcritical} and Theorem~\ref{thm:diskgirth5}
when $R_1$ and $C$ intersect), the sum of the weight of the faces of $G_1$ is again at most $\max(\cyl(6,6),s(12))=\cyl(k,k)$.
In either case, at most $5\cyl(k,k)/s(5)$ edges of $G$ are drawn in $\Sigma_1$,
and by Lemma~\ref{lem:concentric}, at most $10k+5\cyl(k,k)/s(5)<(5\cyl(k,k)+1)/s(5)$ edges of $G$ are incident with
non-contractible $(\le k)$-cycles.  By Lemma~\ref{lem:natr}, we conclude that
$|E(M)|< (5\cyl(k,k)+90)/s(5)$.

Note that $M$ captures $(\le\!4)$-cycles of $G$, and by the preceding estimate and the definition of $\cyl$,
we have $(2/3+26\epsilon)(|R_1|+|R_2|)+20|E(M)|/3<\cyl(|R_1|,|R_2|)$;
therefore, we can assume that we can apply Theorem~\ref{thm:summary}.
Let $G'$ be the $\{R_1,R_2\}$-critical graph embedded in the cylinder with rings $R_1$ and $R_2$
such that $|E(G')|<|E(G)|$, satisfying the conditions of Theorem~\ref{thm:summary}.

In particular, (b) implies that all $(\le\!4)$-cycles in $G'$ are non-contractible.
Furthermore, using the choice of $M$ we conclude that the following conditions hold.
\begin{itemize}
\item \emph{If $|R_2|=7$, then $G'$ contains no triangle distinct from $R_1$.}  Indeed,
consider a triangle $C'$ in $G'$, and let $C$ be the corresponding non-contractible cycle in $G$ from (b) of Theorem~\ref{thm:summary},
of length at most $|C'|+3=6$.  By the choice of $M$, we have $C\subseteq M$, and thus $|C|=|C'|$ and
$C\cap (R_1\cup R_2)\subseteq C'$.  By the assumptions, $R_1$ is the only triangle of $G$, and thus
$C=R_1$ and $R_1\subseteq C'$, implying $C'=R_1$.
\item \emph{If $|R_1|=4$, then all other $4$-cycles in $G'$ are vertex-disjoint from $R_1$.}  Consider a $4$-cycle $C'$
in $G'$ which intersects $R_1$ in a vertex $r$, and let $C$ be the corresponding non-contractible cycle in $G$ from (b) of Theorem~\ref{thm:summary},
of length at most $|C'|+3=7$.  By part 1. of (b), there exists a path $Q$ of length at most three in $G-E(M)$ from $r$ to $C$,
and thus $C$ is an $(R_1,7,3)$-noose.  By the choice of $M$, we have $C\subseteq M$, and thus $|C|=|C'|$ and $C\cap (R_1\cup R_2)\subseteq C'$.
If $C$ were vertex-disjoint from $R_1$, then $G$ would contain an $(R_1,4,3)$-noose vertex-disjoint from $R_1$, and thus $M$ would include all
edges of $G$ incident with $R_1$, contradicting the existence of the path $Q$ in $G-E(M)$.  Hence, $C$ intersects $R_1$, and thus $C=R_1$ by the assumptions.
It follows that $R_1\subseteq C'$, and thus $C'=R_1$. 
\item \emph{If $6\le |R_i|\le 7$ for some $i\in\{1,2\}$, then $R_i$ is an induced cycle in $G'$.} Otherwise, a chord of $R_i$
together with a subpath of $R_i$ would form a non-contractible $(|R_i|-3)$-cycle $C'$.  Let $C$ be the corresponding non-contractible cycle in $G$ from (b) of Theorem~\ref{thm:summary},
of length at most $|C'|+3\le 7$.  Since $C'$ shares at least three vertices with $R_i$, part 1. of (b) implies that $C$ intersects $R_i$,
and thus $C$ is an $(R_i,7,0)$-noose.  By the choice of $M$, we have $C\subseteq M$, and thus $|C|=|C'|$ and $C$ is an $(R_i,4,0)$-noose.
However, we argued before that we can assume that $G$ does not contain $(R_i,4,1)$-nooses with $|R_i|\in\{6,7\}$.
\item \emph{If $|R_i|=6$ for some $i\in\{1,2\}$, then $G'$ contains no triangle $T$ such that 
all vertices of $T$ are internal and have non-adjacent neighbors in $R_i$.}  Otherwise, let $C$ be the corresponding non-contractible cycle in $G$ from (b) of Theorem~\ref{thm:summary},
of length at most $|T|+3=6$.  By part 4. of (b), $C$ is an $(R_i,6,1)$-noose. By the choice of $M$, we have $C\subseteq M$, and thus $|C|=|C'|$ and $C$ is an $(R_i,3,1)$-noose.
However, we argued before that we can assume that $G$ does not contain $(R_i,4,1)$-nooses with $|R_i|\in\{6,7\}$.
\end{itemize}
These constraints enable us to apply Theorem~\ref{thm:sevbnd} to show that $G'$ is connected.
It follows that all its faces are open $2$-cell.
Furthermore, the assumptions on non-contractible cycles from the statement of Lemma~\ref{lem:cyl47}
are satisfied for $G'$, except that $G'$ can contain non-contractible $(\le\!4)$-cycles even if $|R_1|\ge 5$.

Let $X\subset F(G)$ and $\{(J_f,S_f):f\in F(G')\}$ be the cover of $G$ by faces of $G'$ as in Theorem~\ref{thm:summary}(d).
For $f\in F(G')$, let $G^f_1$, \ldots, $G^f_{k_f}$
be the components of the $G$-expansion of $S_f$, where for $1\le i\le k_f$, $G^f_i$ is embedded in the disk with one ring $R^f_i$.
We have $$w(G,\{R_1,R_2\})=\sum_{f\in F(G)} w(f)=\sum_{f\in X} w(f) + \sum_{f\in F(G')} \sum_{i=1}^{k_f} w(G^f_i,\{R^f_i\}).$$

The cases that not all faces of $G'$ are semi-closed $2$-cell, or that $R_1$ is a vertex-like ring in $G'$ but not in $G$,
are dealt with in the same way as in the proof of
Theorem~\ref{thm:4cycles}. Hence, assume that all faces of $G'$ are semi-closed $2$-cell and that $R_1$ is vertex-like in $G'$ only if it is vertex-like in $G$.
If $G'$ does not satisfy the assumptions of Lemma~\ref{lem:cyl47}, then
$|R_1|\ge 5$ and $G'$ contains a $(\le\!4)$-cycle.
Let $C_1$ and $C_2$ be the $(\le\!4)$-cycles in $G'$ such that the closed subset $\Sigma'\subseteq\Sigma$ between $C_1$ and $C_2$ is
as large as possible, and observe that all $(\le\!4)$-cycles in $G'$ belong to the subgraph $G_c$ of $G'$ embedded in $\Sigma'$.
By Theorem~\ref{thm:summary}(a), if $G_c$ is a broken chain, then
it has at most four faces.  Therefore, Theorem~\ref{thm:4cycles} implies that
the total weight of the faces of $G_c$ is at most $\cyl(4,4)$.  Applying induction to the
subgraphs of $G'$ between $R_1$ and $C_1$ and between $R_2$ and $C_2$, we have
$w(G',\{R_1,R_2\})\le \cyl(4,|R_1|)+\cyl(4,|R_2|)+\cyl(4,4)\le \cyl(|R_1|,|R_2|)$.

If $G'$ satisfies the assumptions of Lemma~\ref{lem:cyl47}, then
the same inequality $w(G',\{R_1,R_2\})\le \cyl(|R_1|,|R_2|)$ follows by induction.
Since each face of $G'$ is semi-closed $2$-cell, we conclude that
$w(G,\{R_1,R_2\})\le w(G',\{R_1,R_2\})\le \cyl(|R_1|,|R_2|)$
as in the proof of Theorem~\ref{thm:4cycles}.
\end{proof}

\section{Graphs on surfaces}\label{sec-surfaces}

Let $\gen(g,t,t_0,t_1)$ be a function defined for non-negative integers $g$, $t$, $t_0$ and $t_1$
such that $t\ge t_0+t_1$ as
$$\gen(g,t,t_0,t_1)=120g+48t-4t_1-5t_0-120.$$
Let $\surf(g,t,t_0,t_1)$ be a function defined for non-negative integers $g$, $t$, $t_0$ and $t_1$
such that $t\ge t_0+t_1$ as
\begin{itemize}
\item $\surf(g,t,t_0,t_1)=\gen(g,t,t_0,t_1)+116-42t=8-4t_1-5t_0$ if $g=0$ and $t=t_0+t_1=2$,
\item $\surf(g,t,t_0,t_1)=\gen(g,t,t_0,t_1)+114-42t=6t-4t_1-5t_0-6$ if $g=0$, $t\le 2$ and $t_0+t_1<2$, and
\item $\surf(g,t,t_0,t_1)=\gen(g,t,t_0,t_1)$ otherwise.
\end{itemize}

We will need the following properties of the function $\surf$:
\begin{lemma}\label{lemma-surfineq}
If $g$, $g'$, $t$, $t_0$, $t_1$, $t'_0$, $t'_1$ are non-negative integers, then
the following holds:
\begin{itemize}
\item[{\rm (a)}] Assume that if $g=0$ and $t\le 2$, then $t_0+t_1<t$.
If $t\ge 2$, $t'_0\le t_0$, $t'_1\le t_1$ and $t'_0+t'_1\ge t_0+t_1-2$, then
$\surf(g,t-1,t'_0, t'_1)\le\surf(g,t,t_0,t_1)-1$.
\item[{\rm (b)}] If $g' < g$ and either $g'>0$ or $t\ge 2$, then
$\surf(g',t,t_0, t_1)\le \surf(g,t,t_0, t_1)-120(g-g')+32$.
\item[{\rm (c)}] Let $g''$, $t'$, $t''$, $t''_0$ and $t''_1$
be nonnegative integers satisfying $g=g'+g''$, $t=t'+t''$, $t_0=t'_0+t''_0$, $t_1=t'_1+t''_1$,
either $g''>0$ or $t''\ge 1$,
and either $g'>0$ or $t'\ge 2$.  Then,
$\surf(g',t',t'_0, t'_1)+\surf(g'',t'',t''_0, t''_1)\le \surf(g,t,t_0, t_1)-\delta$,
where $\delta=16$ if $g''=0$ and $t''=1$ and $\delta=56$ otherwise.
\item[{\rm (d)}] If $g\ge 2$, then $\surf(g-2, t, t_0, t_1)\le \surf(g, t, t_0, t_1)-124$
\end{itemize}
\end{lemma}
\begin{proof}
Let us consider the claims separately.
\begin{itemize}
\item[{\rm (a)}] If $g=0$ and $t=2$, then $\surf(g,t,t_0,t_1)\ge 1$, while $\surf(g,t-1,t'_0, t'_1)\le 0$.
If $g=0$ and $t=3$, then $\surf(g,t,t_0,t_1)\ge 9$ and $\surf(g,t-1,t'_0, t'_1)\le 6$.  Finally,
if $g>0$ or $t>3$, then $\surf(g,t,t_0,t_1)=\gen(g,t,t_0,t_1)$ and $\surf(g,t-1,t'_0, t'_1)=\gen(g,t-1,t'_0, t'_1)$,
and $\gen(g,t,t_0,t_1)-\gen(g,t-1,t'_0, t'_1)=48-5(t_0-t'_0)-4(t_0-t'_0)\ge 48-5(t_0+t_1-t'_0-t'_1)\ge 38$.
\item[{\rm (b)}] If $g'>0$ or $t>2$, then $\surf(g',t,t_0, t_1)=\gen(g',t,t_0, t_1)$ and
we have $\surf(g',t,t_0, t_1)=\surf(g,t,t_0, t_1)-120(g-g')$.  If $g'=0$ and $t=2$, then
$\surf(g',t,t_0, t_1)-\surf(g,t,t_0, t_1)+120(g-g')\le 116-42t=32$.
\item[{\rm (c)}] Suppose first that $g''=0$ and $t''=1$, i.e., we have $g=g'$ and $t=t'+1$.
If $g>0$, then $\surf(g,t,t_0, t_1)-\surf(g'',t'',t''_0, t''_1)-\surf(g',t',t'_0, t'_1)=\gen(g,t,t_0, t_1)+(4t''_1+5t''_0)-\gen(g',t',t'_0, t'_1)=48$.
If $g=0$, then $t'\ge 2$ and we have
$\surf(g',t',t'_0, t'_1)\le \gen(g',t',t'_0, t'_1)+116-2\cdot 42=\gen(g',t',t'_0, t'_1)+32$.  Hence,
$\surf(g,t,t_0, t_1)-\surf(g'',t'',t''_0, t''_1)-\surf(g',t',t'_0, t'_1)\ge \gen(g,t,t_0, t_1)+(4t''_1+5t''_0)-(\gen(g',t',t'_0, t'_1)+32)=16$.
In both cases, the claim follows.

Therefore, we can assume that if $g''=0$, then $t''\ge 2$.  Therefore, we have
$\surf(g'',t'',t''_0, t''_1)\le\gen(g'',t'',t''_0, t''_1)+32$ and
$\surf(g',t',t'_0, t'_1)\le\gen(g',t',t'_0, t'_1)+32$.
It follows that $\surf(g,t,t_0, t_1)-\surf(g',t',t'_0, t'_1)-\surf(g'',t'',t''_0, t''_1)\ge
\gen(g,t,t_0, t_1)-\gen(g',t',t'_0, t'_1)-\gen(g'',t'',t''_0, t''_1)-64=120-64=56$.
\item[{\rm (d)}] We have
$\surf(g, t, t_0, t_1)-\surf(g-2, t, t_0, t_1)\ge \gen(g, t, t_0, t_1)-(\gen(g-2, t, t_0, t_1)+116)=124$.
\end{itemize}
\end{proof}

Consider a graph $H$ embedded in a surface $\Pi$ with rings $\QQ$, and let $f$ be a face of $H$.
Let us recall that $\Pi_f$ is the surface whose interior is homeomorphic to $f$, as defined in Section~\ref{sec-summary}.
Let $a_0$ and $a_1$ be the number of weak and non-weak vertex-like rings, respectively, that form one of the facial walks of $f$ by
themselves.  Let $a$ be the number of facial walks of $f$.  We define $\surf(f)=\surf(g(\Pi_f),a,a_0,a_1)$.

Let $G_1$ be a graph embedded in $\Sigma_1$ with rings $\RR_1$ and $G_2$ a graph embedded in $\Sigma_2$ with rings $\RR_2$.
Let $m(G_i)$ denote the number of edges of $G_i$ that are not contained in the boundary of $\Sigma_i$.
Let us write $(G_1,\Sigma_1,\RR_1)\prec (G_2,\Sigma_2,\RR_2)$ to denote that 
the quadruple $(g(\Sigma_1),|\RR_1|,m(G_1),|E(G_1)|)$ is lexicographically smaller than $(g(\Sigma_2),|\RR_2|,m(G_2),|E(G_2)|)$.

If $\RR$ is the set of rings of a graph embedded in a surface, let $t_0(\RR)$ and $t_1(\RR)$ be the number of weak and non-weak vertex-like rings in $\RR$, respectively.
Let $\tilde{\ell}(\RR)$ denote the total number of vertices of the rings in $\RR$; we have $\tilde{\ell}(\RR)=\ell(\RR)+3t_0(\RR)+2t_1(\RR)$.
In order to prove Theorem~\ref{thm:mainsurf}, we show the following more general claim.
\begin{theorem}\mylabel{thm:maingen}
There exists a constant $\eta$ with the following property.  Let $G$ be a graph embedded in
a surface $\Sigma$ with rings $\RR$. If $G$ is $\RR$-critical and has internal girth at least five, then
$w(G,{\RR})\le \tilde{\ell}({\RR}) + \eta\cdot\surf(g(\Sigma),|\RR|, t_0(\RR), t_1(\RR))$.
\end{theorem}
\begin{proof}
Let $\eta=1867+67\cyl(7,7)/s(5)$.
We proceed by induction and assume that the claim holds for all graphs $G'$ embedded in surfaces $\Sigma'$ with rings $\RR'$
such that $(G',\Sigma',\RR')\prec (G,\Sigma,\RR)$.
Let $g=g(\Sigma)$, $t_0=t_0(\RR)$ and $t_1=t_1(\RR)$.
By Theorem~\ref{thm:diskgirth5}, the claim holds if $g=0$ and $|\RR|=1$, hence assume that $g>0$ or $|\RR|>1$.
Similarly, if $g=0$ and $|\RR|=2$, then we can assume that $t_0+t_1\le 1$ by Lemma~\ref{lemma-critshort}.
By Lemmas~\ref{lem:i012}, \ref{lem:crit3conn}, and \ref{lem:diskcritical} and Theorem~\ref{thm-planechar}, $G$ satisfies (I0), (I1), (I2), (I6) and (I9).

Suppose now that there exists a path $P$ of length at most six with ends in distinct rings $R_1,R_2\in {\RR}$.
By choosing the shortest such path, we can assume that $P$ intersects no other rings.
Let $J=P\cup \bigcup_{R\in {\RR}} R$ and let $S=\{f\}$, where $f$ is the face of $J$ incident with edges of $P$.
Let $\{G'\}$ be the $G$-expansion of $S$, let $\Sigma'$ be the surface in that $G'$ is embedded and let ${\RR}'$ be the natural rings of $G'$.
Note that $g(\Sigma')=g$, $|\RR'|=|\RR|-1$, $\tilde{\ell}(\RR')\le \tilde{\ell}(\RR)+12$ and $t_0(\RR')+t_1(\RR')\ge t_0+t_1-2$.
Since $(G',\Sigma',\RR')\prec (G,\Sigma,\RR)$, by induction and by Lemma~\ref{lemma-surfineq}(a) we have
$w(G,{\RR})=w(G',{\RR'})\le \eta\cdot\surf(g,|\RR|-1,t_0(\RR'),t_1(\RR'))+\tilde{\ell}(\RR)+12<\eta\cdot\surf(g,|\RR|,t_0,t_1)+\tilde{\ell}(\RR)$.
Therefore, we can assume that no such path exists.
\claim{cl-i7str}{The distance between every two distinct members of $\RR$ is at least seven.}
In particular, (I7) holds.

Next, we aim to prove property (I3).  For later use, we will consider a more general setting.

\claim{cl-non2nonomni}{Let $H$ be a graph embedded in $\Pi$ with rings $\QQ$ such that at least one face
of $H$ is not open $2$-cell and no face of $H$ is omnipresent.  If $H$ is $\QQ$-critical, has internal girth at least five and $(H,\Pi,\QQ)\preceq (G,\Sigma,\RR)$, then
$$w(H,\QQ)\le \tilde{\ell}(\QQ) + \eta\cdot\Bigl(\surf(g(\Pi), |\QQ|, t_0(\QQ),t_1(\QQ)) - 7 - \sum_{h\in F(H)} \surf(h)\Bigr).$$}
\begin{subproof}
We prove the claim by induction.
Consider for a moment a graph $H'$ of internal girth at least $5$ embedded in a surface $\Pi'$ with rings $\QQ'$
with $(H',\Pi',\QQ')\prec(H,\Pi,\QQ)$, such that either $H'=\QQ'$ or $H'$ is $\QQ'$-critical.  We claim that
\begin{equation}\label{eq-ind}
w(H',\QQ')\le \tilde{\ell}(\QQ') + \eta\cdot\Bigl(\surf(g(\Pi'), |\QQ'|, t_0(\QQ'),t_1(\QQ')) - \sum_{h\in F(H')} \surf(h)\Bigr).
\end{equation}
The claim obviously holds if $H'=\QQ'$, since $w(H',\QQ')\le \tilde{\ell}(\QQ')$ in that case; hence, it suffices to consider the case that
$H'$ is $\QQ'$-critical.
If at least one face of $H'$ is not open $2$-cell and no face of $H'$ is omnipresent, then
this follows by an inductive application of (\ref{cl-non2nonomni}) (we could even strengthen the inequality by $7\eta$).
If all faces of $H'$ are open $2$-cell, then note that $\surf(h)=0$ for every $h\in F(H')$,
and since $(H',\Pi',\QQ')\prec(G,\Sigma,\RR)$, we can apply Theorem~\ref{thm:maingen} inductively to obtain (\ref{eq-ind}).
Finally, suppose that $H'$ has an omnipresent face $f$, let $\QQ'=\{Q_1,\ldots, Q_t\}$ and for $1\le i\le t$,
let $C_i$ be the cuff traced by $Q_i$, let $\Delta_i$ be a closed disk in $\Pi'+\widehat{C_i}$ such that $\widehat{C_i}\subset \Delta_i$ and the boundary of $\Delta_i$
is a subset of $f$, and let $f_i$ denote the boundary walk of $f$ contained in $\Delta_i$.
Since all components of $H'$ are planar and contain only one ring, Lemma~\ref{lem:crit3conn} implies that all faces of $H'$ distinct from $f$ are closed $2$-cell.
Furthermore, each vertex-like ring forms a component of the boundary of $f$ by itself, hence
$\surf(f)=\surf(g(\Pi'),|\QQ'|,t_0(\QQ'),t_1(\QQ'))$.
If $Q_i$ is not a vertex-like ring, then by applying Theorem~\ref{thm:diskgirth5} to the subgraph $H'_i$ of $H'$ embedded in $\widehat{\Delta_i}\setminus \widehat{C_i}$, we conclude that the weight of $H'_i$ is at most
$s(|Q_i|)$ and that $|f_i|\le |Q_i|$.  Note that $s(|Q_i|)-s(|f_i|)\le |Q_i|-|f_i|$.  Therefore, we again obtain (\ref{eq-ind}):
\begin{align*}
w(H',\QQ')&\le |f|+\sum_{i=1}^t s(|Q_i|)-s(|f_i|)\le |f|+\sum_{i=1}^t |Q_i|-|f_i|\\
&=\tilde{\ell}(\QQ')\\
&=\tilde{\ell}(\QQ') + \eta\cdot\Bigl(\surf(g(\Pi'), |\QQ'|, t_0(\QQ'),t_1(\QQ')) - \sum_{h\in F(H')} \surf(h)\Bigr).
\end{align*}

Let us now return to the graph $H$.  Since $H$ is $\QQ$-critical, Theorem~\ref{grotzsch} implies that no component of $H$ is a planar graph without rings.
Let $f$ be a face of $H$ which is not open $2$-cell.  Since $H$ has such a face and $f$ is not omnipresent,
we have $g(\Pi)>0$ or $|\QQ|>2$.  Let $c$ be a simple closed curve in $f$
infinitesimally close to a facial walk $W$ of $f$.
Cut $\Pi$ along $c$ and cap the resulting holes by disks ($c$ is always a $2$-sided curve).
Let $\Pi_1$ be the surface obtained this way that contains $W$, and if $c$ is separating, then let $\Pi_2$ be the other surface.
Since $f$ is not omnipresent, we can choose $W$ so that either $g(\Pi_1)>0$ or
$\Pi_1$ contains at least two rings of $\QQ$.  Let us discuss several cases:

\begin{itemize}
\item {\em The curve $c$ is separating and $H$ is contained in $\Pi_1$.}  In this case $f$ has only one facial walk, and since $f$ is not open $2$-cell,
$\Pi_2$ is not the sphere.  It follows that $g(\Pi_1)=g(\Pi)-g(\Pi_2)<g(\Pi)$, and thus $(H,\Pi_1,\QQ)\prec (H,\Pi,\QQ)$.
Note that the weights of the faces of the embedding of $H$ in $\Pi$ and in $\Pi_1$ are the same, with the exception of $f$
whose weight in $\Pi$ is $|f|$, while the corresponding face in $\Pi_1$ has weight $s(|f|)\ge |f|-8$.
By (\ref{eq-ind}), we have
$$w(H,\QQ)\le \tilde{\ell}(\QQ) +8 + \eta\cdot\Bigl(\surf(g(\Pi_1), |\QQ|, t_0(\QQ),t_1(\QQ)) + \surf(f) - \sum_{h\in F(H)} \surf(h)\Bigr).$$
Note that $\surf(f)=120g(\Pi_2)-72$.
By Lemma~\ref{lemma-surfineq}(b), we conclude that
$$w(H,\QQ)\le\tilde{\ell}(\QQ) +8 + \eta\cdot\Bigl(\surf(g(\Pi), |\QQ|, t_0(\QQ),t_1(\QQ)) -40 - \sum_{h\in F(H)} \surf(h)\Bigr).$$

\item {\em The curve $c$ is separating and $\Pi_2$ contains a nonempty part $H_2$ of $H$.}  Let $H_1$ be the part of $H$ contained
in $\Pi_1$. Let $\QQ_i$ be the subset of $\QQ$ belonging to $\Pi_i$ and $f_i$ the face of $H_i$ corresponding to $f$, for $i\in \{1,2\}$.
Note that $f_1$ is an open disk, hence $\surf(f_1)=0$.  Using (\ref{eq-ind}), we get
\begin{align*}
w(H,\QQ)\le &w(f)-w(f_1)-w(f_2)+\tilde{\ell}(\QQ_1)+\tilde{\ell}(\QQ_2)\\
&+\eta\cdot\sum_{i=1}^2\surf(g(\Pi_i), |\QQ_i|, t_0(\QQ_i),t_1(\QQ_i))\\
&+\eta\cdot\Bigl(\surf(f)-\surf(f_2) - \sum_{h\in F(H)} \surf(h)\Bigr).
\end{align*}
Note that $w(f)-w(f_1)-w(f_2)\le 16$ and
$\tilde{\ell}(\QQ_1)+\tilde{\ell}(\QQ_2)=\tilde{\ell}(\QQ)$.  Also, $\surf(f)-\surf(f_2)\le 48$,
and when $g(\Pi_f)=0$ and $f$ has only two facial walks, then $\surf(f)-\surf(f_2)\le 6$.

By Lemma~\ref{lemma-surfineq}(c), we have 
$$\sum_{i=1}^2\surf(g(\Pi_i), |\QQ_i|, t_0(\QQ_i),t_1(\QQ_i)) \le\surf(g(\Pi), |\QQ|, t_0(\QQ),t_1(\QQ)) - \delta,$$
where $\delta=16$ if $g(\Pi_2)=0$ and $|\QQ_2|=1$ and $\delta=56$ otherwise.
Note that if $g(\Pi_2)=0$ and $|\QQ_2|=1$, then $g(\Pi_f)=0$ and $f$ has only two facial walks.
We conclude that $\surf(f)-\surf(f_2)-\delta \le -8$.  Therefore,
$$w(H,\QQ)\le \tilde{\ell}(\QQ)+16+\eta\cdot\Bigl(\surf(g(\Pi), |\QQ|, t_0(\QQ),t_1(\QQ))-8-\sum_{h\in F(H)} \surf(h)\Bigr).$$

\item {\em The curve $c$ is not separating.} Let $f_1$ be the face of $H$ (in the embedding in $\Pi_1$) bounded by $W$ and $f_2$ the other face corresponding to
$f$.  Again, note that $\surf(f_1)=0$.  By (\ref{eq-ind}) applied to $H$ embedded in $\Pi_1$, we obtain the following for the weight of $H$ in $\Pi$:
\begin{align*}
w(H,\QQ)\le &w(f)-w(f_1)-w(f_2) + \tilde{\ell}(\QQ)\\
&+ \eta\cdot\surf(g(\Pi_1), |\QQ|, t_0(\QQ),t_1(\QQ))\\
&+ \eta\cdot\Bigl(\surf(f) - \surf(f_2) - \sum_{h\in F(H)} \surf(h)\Bigr).
\end{align*}
Since $c$ is two-sided, $g(\Pi_1)=g(\Pi)-2$, and
$$\surf(g(\Pi_1), |\QQ|, t_0(\QQ),t_1(\QQ))=\surf(g(\Pi), |\QQ|, t_0(\QQ),t_1(\QQ))-124$$
by Lemma~\ref{lemma-surfineq}(d).
Since $\surf(f) - \surf(f_2)\le 48$ and $w(f)-w(f_1)-w(f_2)\le 16$, we have
$$w(H,\QQ)\le \tilde{\ell}(\QQ)+16+\eta\cdot\Bigl(\surf(g(\Pi), |\QQ|, t_0(\QQ),t_1(\QQ))-76-\sum_{h\in F(H)} \surf(h)\Bigr).$$
\end{itemize}
The results of all the subcases imply (\ref{cl-non2nonomni}).
\end{subproof}

\claim{cl-omni}{Let $H$ be a graph embedded in $\Sigma$ with rings $\RR$ and let $f$ be an omnipresent face of $H$.
If $H$ is $\RR$-critical, has internal girth at least five, and at least one component of $H$ is not very exceptional, then
$$w(H,\RR)\le \tilde{\ell}(\RR)-\kappa=\tilde{\ell}(\RR) - \kappa +  \eta\cdot\Bigl(\surf(g, |\RR|, t_0,t_1) - \sum_{h\in F(H)} \surf(h)\Bigr),$$
where $\kappa=5-5s(5)$ if $H$ has exactly one component not equal to a ring and this component is exceptional,
$\kappa=5+5s(5)$ if $H$ has exactly one component not equal to a ring and this component is not exceptional,
and $\kappa=6$ otherwise.}
\begin{subproof}
Since $H$ is $\RR$-critical and $f$ is an omnipresent face, each component of $H$ is planar and contains
exactly one ring.  In particular, all faces of $H$ distinct from $f$ are closed $2$-cell.
For $R\in\RR$, let $H_R$ be the component of $H$ containing $R$.
Exactly one boundary walk $W$ of $f$ belongs to $H_R$.  Cutting along $W$ and capping the hole by a disk,
we obtain an embedding of $H_R$ in a disk with one ring $R$. Let $f_R$ be the face of this embedding bounded by $W$.
Note that either $H_R=R$ or $H_R$ is $\{R\}$-critical. If $R$ is a vertex-like ring, then by Theorem~\ref{thm:diskgirth5} we have $H_R=R$; hence,
every vertex-like ring in $\RR$ forms a facial walk of $f$, and $\surf(f)=\surf(g,|\RR|,t_0,t_1)$.
Consequently, $\surf(g, |\RR|, t_0,t_1)=\sum_{h\in F(H)} \surf(h)$, and it suffices to prove the first inequality of the claim.

Suppose that $H_R\neq R$ for a ring $R\in \RR$. Theorem~\ref{thm:diskgirth5} implies
$w(H_R,\{R\})\le s(|R|-\rho_R)+\alpha_R$, where 
$$(\rho_R,\alpha_R)=\begin{cases}
(3,s(5))&\text{ if $H_R$ is very exceptional}\\
(5,5s(5))&\text{ if $H_R$ satisfies (E4) or (E5)}\\
(5,-5s(5))&\text{ if $H_R$ is not exceptional.}
\end{cases}$$
Since $f_R$ is a face of $H_R$ and $s(y)-s(x)>5s(5)$ for every $y>x\ge 5$, we have $|f_R|\le |R|-\rho_R$.  Furthermore,
$w(H_R,\{R\})-w(f_R)\le s(|R|-\rho_R)+\alpha-s(|f_R|)\le |R|-|f_R|-\rho_R+\alpha_R$.
Since at least one component of $H$ is not very exceptional, summing over all the rings we obtain
\begin{align*}
w(H,\RR)&= w(f)+\sum_{R\in \RR} (w(H_R,\{R\})-w(f_R))\\
&\le |f|+\sum_{R\in\RR} (|R|-|f_R|)-\kappa\\
&= \tilde{\ell}(\RR)-\kappa.
\end{align*}
\end{subproof}

\claim{cl-rep1}{Let $H$ be an $\RR$-critical graph embedded in $\Sigma$ with rings $\RR$
so that all faces of $H$ are open $2$-cell.
If $H$ is $\RR$-critical, has internal girth at least five, $|E(H)|\le |E(G)|$ and a face $f$ of $H$ is not semi-closed $2$-cell,
then $$w(H,\RR)\le \tilde{\ell}(\RR) + \eta\cdot\Bigl(\surf(g,|\RR|, t_0, t_1)-1/2\Bigr).$$}
\begin{subproof}
Since $f$ is not semi-closed $2$-cell, there exists a vertex $v$ appearing at least twice in the facial walk of $f$ that is not the main vertex
of a vertex-like ring forming part of the boundary of $f$.
There exists a simple closed curve $c$ going through the interior of $f$ and joining two of the appearances of $v$.
Cut the surface along $c$ and patch the resulting hole(s) by disk(s).  Let $v_1$ and $v_2$ be the two vertices to that $v$ is split.
For $i=1,2$, if $v_i$ is not incident with a cuff, drill a new hole next to it in the incident patch and add a triangle $T_i$ tracing its boundary,
with vertex set consisting of $v_i$ and two new vertices.

If $c$ is separating, then let $H_1$ and $H_2$ be the resulting graphs embedded in the two
surfaces $\Sigma_1$ and $\Sigma_2$ obtained by this construction; if $c$ is not separating, then let $H_1$
be the resulting graph embedded in a surface $\Sigma_1$.  We choose the labels so that $v_1\in V(H_1)$.
If $c$ is two-sided, then let $f_1$ and $f_2$ be the faces to that $f$ is split by $c$, where $f_1$ is a face of $H_1$.
If $c$ is one-sided, then let $f_1$ be the face in $\Sigma_1$ corresponding to $f$.
Note that $|f_1|+|f_2|\le |f|+6$ in the former case, and thus
$w(f)-w(f_1)-w(f_2)\le 10$.  Similarly, in the latter case we have $w(f)\le w(f_1)$.

If $c$ is separating, then for $i\in\{1,2\}$, let $\RR_i$ consist the rings of $\RR$ contained in $\Sigma_i$, and
if none of these rings contains $v_i$ (so that $T_i$ exists), then also of the vertex-like ring $T_i$.
Here, we designate $T_i$ as weak if $v$ is an internal vertex, $\Sigma_{3-i}$ is a cylinder and the ring of $H_{3-i}$ distinct from $T_{3-i}$ is a vertex-like ring.
If $c$ is not separating, then let $\RR_1$ consist of the rings of $\RR$, together with those of $T_1$ and $T_2$ that exist.
In this case, we treat $T_1$ and $T_2$ as non-weak vertex-like rings.

Suppose first that $c$ is not separating.  Note that $H_1$ has one or two more rings (of length $1$) than $H$ and
$g(\Sigma_1)\in \{g-1,g-2\}$ (depending on whether $c$ is one-sided or not), and that $H_1$ has at least two rings.
If $H_1$ has only one more ring than $H$, then
\begin{align*}
\surf(g(\Sigma_1),|\RR_1|,t_0(\RR_1),t_1(\RR_1))&\le \surf(g-1,|\RR|+1,t_0,t_1+1)\\
&\le\gen(g-1,|\RR|+1,t_0,t_1+1)+32\\
&=\gen(g,|\RR|,t_0,t_1)-44\\
&=\surf(g,|\RR|,t_0,t_1)-44.
\end{align*}
Let us now consider the case that $H_1$ has two more rings than $H$ (i.e., that $v$ is an internal vertex).  If $g(\Sigma_1)=0$ and $|\RR_1|=2$,
then note that both rings of $H_1$ are vertex-like rings.  Lemma~\ref{lemma-critshort} implies that
$H_1$ has only one edge; but the corresponding edge in $H$ would form a loop, which is a contradiction.
Consequently, we have $g(\Sigma_1)\ge 1$ or $|\RR_1|\ge 3$, and
\begin{align*}
\surf(g(\Sigma_1),|\RR_1|,t_0(\RR_1),t_1(\RR_1))&\le \surf(g-1,|\RR|+2,t_0,t_1+2)\\
&=\surf(g,|\RR|,t_0,t_1)-32.
\end{align*}
We apply Theorem~\ref{thm:maingen} inductively to $H_1$, concluding that
$w(H,\RR)\le \tilde{\ell}(\RR)+12+\eta\cdot\Bigl(\surf(g,|\RR|,t_0,t_1) - 32\Bigr)$, and the claim follows.

Next, we consider the case that $c$ is separating.  Let us remark that $H_i$ is $\RR_i$-critical for $i\in\{1,2\}$.
This follows from Lemma~\ref{lemma-crcon}, unless $T_i$ is a weak vertex-like ring.  However, in that case Lemma~\ref{lemma-critshort}
implies that $H_{3-i}$ contains only one edge not belonging to the rings, and the $\RR_i$-criticality of $H_i$ is argued in the same way
as in the proof of Theorem~\ref{thm:4cycles}.
Thus, we can apply Theorem~\ref{thm:maingen} inductively to $H_1$ and $H_2$, and we have
\begin{align*}
w(H,\RR)&=w(H_1,\RR_1)+w(H_2,\RR_2) + w(f)-w(f_1)-w(f_2)\\
&\le \tilde{\ell}(\RR)+12 + \eta\cdot\sum_{i=1}^2\surf(g(\Sigma_i),|\RR_i|,t_0(\RR_i),t_1(\RR_i))
\end{align*}
Therefore, it suffices to prove that
\begin{equation}\label{eq-loccut}
\sum_{i=1}^2\surf(g(\Sigma_i),|\RR_i|,t_0(\RR_i),t_1(\RR_i))\le \surf(g,|\RR|,t_0,t_1)-1.
\end{equation}

If $g(\Sigma_1)=0$ and $|\RR_1|=1$, then (since $v$ is not the main vertex of a vertex-like ring),
we have $t_0(\RR_1)=t_1(\RR_1)=0$ and $\surf(g(\Sigma_1),|\RR_1|,t_0(\RR_1),t_1(\RR_1))=0$; and furthermore,
$g(\Sigma_2)=g$, $|\RR_2|=|\RR|$, $t_0(\RR_2)=t_0$, and $t_1(\RR_2)=t_1+1$.
Consequently,
\begin{align*}
\sum_{i=1}^2\surf(g(\Sigma_i),|\RR_i|,t_0(\RR_i),t_1(\RR_i))&=\surf(g, |\RR|, t_0, t_1+1)\\
&\le  \surf(g,|\RR|,t_0,t_1)-4,
\end{align*}
which implies (\ref{eq-loccut}).  Hence, we can assume that if $g(\Sigma_1)=0$, then $|\RR_1|\ge 2$,
and symmetrically, if $g(\Sigma_2)=0$, then $|\RR_2|\ge 2$.

If $|\RR_1|+|\RR_2|=|\RR|+1$ (and thus $t_1(\RR_1)+t_1(\RR_2)=t_1+1$), we have
\begin{align*}
\sum_{i=1}^2\surf(g(\Sigma_i),|\RR_i|,t_0(\RR_i),t_1(\RR_i))&\le \sum_{i=1}^2(\gen(g(\Sigma_i),|\RR_i|,t_0(\RR_i),t_1(\RR_i))+32)\\
&=\gen(g,|\RR|,t_0,t_1)-12\\
&=\surf(g,|\RR|,t_0,t_1)-12.
\end{align*}
This implies (\ref{eq-loccut}). Therefore, we can assume that $|\RR_1|+|\RR_2|=|\RR|+2$, i.e., $v$ is an internal vertex.
Suppose that for both $i\in\{1,2\}$, we have $g(\Sigma_i)>0$ or $|\RR_i|>2$.  Then,
\begin{align*}
\sum_{i=1}^2\surf(g(\Sigma_i),|\RR_i|,t_0(\RR_i),t_1(\RR_i))&= \sum_{i=1}^2\gen(g(\Sigma_i),|\RR_i|,t_0(\RR_i),t_1(\RR_i))\\
&=\surf(g,|\RR|,t_0,t_1)-32.
\end{align*}
and (\ref{eq-loccut}) follows.

Hence, we can assume that say $g(\Sigma_1)=0$ and $|\RR_1|=2$.  Then, $\RR_1=\{T_1,R_1\}$ for some ring $R_1$,
$g(\Sigma_2)=g$ and $|\RR_2|=|\RR|$.  Since $H_1$ is $\RR_1$-critical, Corollary~\ref{cor-critshort}
implies that $R_1$ is not a weak vertex-like ring.  If $R_1$ is a vertex-like ring, then $T_2$ is a weak
vertex-like ring of $\RR_2$ which replaces the non-weak vertex-like ring $R_1$.
Therefore,
$\surf(g(\RR_2),|\RR_2|,t_0(\RR_2),t_1(\RR_2))=\surf(g,|\RR|,t_0,t_1)-1$.
Furthermore, $\surf(g(\RR_1),|\RR_1|,t_0(\RR_1),t_1(\RR_1))=\surf(0,2,0,2)=0$,
and (\ref{eq-loccut}) follows.

Finally, consider the case that $|R_1|\ge 3$.  By symmetry, we can assume that
if $g(\Sigma_2)=0$ and $|\RR_2|=2$, then also $\RR_2$ contains a non-vertex-like ring.
Since $\RR_2$ is obtained from $\RR$ by replacing $R_1$ by a non-weak vertex-like ring $T_2$,
we have $\surf(g(\RR_2),|\RR_2|,t_0(\RR_2),t_1(\RR_2))=\surf(g,|\RR|,t_0,t_1)-4$.
Furthermore, $\surf(g(\RR_1),|\RR_1|,t_0(\RR_1),t_1(\RR_1))=\surf(0,2,0,1)=2$.
Consequently,
$$\sum_{i=1}^2\surf(g(\Sigma_i),|\RR_i|,t_0(\RR_i),t_1(\RR_i))\le \surf(g,|\RR|,t_0,t_1)-2.$$

Therefore, inequality (\ref{eq-loccut}) holds.
\end{subproof}

By (\ref{cl-non2nonomni}), (\ref{cl-omni}) and (\ref{cl-rep1}), we can assume that $G$ satisfies (I3).
Next, we consider short paths joining ring vertices.

Suppose that $G$ contains a path $P$ of length at most $11$ joining two distinct vertices $u$ and $v$ of a ring $R\in\RR$,
such that $V(P)\cap V(R)=\{u,v\}$ and $R\cup P$ contains no contractible cycle.
Since the distance between any two rings in $G$ is at least seven by (\ref{cl-i7str}), all vertices of $V(P)\setminus \{u,v\}$ are internal.
Let $J$ be the subgraph of $G$ consisting of $P$ and of the union of the rings, and let $S$ be the set of faces of $J$.
Let $\{G_1,\ldots, G_k\}$ be the $G$-expansion of $S$,
and for $1\le i\le k$, let $\Sigma_i$ be the surface in that $G_i$ is embedded and let $\RR_i$ be the natural rings of $G_i$.  Note that $\sum_{i=1}^k t_0(\RR_i)=t_0$
and $\sum_{i=1}^k t_1(\RR_i)=t_1$.  Let $r=\left(\sum_{i=1}^k |\RR_i|\right)-|\RR|$ and observe that
either $r=0$ and $k=1$, or $r=1$ and $1\le k\le 2$ (depending on whether the curve in $\widehat{\Sigma}$ corresponding to a cycle in $R\cup P$
distinct from $R$ is one-sided, two-sided and non-separating or two-sided and separating). Furthermore,
$\sum_{i=1}^k g(\Sigma_i)=g+2k-r-3$.

We claim that $(G_i,\Sigma_i,\RR_i)\prec(G,\Sigma,\RR)$ for $1\le i\le k$.  This is clearly the case, unless $g(\Sigma_i)=g$.
Then, we have $k=2$, $r=1$ and $g(\Sigma_{3-i})=0$.  Since $R\cup P$ contains no contractible cycle, $\Sigma_{3-i}$ is not a disk,
hence $|\RR_{3-i}|\ge 2$ and $|\RR_i|<|\RR|$, again implying $(G_i,\Sigma_i,\RR_i)\prec(G,\Sigma,\RR)$.

By induction, we have
$w(G_i,\RR_i)\le \tilde{\ell}(\RR_i)+\eta\cdot\surf(g(\Sigma_i),|\RR_i|,t_0(\RR_i),t_1(\RR_i))$, for $1\le i\le k$.
Since every face of $G$ is a face of $G_i$ for some $i\in\{1,\ldots, k\}$
and $\sum_{i=1}^k\tilde{\ell}(\RR_i)\le \tilde{\ell}(\RR)+22$, we conclude that
$$w(G,\RR)\le \tilde{\ell}(\RR)+22+\eta\cdot\sum_{i=1}^k\surf(g(\Sigma_i),|\RR_i|,t_0(\RR_i),t_1(\RR_i)).$$
Note that for $1\le i\le k$, we have that $\Sigma_i$ is not a disk and $\RR_i$ contains at least one non-vertex-like ring,
and thus $\surf(g(\Sigma_i),|\RR_i|,t_0(\RR_i),t_1(\RR_i))\le \gen(g(\Sigma_i),|\RR_i|,t_0(\RR_i),t_1(\RR_i))+30$.
Therefore,
\begin{align*}
\lefteqn{\sum_{i=1}^k\surf(g(\Sigma_i),|\RR_i|,t_0(\RR_i),t_1(\RR_i))}\\
&\le\sum_{i=1}^k(\gen(g(\Sigma_i),|\RR_i|,t_0(\RR_i),t_1(\RR_i))+30)\\
&\le\surf(g,|\RR|,t_0,t_1)+120(2k-r-3)+48r-120(k-1)+60\\
&=\surf(g,|\RR|,t_0,t_1)+120k-72r-180\\
&\le\surf(g,|\RR|,t_0,t_1)-12.
\end{align*}
The inequality of Theorem~\ref{thm:maingen} follows; therefore, we can assume that
\claim{cl-nostrp}{if $P$ is a path of length at most $11$ joining two distinct vertices of a ring $R$,
then $R\cup P$ contains a contractible cycle.}
Let us note that since $g>0$ or $|\RR|\ge 2$, this contractible cycle is unique.

Consider now a path $P$ of length at most four, such that its ends $u$ and $v$ are distinct ring vertices and all other vertices of $P$ are internal.
By (I7), both ends of $P$ belong to the same ring $R$; let $P$, $P_1$ and $P_2$ be the paths in $R\cup P$ joining $u$ and $v$.
By (\ref{cl-nostrp}), we can assume that $P\cup P_2$ is a contractible cycle.  Suppose that the disk bounded by $P\cup P_2$ neither
is a face nor consists of two $5$-faces.  By Theorem~\ref{thm-planechar}, we have $|P\cup P_2|\ge 9$.
Let $J$, $S$, $G_i$, $\Sigma_i$ and $\RR_i$ (for $i\in\{1,2\}$) be defined as in the proof of (\ref{cl-nostrp}),
where $\Sigma_2$ is a disk and $\RR_2$ consists of a single ring corresponding to $P\cup P_2$.
Since $g(\Sigma_1)=g$, $|\RR_1|=|\RR|$ and $|E(G_1)|<|E(G)|$, by induction
we have $w(G_1,\RR_1)\le \tilde{\ell}(\RR_1)+\eta\cdot\surf(g,|\RR|,t_0,t_1)$.  Note that $\tilde{\ell}(\RR_1)=\tilde{\ell}(\RR)+|P|-|P_2|$.
Furthermore, Theorem~\ref{thm:diskgirth5} implies $w(G_2,\RR_2)\le s(|P|+|P_2|)=|P|+|P_2|-8$.
Therefore, $w(G,\RR)\le \tilde{\ell}(\RR)+\eta\cdot\surf(g,|\RR|,t_0,t_1)+2|P|-8$.
Since $|P|\le 4$, the claim of Theorem~\ref{thm:maingen} follows.  Therefore, we can assume that the disk bounded by $P\cup P_1$ is
either a face or consists of two $5$-faces.  The same calculation also excludes the possibility that $|P|\le 2$,
since $s(|P|+|P_2|)\le |P|+|P_2|-4$ for any $P$ and $P_2$ such that $|P|+|P_2|\ge 5$.  In particular, we can assume that (I4) holds for $G$.

Suppose that $G$ contains two adjacent vertices $r_1$ and $r_2$ of degree two that do not belong to a vertex-like ring.  Let $R$ be the ring incident
with $r_1$ and $r_2$, and note that $|R|\ge 4$.  By (I4), the face $f$ incident with $r_1r_2$
has length at least six.  Let $G'$ be the graph obtained from $G$ by contracting the edge $r_1r_2$, let $\RR'$ be the set
of rings of $G'$ obtained from $\RR$ by contracting edge $r_1r_2$ in $R$, and let $f'$ be the face of $G'$ corresponding to $f$.
Observe that $G'$ is $\RR'$-critical.  Suppose that $G'$ contains a $(\le\!4)$-cycle $C'$ distinct from the rings.
Then $G$ contains a $(\le\!5)$-cycle $C$ distinct from the rings containing $r_1r_2$.  Since $G$ has internal girth at least $5$,
we have $|C|=5$, and we obtain a contradiction with (I4).  Therefore, $G'$ has internal girth at least $5$.
By induction, we have $w(G',\RR')=\tilde{\ell}(\RR')+\eta\cdot\surf(g, |\RR|, t_0, t_1)$, and since $\tilde{\ell}(\RR)=\tilde{\ell}(\RR')+1$ and
$w(f)\le w(f')+1$, $G$ satisfies the inequality of Theorem~\ref{thm:maingen}.  Therefore, assume that $G$ satisfies (I5).
Together with the previous paragraph, this implies that $G$ is well-behaved.

Suppose that $G$ contains a non-contractible cycle $C$ of length at most $12$ that does not surround any of the rings.
By (I7), $C$ intersects at most one ring, and by (\ref{cl-nostrp}), $C$ shares at most one vertex with this ring
(as otherwise each subpath of $C$ between consecutive intersections with this ring $R$ would be homotopically equivalent
to a path in $R$, and thus $C$ would be either contractible or homotopically equivalent to $R$, the latter implying that $C$
surrounds $R$).
Let $s=1$ if $C$ intersects a ring, and $s=0$ otherwise.
Let $J$ be the subgraph of $G$ consisting of $C$ and of the union of the rings, and let $S$ be the set of faces of $J$.
Let $\{G_1,\ldots, G_k\}$ be the $G$-expansion of $S$,
and for $1\le i\le k$, let $\Sigma_i$ be the surface in that $G_i$ is embedded and let $\RR_i$ be the natural rings of $G_i$.  Let $r=\left(\sum_{i=1}^k |\RR_i|\right)-|\RR|$.
Note that either $r+s=1$ and $k=1$, or $r+s=2$ and $1\le k\le 2$.  Observe that $\sum_{i=1}^k g(\Sigma_i)=g-s-r+2k-2$.
Furthermore, $\sum_{i=1}^k t_0(\RR_i)+\sum_{i=1}^k t_1(\RR_i)\ge t_0+t_1-s$ and $\sum_{i=1}^k\tilde{\ell}(\RR_i)\le \tilde{\ell}(\RR)+24$.

If $g(\Sigma_1)=g$, then $k=2$ and $g(\Sigma_2)=0$; furthermore, $\Sigma_2$ has at least two cuffs, and if $s=0$, then it has
at least three cuffs, since $C$ does not surround a ring.  Thus, if $g(\Sigma_1)=g$, then $r=2-s$ and consequently
$|\RR_1|=|\RR|+r-|\RR_2|=|\RR|+2-s-|\RR_2|<|\RR|$.  The same argument
can be applied to $\Sigma_2$ if $k=2$, hence $(G_i,\Sigma_i,\RR_i)\prec(G,\Sigma,\RR)$ for $1\le i\le k$.  

By induction, we conclude that
$$w(G,\RR)\le \tilde{\ell}(\RR)+24+\eta\cdot\sum_{i=1}^k\surf(g(\Sigma_i),|\RR_i|,t_0(\RR_i),t_1(\RR_i)).$$
For $1\le i\le k$, let $\delta_i=72$ if $g(\Sigma_i)=0$ and $|\RR_i|=1$, let $\delta_i=30$ if $g(\Sigma_i)=0$ and $|\RR_i|=2$, and let $\delta_i=0$ otherwise,
and note that since $\RR_i$ contains a non-vertex-like ring, we have $\surf(g(\Sigma_i),|\RR_i|,t_0(\RR_i),r_1(\RR_i))=\gen(g(\Sigma_i),|\RR_i|,t_0(\RR_i),r_1(\RR_i))+\delta_i$.

If $k=2$, then recall that since $C$ does not surround a ring, we have either $g(\Sigma_i)>0$ or $|\RR_i|\ge 3-s$ for
$i\in \{1,2\}$; hence, $\delta_1+\delta_2\le 30s$.

If $k=1$, then note that $G$ is not embedded in the projective plane with no rings (Thomassen~\cite{thom-torus} proved that every projective planar graph of girth at least five is $3$-colorable);
hence, if $s=0$, then either $g(\Sigma_1)>0$, or $|\RR_1|\ge 2$.  Consequently, we have $\delta_1\le 30+42s$.

Combining the inequalities, we obtain $\sum_{i=1}^k \delta_i\le 60+42s-30k$, and
\begin{align*}
\lefteqn{\sum_{i=1}^k\surf(g(\Sigma_i),|\RR_i|,t_0(\RR_i),t_1(\RR_i))}\\
&=\sum_{i=1}^k(\gen(g(\Sigma_i),|\RR_i|,t_0(\RR_i),t_1(\RR_i))+\delta_i)\\
&\le\surf(g,|\RR|,t_0,t_1)+120(2k-r-s-2)+48r-120(k-1)+5s+\sum_{i=1}^k\delta_i\\
&\le\surf(g,|\RR|,t_0,t_1)+90k-72(r+s)-60\\
&\le\surf(g,|\RR|,t_0,t_1)-24.
\end{align*}
This implies the inequality of Theorem~\ref{thm:maingen}.  Therefore, assume that every non-contractible cycle of length at most $12$
surrounds a ring.  In particular, $G$ satisfies (I8).

Suppose that $G$ contains an essential $\Theta$-subgraph $H$ with at most $12$ vertices.  Let $P_1$, $P_2$, and $P_3$
be the paths forming the $\Theta$-subgraph, and for $1\le i<j\le 3$, let $K_{ij}$ be the cycle $P_i\cup P_j$.
Since $|K_{ij}|\le 12$ and $K_{ij}$ is non-contractible, we conclude that $K_{ij}$ surrounds a ring $R_{ij}$ with cuff $C_{ij}$.
Let $\Delta_{ij}$ be the closed disk bounded by $K_{ij}$ in $\Sigma+\widehat{C_{ij}}$.  Note that $P_{6-i-j}$ intersects
$\Delta_{ij}$ only in its endpoints, as otherwise $H$ would be drawn in $\Delta_{ij}$ and it would contain a contractible cycle.
We conclude that $\Sigma$ is the sphere with three holes, each bounded by one of the cuffs $C_{12}$, $C_{23}$, and $C_{13}$.
Let $J=H\cup \bigcup_{R\in {\RR}} R$, let $S$ be the set of faces of $J$, and let
$\{G_1,\ldots, G_k\}$ be the $G$-expansion of $S$.  For $1\le a\le k$, let $\Sigma_a$ be the surface in that $G_a$ is embedded and let $\RR_a$ be the natural rings of $G_a$.
Note that $\Sigma_a$ is either a disk, or a cylinder corresponding to the part of $\Sigma$ between $R_{ij}$ and $K_{ij}$ for
some $1\le i<j\le 3$ such that $R_{ij}$ and $K_{ij}$ are disjoint; and in particular, $(G_a,\Sigma_a,\RR_a)\prec(G,\Sigma,\RR)$.
Let $s\le 3$ be the number of indices $a$ such that $\Sigma_a$
is a cylinder.  We have $\sum_{a=1}^k (t_0(\RR_a)+t_1(\RR_a))\ge t_0+t_1-(3-s)$ and
$\sum_{a=1}^k\tilde{\ell}(\RR_a)\le \tilde{\ell}(\RR)+26$.
By induction, we have that
\begin{align*}
w(G,\RR)&\le \tilde{\ell}(\RR)+26+\eta\cdot\sum_{a=1}^k\surf(g(\Sigma_a),|\RR_a|,t_0(\RR_a),t_1(\RR_a))\\
&=\tilde{\ell}(\RR)+26+\eta\cdot\Bigl(6s-4\sum_{a=1}^k t_1(\RR_i)-5\sum_{a=1}^k t_0(\RR_i)\Bigr)\\
&\le\tilde{\ell}(\RR)+26+\eta\cdot(6s-4t_1-5t_0+5(3-s))\\
&=\tilde{\ell}(\RR)+26+\eta\cdot(\surf(g,t,t_0,t_1)-9+s)<\tilde{\ell}(\RR)+\eta\cdot\surf(g,t,t_0,t_1).
\end{align*}
This gives the inequality of Theorem~\ref{thm:maingen}; therefore, assume that 
$G$ contains no essential $\Theta$-subgraph with at most $12$ vertices.

For each ring $R\in \RR$, let $M_R$ be the set of all edges incident with cycles of $G$ of length at most $7$ that surround
$R$, and let $C_R$ be such a cycle chosen so that the part $\Sigma_R$ of $\Sigma$ between $R$ and $C_R$ is as large as possible.
By Lemma~\ref{lem:concentric}, at most $70$ edges of $M_R$ are drawn outside of $\Sigma_R$.
Let $K_R$ be a $(\le\!7)$-cycle in $G\cap \Sigma_R$ chosen so that the part $\Sigma'_R$ of $\Sigma$ between $R$ and $K_R$ (including $R$, but excluding $K_R$)
is as small as possible.  Applying Lemma~\ref{lem:concentric} to the subgraph of $G$ drawn in $\Sigma_R$ with rings $R$ and $C_R$,
we see that at most $70$ edges of $M_R\cap \Sigma_R$ are drawn in $\Sigma'_R$.  We claim that at most $5\cyl(7,7)/s(5)$ edges of $G$ are drawn in $\Sigma_R\setminus\Sigma'_R$: When $K_R$ and $C_R$
are vertex-disjoint, this follows from Lemma~\ref{lem:cyl47}.  When $K_R$ intersects $C_R$, this is implied by Lemma~\ref{lem:diskcritical}
and Theorem~\ref{thm:diskgirth5}, since $\cyl(7,7)>s(14)$.  We conclude that $|M_R|\le 140+5\cyl(7,7)/s(5)$.

Let $M$ consist of all rings of length at most four and of all non-contractible cycles in $G$ of length at most $7$.
Observe that $M=\bigcup_{R\in\RR} M_R$, and thus
$|E(M)|\le (140+5\cyl(7,7)/s(5))|\RR|$.  Note that $M$ captures all $(\le\!4)$-cycles in $G$.
If $w(G,\RR)\le 8g+8|\RR|+(2/3+26\epsilon)\tilde{\ell}(\RR)+20|E(M)|/3-16$,
then $w(G,\RR)\le \tilde{\ell}(\RR)+\eta\cdot\surf(g,|\RR|,t_0,t_1)$ by the choice of $\eta$, and Theorem~\ref{thm:maingen} is true.
Therefore, assume that this is not the case, and since $\tilde{\ell}(\RR)\ge\ell(\RR)$, the assumptions of Theorem~\ref{thm:summary} are satisfied.

Let $G'$ be an $\RR$-critical graph embedded in $\Sigma$
such that $|E(G')|<|E(G)|$, satisfying the conditions of Theorem~\ref{thm:summary}.  In particular, (b) together with the
choice of $M$ implies that $G'$ has internal girth at least five.
Let $X\subset F(G)$ and $\{(J_f,S_f):f\in F(G')\}$ be the cover of $G$ by faces of $G'$ as in Theorem~\ref{thm:summary}(d).
For $f\in F(G')$, let $\{G^f_1, \ldots, G^f_{k_f}\}$
be the $G$-expansion of $S_f$ and for $1\le i\le k_f$, let $\Sigma^f_i$ be the surface in that $G^f_i$ is embedded and
let $\RR^f_i$ denote the natural rings of $G^f_i$.  We have 
\begin{equation}\label{eq-cover}
w(G,\RR)=\sum_{f\in F(G)} w(f)=\sum_{f\in X} w(f) + \sum_{f\in F(G')} \sum_{i=1}^{k_f} w(G^f_i,\RR^f_i).
\end{equation}

Consider a face $f\in F(G')$.  We have $g(\Sigma_f)\le g$.  If $g(\Sigma_f)=g$, then every component of $G'$ is planar,
and since $G'$ is $\RR$-critical, each component of $G'$ contains at least one ring of $\RR$; consequently, $f$ has at most
$|\RR|$ facial walks and $\Sigma_f$ has at most $|\RR|$ cuffs.  Since the surfaces embedding the components of the $G$-expansion of $S_f$ are
fragments of $\Sigma_f$, we have $(G^f_i,\Sigma^f_i,\RR^f_i)\prec (G,\Sigma,\RR)$ for $1\le i\le k_f$:
otherwise, we would have $m(G^f_i)=m(G)$, hence by the definition of $G$-expansion, the boundary of $S_f$ would have to be equal
to the union of rings in $\RR$, contrary to the definition of a cover of $G$ by faces of $G'$.

Therefore, we can apply Theorem~\ref{thm:maingen}
inductively to $G^f_i$ and we get $w(G^f_i,\RR^f_i)\le \tilde{\ell}(\RR^f_i)+\eta\cdot\surf(g(\Sigma^f_i),|\RR^f_i|,t_0(\RR^f_i), t_1(\RR^f_i))$.
Observe that since $\{\Sigma^f_1,\ldots,\Sigma^f_{k_f}\}$ are fragments of $\Sigma_f$, we have
$$\sum_{i=1}^{k_f} \surf(g(\Sigma^f_i),|\RR^f_i|,t_0(\RR^f_i),t_1(\RR^f_i))\le \surf(f),$$
and we obtain
\begin{equation}\label{eq-partsfprel}
\sum_{i=1}^{k_f} w(G^f_i,\RR^f_i)\le |f|+\el(f)+\eta\cdot\surf(f).
\end{equation}

In case that $f$ is open $2$-cell, all fragments of $f$ are disks and we can use Theorem~\ref{thm:diskgirth5}
instead of Theorem~\ref{thm:maingen}, getting the stronger inequality $w(G^f_i,\RR^f_i)\le s(\tilde{\ell}(\RR^f_i))$
for $1\le i\le k_f$.  Summing these inequalities, we can strengthen (\ref{eq-partsfprel}) to
\begin{equation}\label{eq-partsf}
\sum_{i=1}^{k_f} w(G^f_i,\RR^f_i)\le w(f)+\el(f)+\eta\cdot\surf(f).
\end{equation}

The inequalities (\ref{eq-cover}), (\ref{eq-partsf}) and Theorem~\ref{thm:summary}(d) imply that
\begin{align}
w(G,\RR)&\le |X|s(6)+\sum_{f\in F(G')} (w(f)+\el(f)+\eta\cdot\surf(f))\nonumber\\
&\le w(G',\RR)+s(6)+10+\eta\cdot\sum_{f\in F(G')} \surf(f).\label{eq-main}
\end{align}
If $G'$ has a face that is not open $2$-cell and no face of $G'$ is omnipresent, then
(\ref{cl-non2nonomni}) implies that
$$w(G',\RR)\le \tilde{\ell}(\RR) + \eta\cdot\Bigl(\surf(g, |\RR|, t_0,t_1) - 7 - \sum_{f\in F(G')} \surf(f)\Bigr),$$
and consequently $G$ satisfies the outcome of Theorem~\ref{thm:maingen}.
Therefore, we can assume that either all faces of $G'$ are open $2$-cell, or $G'$ has an omnipresent face.
Similarly, using (\ref{cl-rep1}) we can assume that if no face of $G'$ is omnipresent, then
all of them are semi-closed $2$-cell. 

Suppose first that $G$ has no omnipresent face.
If $G'$ has a vertex-like ring that is not vertex-like in $G$, then
by the induction hypothesis
$w(G',\RR)\le \tilde{\ell}(\RR) + \eta\cdot\surf(g, |\RR|, t_0,t_1+1)\le \tilde{\ell}(\RR) + \eta\cdot(\surf(g, |\RR|, t_0,t_1)-1)$,
and since $\surf(f)=0$ for every $f\in F(G')$, (\ref{eq-main}) implies that $G$ satisfies the outcome of Theorem~\ref{thm:maingen}.
Hence, assume that all vertex-like rings of $G'$ are vertex-like in $G$.
Using (\ref{eq-cover}) and Theorem~\ref{thm:summary}(d) and (e)
and applying Theorem~\ref{thm:maingen} inductively to $G'$, we have
\begin{align*}
w(G,\RR)&\le|X|s(6)+\sum_{f\in F(G')} (w(f)-c(f))\\
&=w(G',\RR)+|X|s(6)-\sum_{f\in F(G')} c(f)\\
&\le w(G',\RR)\le \tilde{\ell}(\RR)+\eta\cdot\surf(g,|\RR|,t_0,t_1),
\end{align*}
showing that $G$ satisfies the outcome of Theorem~\ref{thm:maingen}.

It remains to consider the case that $G'$ has an omnipresent face $h$.
Then, every component of $G$ is a plane graph with one ring, and by
Lemma~\ref{lem:crit3conn}, we conclude that every face
of $G$ different from $h$ is closed $2$-cell and $G'$ satisfies (I6).
By Theorem~\ref{thm:diskgirth5}, every vertex-like ring of $G'$ is isolated.
By (\ref{eq-cover}) and Theorem~\ref{thm:summary}(d) and (e) and by (\ref{eq-partsf}), we have
\begin{align*}
w(G,\RR)&\le |X|s(6)+\sum_{f\in F(G'),f\neq h} (w(f)-c(f))+\sum_{i=1}^{k_h}w(G^h_i,\RR^h_i)\\
&=w(G',\RR)+|X|s(6)+(c(h)-w(h))-\sum_{f\in F(G')} c(f)+\sum_{i=1}^{k_h}w(G^h_i,\RR^h_i)\\
&\le w(G',\RR)+c(h)-w(h)+\sum_{i=1}^{k_h}w(G^h_i,\RR^h_i)\\
&\le w(G',\RR)+c(h)+\el(h)+\eta\cdot\surf(g,|\RR|,t_0,t_1)
\end{align*}
By Theorem~\ref{thm:summary}(f), at least one component of $G'$ is not very exceptional.
We use (\ref{cl-omni}) to bound the weight of $G'$.  We obtain ($\kappa$ is defined as in (\ref{cl-omni}))
$$w(G,\RR)\le\tilde{\ell}(\RR)+\eta\cdot\surf(g,|\RR|,t_0,t_1)+c(h)+\el(h)-\kappa.$$
By Theorem~\ref{thm:summary}(d), we have $\el(h)\le 5$.
By the definition of $\kappa$ and of the contribution of $h$, it follows that
$c(h)+\el(h)\le \kappa$.  Therefore, $$w(G,\RR)\le \tilde{\ell}(\RR)+\eta\cdot\surf(g,|\RR|,t_0,t_1)$$ as required.
\end{proof}

Let us remark that Theorem~\ref{thm:maingen} implies the special case of Theorem~\ref{thm:corner} for graphs with no $4$-cycles,
by considering the triangles to be rings (we need to first split their vertices so that they become vertex-disjoint, then
drill holes in them).  Furthermore, Theorem~\ref{thm:mainsurf} follows as a special case when the set of rings is empty.

\bibliographystyle{acm}
\bibliography{4critsurf}

\section*{Appendix}

Here, we describe modifications to the proof of \cite[Theorem~9.1]{trfree2} to establish the statement of part (b)
of the corresponding Theorem~\ref{thm:summary} of the current paper.  We use terminology defined in \cite{trfree2}
without repeating the definitions here, as reading this Appendix is only meaningful in the context of that paper.

In the proof of \cite[Theorem~9.1]{trfree2}, we first establish existence of a good configuration $\gamma$ which strongly
appears in $G$ and does not touch $M$.  Then, we let $G_1$ be a $\gamma$-reduction of $G$ with respect to some precoloring $\phi$ of $\RR$
which does not extend to a $3$-coloring of $G$, and we let $G'$ be an $\RR$-critical subgraph of $G_1$.  Consider now a $(\le\!4)$-cycle $C'$ in $G'$.
In \cite[Lemma~6.2]{trfree2}, we establish that either a lift of $C'$ is a cycle $C$ in $G$, or $C'$ is non-contractible and $G$ contains a non-contractible cycle $C$
touching $\gamma$ with $|C|\le |C'|+3$.  In the former case, observe that $C$ (and thus also $C'$) is non-contractible by (I9) and (I4).
Let us now consider each of the subclaims of (b) separately.
\begin{enumerate}
\item  In \cite[Lemma~6.2]{trfree2}, we state that when no lift of $C'$ is a cycle in $G$, then
all ring vertices of $C'$ belong to $C$.  However, this is not quite true---it can happen that $\mathcal{I}_\gamma$ contains a vertex $v$ of
$C$ as well as a ring vertex $r$ not belonging to $C$, in which case we have $r\in V(C')$ after the identification of the vertices of $\mathcal{I}_\gamma$.
However, clearly only one such ring vertex $r\in V(C')\setminus V(C)$ can exist, and $r$ is joined to $v$ in $G$ by the replacement path $Q$ of the configuration
$\gamma$ of length at most three.  Since $\gamma$ does not touch $M$, we have $E(Q)\cap E(M)=\emptyset$ (in the case of the configuration $\R5$,
the replacement path also contains the edge $v_6x_6$ not incident with ${\cal F}_\gamma$; however, if $v_6x_6\in E(M)$, then also $v_6v_5$ or $v_6v_7$ would belong
to $E(M)$, since $M$ has minimum degree at least two).  The same argument applies in the case that a lift of $C'$ is a cycle in $G$.

\item Since $C\not\subseteq M$ and $M$ captures $(\le\!4)$-cycles, we have $|C|\ge 5$, and thus $C$ is not a lift of $C'$.
Since $\gamma$ strongly appears in $G$, either $\mathcal{A}_\gamma=\emptyset$ or $\mathcal{A}_\gamma$ contains an internal vertex, and since
$C'$ only contains ring vertices, $C'$ does not contain a new edge (added between the vertices of $\mathcal{A}_\gamma$ during the reduction of $\gamma$).
Hence, $C'$ is obtained from $C$ by contracting a replacement path between vertices of $\mathcal{I}_\gamma$.  If $|\mathcal{I}_\gamma|\le 2$, this implies
$V(C')\subseteq V(C)$.  If $|\mathcal{I}_\gamma|=3$, then $\gamma$ is $\R3$, and since $\gamma$ strongly appears in $G$,
we conclude that $V(C')\not\subseteq V(C)$ only when $v_4$, $v_5$, and $v_6$ are ring vertices, $v_6$ is the only ring neighbor of $v_1$, and $v_4$ is the only ring neighbor of $v_3$.
By (I4), we conclude that $v_4,v_5,v_6\in V(R)$ for a ring $R\in\RR$, $C$ is the concatenation of the path $R-v_5$ with the path $v_6v_1v_2v_3v_4$,
and $C'=R$.  However, then $R$ is a lift of $C'$, and we can choose $C=R$ instead.

\item If $C$ is not a lift of $C'$, then $C$ touches $\gamma$, and since $\gamma$ does not touch $M$, we have $C\not\subseteq M$.
Hence, we can assume that $C$ is a lift of $C'$.  Then clearly $|C|=|C'|$ and $C\cap \bigcup \RR\subseteq C'$.

\item If $C$ is not a lift of $C'$, then the statement is proved in \cite[Lemma~6.2]{trfree2} (even in a stronger form guaranteeing the existence
of two edges between $C$ and $R$).  If $C$ is a lift of $C'$, then at most two of the three edges joining $C'$ to $R$ in $G'$ can arise from
the addition of a new edge between the vertices of $\mathcal{A}_\gamma$ and the identification of the vertices of $\mathcal{I}_\gamma$,
and thus in $G$, the cycle $C$ has a neighbor in $R$.
\end{enumerate}
\end{document}